\documentclass[12pt]{article}        

 \usepackage[T1]{fontenc}              
 \usepackage[latin1]{inputenc}                        
 \usepackage[width=16cm, height=22cm]{geometry}   
 \usepackage{amsmath,amssymb}         
\usepackage{hyperref} 
 \usepackage[amsmath,thmmarks,hyperref]{ntheorem} 
 
 \usepackage{icomma}                   
 \usepackage{enumerate}              
 \usepackage{color}                   
  \usepackage{setspace}                 
 \usepackage{needspace}                
  \usepackage{url}                      
 \usepackage{mathrsfs}                 
 \usepackage{graphicx}                 
 \usepackage{slashed}
  \usepackage{tikz-cd}                     
  \usetikzlibrary{arrows,decorations.pathmorphing,backgrounds,positioning,fit,bending}
\usepackage{verbatim}
\usepackage{dutchcal}
\usepackage{wasysym}
\usepackage{textcomp}

\usepackage{titlesec}

\usepackage{tcolorbox}
\newcounter{comments}
\newenvironment{displaycomment}{\begin{list}{}{\rightmargin=1cm\leftmargin=1cm}\item\sf\begin{small}}{\end{small}\end{list}}

 \numberwithin{equation}{section}
 
 \usepackage[numbers,square]{natbib}
  \definecolor{MyBlue}{RGB}{25,106,180}

\theoremstyle{nonumberplain}  
\theoremheaderfont{\itshape}  
\theorembodyfont{\normalfont}  
\theoremseparator{.}  
\theoremsymbol{$\Box$}
\newtheorem{proof}{Proof} 
\theoremsymbol{$\blacksquare$}

\theoremstyle{plain}  
\theoremheaderfont{\bfseries\scshape}  
\theorembodyfont{\itshape}  
\theoremsymbol{~}
\theoremseparator{.}  
\newtheorem{proposition}{Proposition}[section]  
 
\newtheorem{corollary}[proposition]{Corollary}  
\newtheorem{lemma}[proposition]{Lemma}  
\newtheorem{theorem}[proposition]{Theorem}

\theorembodyfont{\normalfont}  
 
\newtheorem{remark}[proposition]{Remark}
\newtheorem{example}[proposition]{Example}  
  
\newtheorem{definition}[proposition]{Definition}

\newtheorem{notation}[proposition]{Notation} 
\newtheorem{construction}[proposition]{Construction}  
\newtheorem{signDiscussion}[proposition]{Sign Discussion}

\theoremstyle{nonumberplain}  
\newtheorem{definitionNoNumber}[proposition]{Definition} 
 
\theoremheaderfont{\bfseries\scshape}  
\theorembodyfont{\itshape}  
\newtheorem{theoremA}[proposition]{Theorem A} 
\newtheorem{theoremB}[proposition]{Theorem B} 

\theoremstyle{nonumberplain}
\theorembodyfont{\normalfont}  

\DeclareMathOperator*{\sign}{sign}

\newcommand*{\res}{\operatorname{res}}
\newcommand{\R}{\mathbb{R}}

\newcommand{\A}{\mathcal{A}}
\newcommand{\T}{\mathfrak{T}}
\newcommand{\CC}{\mathcal{C}}
\newcommand{\M}{\mathfrak{M}}
\newcommand{\B}{\mathcal{B}}
\newcommand{\SB}{\mathfrak{S}}

\newcommand{\F}{\mathfrak{F}}

\newcommand{\N}{\mathbb{N}}
\newcommand{\Z}{\mathbb{Z}}
\newcommand{\C}{\mathbb{C}}

\newcommand{\Lag}{\mathrm{Lag}}
\newcommand{\ind}{\operatorname{ind}}
\newcommand{\Cl}{\mathrm{Cl}}

\newcommand{\End}{\mathrm{End}}

\renewcommand{\L}{\mathfrak{L}}

\newcommand{\Spin}{\mathrm{Spin}}

\newcommand{\SO}{\mathrm{SO}}

\newcommand{\Pf}{\mathfrak{Pf}}
\newcommand{\FusTwoGrb}{\mathbb{F}\mathrm{us}}
\newcommand{\Fus}{\mathfrak{Fus}}
\newcommand{\Mul}{\mathfrak{Mul}}
\newcommand{\id}{\mathrm{id}}

\newcommand{\m}{\mathfrak{m}}

\newcommand{\graph}{\mathrm{graph}}
\newcommand{\Hom}{\mathrm{Hom}}
\newcommand{\Aut}{\mathrm{Aut}}

\renewcommand{\rho}{\varrho}

\newcommand{\sAlg}{\text{\textsc{sAlg}}}
\newcommand{\svNAlg}{\text{\textsc{svNAlg}}}

\newcommand{\sLine}{\text{\textsc{sLine}}}
\newcommand{\Bdl}{\text{\textsc{Bdl}}}

\newcommand{\Gerb}{\text{\textsc{Gerb}}}

\newcommand{\sGerb}{\text{\textsc{sGerb}}}

\newcommand{\sTwoVect}{\text{\textsc{2\text{-}sVect}}}
\newcommand{\String}{\mathrm{String}}
\newcommand{\U}{\mathrm{U}}
\renewcommand{\O}{\mathrm{O}}

\newcommand{\ev}{\mathrm{ev}}
\newcommand{\pr}{\mathrm{pr}}
\newcommand{\orclass}{\mathrm{or}}
\newcommand{\DDclass}{\textsc{dd}}
\newcommand{\CCclass}{\textsc{cc}}
\newcommand{\Imp}{\mathrm{Imp}}
\newcommand{\ImpLine}{\mathfrak{Imp}}
\newcommand{\op}{\mathrm{op}}
\newcommand{\lact}{\triangleright}
\newcommand{\LagGrb}{\mathcal{L}\!\mathcal{a}\!\mathcal{g}}
\newcommand{\Lift}{\mathcal{L}\!\text{\itshape{ift}}}
\newcommand{\LiftTwoGrb}{\mathbb{L}\mathrm{ift}}
\newcommand{\LagTwoGrb}{\mathbb{L}\mathrm{ag}}
\newcommand{\ract}{\triangleleft}
\newcommand{\BB}{\mathrm{B}}
\newcommand{\defeq}{~\stackrel{\text{def}}{=}~}
\newcommand{\fuse}{\circledast}

\title{The spinor bundle on loop space}
\author{Matthias Ludewig}
\date{}

\begin{document}

\maketitle

\begin{center}
\itshape{Dedicated to Peter Teichner \\ on the occasion of his 60th birthday.}
\end{center}

\medskip

\begin{abstract}
We give a construction of the spinor bundle of the loop space of a string manifold together with its fusion product, inspired by ideas from Stolz and Teichner. 
The spinor bundle is a super bimodule bundle for a bundle of Clifford von Neumann algebras over the free path space, and the fusion product is defined using Connes fusion of such bimodules.
As the main result, we prove that a spinor bundle with fusion product on a manifold $X$ exists if and only $X$ admits a string structure.
\end{abstract}

\tableofcontents

\section*{Introduction}
\addcontentsline{toc}{section}{\protect\numberline{}Introduction}%

The idea study the spinor bundle of the free loop space of a Riemannian manifold came up in the 1980's, in the context of string theory \cite{KillingbackWorldSheet}. 
Most notably, Witten formally calculated the $S^1$-equivariant index of the Dirac operator on loop space, arriving at a modular form valued characteristic class, now known as \emph{Witten genus} \cite{WittenIndexLoopSpace}.

Recall that a manifold $X$ is \emph{string} if it is spin and the characteristic class $\frac{1}{2}p_1(X)$ vanishes.
The string condition implies that the loop space $LX$ of $X$ admits a spin structure, in the sense that its structure group (which, for a spin manifold is canonically reduced to $L\Spin(d)$) can be lifted to its basic central extension.
However, early on it became apparent that the spin condition for $LX$ does \emph{not} in turn imply the string condition for the manifold $X$ itself. 
More precisely, a spinor bundle on $LX$ can be defined already if the transgression $\tau(\frac{1}{2}p_1(X))$ vanishes, a weaker condition than the vanishing of $\frac{1}{2}p_1(X)$ itself \cite{McLaughlin}.

As the vanishing of $\frac{1}{2}p_1(X)$ (and not only of its transgression) is crucial for Witten's machinery to work (in particular, only under this condition, the Witten genus takes values in topological modular forms \cite{Zagier, AndoHopkinsStrickland}), one is lead to ask for an additional structure on the loop space spinor bundle which only exists under the string condition on $X$.
In the very influential draft \cite{StolzTeichnerSpinorBundle}, this answer was conjectured by Stolz and Teichner almost 20 years ago: The spinor bundle over $LX$ admits a \emph{fusion product} if and only if $X$ admits a string structure.
The purpose of the current paper is to make this statement precise and to give a proof.

Our work is highly influenced by the work of Kristel and Waldorf \cite{KristelWaldorf1, KristelWaldorf2, KristelWaldorf3}, based on the PhD thesis of Kristel \cite{KristelThesis}.
Given a string structure on a manifold $X$, Kristel and Waldorf construct a spinor bundle on the loop space of $X$, together with a fusion product.
Apart from simplifying and somewhat generalizing their construction, the main contribution of this paper is to show the converse implication: If a spinor bundle with fusion product on $LX$ exists, then $X$ admits a string structure.

\medskip

We now give an overview over the contents of the present paper.

\paragraph{The spinor bundle.}

After reviewing some required preliminary facts on Clifford algebras, Lagrangians and Fock spaces, we give a construction of the spinor bundle on loop space.
Since it is related to a central extension by $\U(1)$, the construction of the spinor bundle on loop space is  analogous to the construction of the spinor bundle on an even-dimensional spin$^c$ manifold, where a spinor bundle $\SB$ can be defined as a bundle of irreducible super $\Cl(TX)$-modules.
In the infinite-dimensional case, it is crucial that we complete the relevant bundle of algebraic Clifford algebras to a bundle of von Neumann algebras.
Throughout the paper, we crucially depend on the notion of a (locally trivial) bundle of von Neumann algebras, and, correspondingly, bundles of bimodules for these.
While these notions have not been worked out before, they are more or less straight forward to define; we give an account in Appendix~\ref{SectionVNBundlesBimodules}.

We now briefly summarize our construction. 
For a loop $\gamma \in LX$, consider the Hilbert space
\begin{equation*}
H_\gamma = L^2(S^1, \mathbb{S} \otimes \gamma^*T^\C X),
\end{equation*}
where $\mathbb{S}$ denotes the spinor bundle on $S^1$ for the bounding spin structure.
Up to the twist by $\mathbb{S}$ (which is negligible in even dimensions but necessary in odd dimensions for index-theoretic reasons), this is just the completion of the (complexified) tangent bundle $TLX$ with respect to its canonical $L^2$-metric.
$H_\gamma$ is a \emph{``real''} Hilbert space, i.e., a complex Hilbert space with a real structure, and the twisting by the bundle $\mathbb{S}$ ensures that it carries a canonical \emph{polarization} $\Lag$, i.e., an equivalence class of Lagrangian subspaces $L \subset H_\gamma$.
The algebraic Clifford algebras $\Cl(H_\gamma)$ can be completed, with respect to this polarization, to super von Neumann algebras $\CC_\gamma$ of type I$_\infty$, which, varying $\gamma$, form a (continuous) bundle of super von Neumann algebras over $LX$ \cite{LudewigClifford}.
The definition of a loop space spinor bundle is now directly analogous to the finite-dimensional case.

\begin{definitionNoNumber}
A \emph{loop space spinor bundle} is a (continuous) bundle $\SB$ of irreducible super left modules for the Clifford von Neumann algebra bundle $\CC$.
\end{definitionNoNumber}

It is important here to insist that the fibers $\SB_\gamma$, $\gamma \in LX$, are a modules for the von Neumann algebra  $\CC_\gamma$ and not only for the algebraic Clifford algebra $\Cl(H_\gamma)$; this is necessary to ensure that $\SB_\gamma$ has the correct isomorphism type.

We show that the obstruction for the construction of a loop space spinor bundle is geometrically represented by the \emph{Lagrangian gerbe} $\LagGrb_{LX}$, a super bundle gerbe on $LX$ previously considered in \cite{Ambler, KristelWaldorf2}, but without taking into account its grading.
In fact, if $X$ is spin, then $\LagGrb_{LX}$ is ungraded, and isomorphic to the obstruction gerbe for lifting the structure group of $LX$ from $L\Spin(d)$ to its basic central extension.
Hence we obtain the following result.

\begin{theoremA}
For a spin manifold $X$ of dimension $d\geq 5$, a loop space spinor bundle $\SB$ on $LX$ exists if and only if $LX$ admits a lift of structure groups to the basic central extension of $L\Spin(d)$.
\end{theoremA}

Constructions of some version of a spinor bundle on loop space have been previously given by several authors \cite{Ambler, SperaWurzbacherSpinors}, starting with McLaughlin \cite[p.~150]{McLaughlin}.
However, as far as we know, as bundle of modules for a von Neumann algebra bundle, a spinor bundle first appears in \cite{KristelWaldorf2}.

A principal achievement in the work \cite{KristelWaldorf2} of Kristel and Waldorf is to make the spinor bundle smooth, by working within a framework of \emph{rigged} von Neumann algebras and modules.
We incorporate smoothness differently:
While $\SB$, just as the Fock space bundle over the bundle of Lagrangians is only a continuous bundle, it turns out that the Lagrangian gerbe is a smooth bundle gerbe (stemming from the fact that the Pfaffian line bundle has a canonical smooth structure).
We then define a \emph{smoothing structure} on a spinor bundle $\SB$ as a certain additional structure ensuring that spinor bundles with smoothing structures correspond precisely to smooth trivializations of the Lagrangian gerbe.

\paragraph{The fusion product.}

Next we explain the notion of a fusion product on a given spinor bundle $\SB$ over $LX$.
Roughly speaking, this encodes a certain compatibility of $\SB$ with respect to the decomposition of loops into paths.
Throughout, we use the path space $PX$ of smooth paths in $X$ that have vanishing derivatives to all orders at both end points, which ensures that we can glue two paths $\gamma_1, \gamma_2 \in PX$ with the same end points to a smooth loop $\gamma_1 \fuse \gamma_2 \in LX$.
For a path $\gamma \in PX$, we consider the Hilbert space
\begin{equation*}
  V_\gamma = L^2([0, \pi], \gamma^*TX).
\end{equation*}
We show that the Clifford algebra $\Cl(V_\gamma)$ has a canonical completion to a von Neumann algebra $\A_\gamma$ (more precisely, to a hyperfinite type III$_1$ factor), and that these von Neumann algebras glue together to a locally trivial bundle $\A$ of von Neumann algebras over $PX$.

As claimed by Stolz and Teichner \cite{StolzTeichnerSpinorBundle}, if $\gamma_1, \gamma_2 \in PX$ are two loops with common end points, the spinor bundle $\SB_{\gamma_1 \fuse \gamma_2}$ becomes an $\A_{\gamma_2}$-$\A_{\gamma_1}$-bimodule in a canonical way.
Constructing this bimodule structure precisely is a somewhat non-trivial task and the main achievement of \S\ref{SectionFusionProduct}.
Our construction then produces a (continuous) bundle of bimodules over the subspace $PX^{[2]} \subset LX$ of loops that decompose into two paths in $PX$.

\begin{definitionNoNumber}
A \emph{fusion product} for a loop space spinor bundle $\SB$ consists of an grading preserving unitary isomorphism
\begin{equation*}
  \Upsilon : \SB_{\gamma_2 \fuse \gamma_3} \boxtimes_{\A_{\gamma_2}} \SB_{\gamma_1 \fuse \gamma_2} \longrightarrow \SB_{\gamma_1 \fuse \gamma_3}
\end{equation*}
of $\A_{\gamma_3}$-$\A_{\gamma_1}$-bimodules for each triple of paths $\gamma_1, \gamma_2, \gamma_3 \in PX$ with common end points.
We require $\Upsilon$ to depend continuously on the paths $\gamma_i$ and to be associative for each suitable quadruple of paths.
\end{definitionNoNumber}

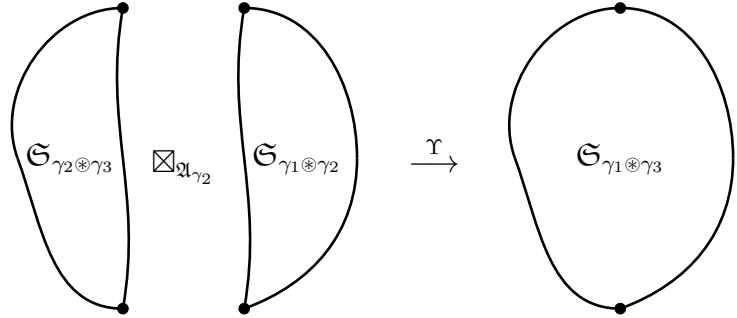
\begin{equation*}
\begin{aligned}
\begin{tikzpicture}
  \coordinate (start) at (0,-2);
  \coordinate (end) at (0,2);
  \coordinate (leftmiddle) at (-1.4, 0);
  \filldraw[black] (start) circle (2pt);
  \filldraw[black] (end) circle (2pt);
  
  \node at (-0.7,0) {$\SB_{\gamma_2 \fuse \gamma_3}$};
  \draw[line width=1pt] (start) to[out=80, in=-100] (end);

  \draw[line width=1pt] (start) to[out=180, in=-70] (leftmiddle);
  \draw[line width=1pt] (leftmiddle) to[out=110, in=-180] (end);
\end{tikzpicture}
\end{aligned}
\boxtimes_{\A_{\gamma_2}}
\begin{aligned}
\begin{tikzpicture}
  \coordinate (start) at (0,-2);
  \coordinate (end) at (0,2);
  \coordinate (rightmiddle) at (1.5,0);
  \filldraw[black] (start) circle (2pt);
  \filldraw[black] (end) circle (2pt);
  
  \node at (0.7,0) {$\SB_{\gamma_1 \fuse \gamma_2}$};

  \draw[line width=1pt] (start) to[out=20, in=-90] (rightmiddle);
  \draw[line width=1pt] (rightmiddle) to[out=90, in=0] (end);

  \draw[line width=1pt] (start) to[out=80, in=-100] (end);

\end{tikzpicture}
\end{aligned}
~~~~
\stackrel{\Upsilon}{\longrightarrow}
~~
\begin{aligned}
\begin{tikzpicture}
  \coordinate (start) at (0,-2);
  \coordinate (end) at (0,2);
  \coordinate (leftmiddle) at (-1.4, 0);
  \coordinate (rightmiddle) at (1.5,0);
  \filldraw[black] (start) circle (2pt);
  \filldraw[black] (end) circle (2pt);
  
  \node at (0,0) {$\SB_{\gamma_1 \fuse \gamma_3}$};

  \draw[line width=1pt] (start) to[out=20, in=-90] (rightmiddle);
  \draw[line width=1pt] (rightmiddle) to[out=90, in=0] (end);

  \draw[line width=1pt] (start) to[out=180, in=-70] (leftmiddle);
  \draw[line width=1pt] (leftmiddle) to[out=110, in=-180] (end);
\end{tikzpicture}
\end{aligned}
\end{equation*}

\vspace{0.2cm}

In the definition of the fusion product, $\boxtimes$ denotes the \emph{Connes fusion product} over $\A_2$, i.e., a bundle version of the appropriate tensor product of (Hilbert) bimodules for von Neumann algebras; see Appendix~\ref{SectionConnesFusion}\&\ref{SectionConnesFusionBundle}.
We stress that it is crucial in the above definition to work in a von Neumann algebra setting, as the \emph{algebraic} tensor product of $\SB_{\gamma_2 \fuse \gamma_3}$ and $\SB_{\gamma_1 \fuse \gamma_2}$ over $\Cl(V_{\gamma_2})$ will \emph{not} be isomorphic to $\SB_{\gamma_1 \fuse \gamma_3}$.
The main result of this paper is the following theorem, which gives a proof of the assertions formulated as Theorems~1 \& 2 in \cite{StolzTeichnerSpinorBundle}.

\begin{theoremB}
  Let $X$ be an oriented Riemannian manifold of dimension $d \geq 5$.
  Then there exists a spinor bundle $\SB$ with fusion product over $X$ if and only if $X$ admits a string structure.
\end{theoremB}

In fact, we will prove the following more refined statement, which starts with an arbitrary loop space spinor bundle $\SB$.
We show that if $X$ admits a string structure, then -- while there may not exist a fusion product for $\SB$ itself -- there exists a line bundle $\mathfrak{T}$ over $LX$ such that $\SB \otimes \mathfrak{T}$ admits a fusion product.
This extends the result of \cite{KristelWaldorf3}.
Moreover, in the presence of a smoothing structures for a spinor bundle $\SB$, as alluded to above, one has a natural notion of smoothness for fusion products.

As discussed by Stolz and Teichner \cite{StolzTeichnerSpinorBundle}, a spinor bundle with a fusion product over $LX$ behaves locally in $X$, and should therefore be viewed a ``higher'' differential geometric object over the manifold itself.
We argue that, using the language developed in \cite{kristel20212vector}, this \emph{stringor bundle} is naturally a super 2-vector bundle; we give a definition in \S\ref{SectionStringorBundle}.

\paragraph{Super bundle 2-gerbes.}

Our proof of Thm.~B uses the machinery of higher differential geometry, more precisely that of \emph{super bundle 2-gerbes}, which where first introduced (without grading) in \cite{StevensonBundle2Gerbes} and whose definition generalizes in a rather straightforward fashion to the super case.
As many facts needed seemed not to be available in the literature (at least for the super case), we give an account in \S\ref{SectionPreliminariesSuperBundleGerbes}.

To prove Thm.~B, we show that given a spinor bundle $\SB$ over $LX$, the obstruction for the existence of a fusion structure on $\mathfrak{S}$ is given by a certain super bundle 2-gerbe $\FusTwoGrb(\mathfrak{S})$ over $X \times X$, which we call the \emph{fusion 2-gerbe}.
This higher gerbe is easily described in terms of the path fibration $PX \to X \times X$: It is given by a super line bundle over the space $PX^{[3]}$ of triples $(\gamma_1, \gamma_2, \gamma_3)$ of paths with common end points.
Explicitly, the fiber of this line bundle over such a triple is the space $\underline{\Hom}(\SB_{\gamma_2 \fuse \gamma_3} \boxtimes_{\A_{\gamma_2}} \SB_{\gamma_1 \fuse \gamma_2}, \SB_{\gamma_1 \fuse \gamma_3})$ of $\A_3$-$\A_1$-bimodule homomorphisms.

On the other hand, it turns out that (via an explicit, canonical isomorphism of super bundle 2-gerbes) $\FusTwoGrb(\mathfrak{S})$ is isomorphic to the another super bundle 2-gerbe $\LagTwoGrb_X$ on $X \times X$, defined independently of a loop space spinor bundle $\mathfrak{S}$.
We call $\LagTwoGrb_X$ the \emph{Lagrangian 2-gerbe} (see \S\ref{SectionLagrangian2Gerbe}).
We show that $\LagTwoGrb_X$ is ungraded if and only $X$ is spin, and that, in this case, it is isomorphic to a certain lifting 2-gerbe (often called Chern-Simons 2-gerbe), which in dimension $d \geq 5$ is known to characterize the string condition.

\paragraph{Further discussion.}

The main new idea of this paper is to combine the theory of von Neumann algebras with the techniques used in higher differential geometry, in order to represent the obstruction against existence of a spinor bundle and/or fusion product within the geometric framework of bundle gerbes.
However, for this to work, we had to make some modifications to the existing theory; in particular, we had to develop some theory of locally trivial super von Neumann algebra bundles, which to our knowledge is not present in the existing literature and therefore is presented in Appendix \S\ref{SectionVNBundlesBimodules}.

Moreover, it is crucial to work in the \emph{super}, i.e., $\Z_2$-graded, setting throughout.
This becomes particularly apparent for the Lagrangian 2-gerbe $\LagTwoGrb_X$: While super bundle gerbes become ordinary bundle gerbes after forgetting the grading, a super bundle 2-gerbe does \emph{not} define an ordinary bundle 2-gerbe.

What is not contained in this work is an investigation of the \emph{geometric} features of the loop space spinor bundle envisioned in \cite{StolzTeichnerSpinorBundle}, in particular the construction of the \emph{conformal connection} on the spinor bundle \cite[Definition~3]{StolzTeichnerSpinorBundle}.
This seems to require working with super bundle (2-)gerbes with connection throughout, which is beyond the scope of the current paper.

\paragraph{Acknowledgements.}

It is a pleasure to thank Peter Kristel, Andr\'e Henriques, Stephan Stolz and Konrad Waldorf for helpful discussions, and Peter Teichner for greatly inspiring my research and for his support over many years.
I would also like to  gratefully acknowledge support from SFB 1085 ``Higher invariants'' funded by the German Research Foundation (DFG).

\setcounter{section}{-1}

\section{Preliminaries on Clifford algebras}
\label{SectionCliffordAlgebras}

In this section, we review well-known classical results on Clifford algebras and Fock representations.

Throughout, by a \emph{``real'' Hilbert space} $H$, we mean a complex Hilbert space with a real structure, i.e., a complex antilinear involution $v \mapsto \overline{v}$.
An \emph{orthogonal transformation} between ``real'' Hilbert spaces $H$, $H^\prime$ is a unitary map $H \to H^\prime$ that intertwines the real structures. 
In particular, for $H^\prime = H$, we have the \emph{orthogonal group} $\O(H) \subset \U(H)$ of $H$.

\subsection{Clifford algebras} 
\label{SectionFock}

The (algebraic) \emph{Clifford algebra} $\Cl(H)$ of a ``real'' Hilbert space is the quotient of the tensor algebra of $H$ by the Clifford relations
\begin{equation} \label{CliffordRelations}
  v \cdot w + w \cdot v  = - 2 \langle \overline{v}, w \rangle \cdot \mathbf{1}, \qquad v, w \in H.
\end{equation}
Here the inner product on $H$ is taken to be complex antilinear in the first component so that the expression on the right hand side of \eqref{CliffordRelations} is complex bilinear.
$\Cl(H)$ is a $*$-algebra, with $*$-operation determined on the generating set $H \subset \Cl(H)$ by the formula
\begin{equation} \label{StarOperation}
 v^* = -\overline{v}, \qquad v \in H.
\end{equation}
It is a fact that $\Cl(H)$ has a unique norm satisfying the $C^*$-identity \cite{Araki1}.
The tensor algebra of $H$ has a natural $\Z$-grading by tensor number, which is not respected by the Clifford relations. 
However, since both sides of \eqref{CliffordRelations} are even, the grading is respected modulo 2, hence $\Cl(H)$ is a super algebra.

The Clifford algebra has the universal property that for any unital $*$-algebra $A$ and any linear map $f : H \to A$ satisfying the relations
\begin{equation*}
  f(v)f(w) + f(w)f(v) = -2\langle \overline{v}, w\rangle \cdot \mathbf{1}_A \qquad \text{and} \qquad f(\overline{v}) = - f(v)^*,
\end{equation*}
for all $v, w \in H$, there exists a unique $*$-homomorphism $\Cl(H) \to A$ extending $f$.
In particular, for any orthogonal transformation $g : H \to H^\prime \subset \Cl(H^\prime)$ between ``real'' Hilbert spaces, we obtain an induced $*$-isomorphism 
\begin{equation}
\label{InducedStarIso}
\Cl_g : \Cl(H) \to \Cl(H^\prime).
\end{equation}
If $H^\prime = H$ so that $\Cl_g$ is an automorphism of $\Cl(H)$, these are often called ``Bogoliubov automorphisms''.
The assignment $g \mapsto \Cl_g$ is compatible with composition.

If $H$ is a Hilbert space with a real structure $v \mapsto \overline{v}$, its \emph{opposite} $-H$ is the ``real'' Hilbert space with the same underlying complex vector space, but with real structure $v \mapsto -\overline{v}$.
The universal property of the Clifford algebra provides a canonical isomorphism
\begin{equation}
\label{CanonicalIsomorphismOpposite}
  \Cl(-H) \cong \Cl(H)^{\op}
\end{equation}
of the Clifford algebra of the opposite Hilbert space to the opposite super algebra of $\Cl(H)$ (see \S~\ref{SectionASuperVN}),
which is the determined by $v \mapsto v^{\op}$ on the generating subset $H \subset \Cl(-H)$.

\subsection{Lagrangians and Fock representations}

A \emph{Lagrangian} in a ``real'' Hilbert space $H$ is a subspace $L\subset H$ with $\overline{L} = L^\perp$.
For a Lagrangian $L \subset H$, denote the \emph{Fock space} associated to $L$ by 
\begin{equation*}
\F_L \defeq \overline{\Lambda L} = \overline{\bigoplus_{n=0}^\infty \Lambda^n L}.
\end{equation*}
Here the right hand side denotes the Hilbert space completion of the algebraic direct sum with respect to its canonical inner product, so that $\F_L$ is a Hilbert space.
$\F_L$ is $\Z_2$-graded via the even/odd grading of the exterior algebra.
Consequently, the space $\BB(\F_L)$ of bounded operators on $\F_L$ is a super von Neumann algebra.
There is a canonical grading preserving $*$-representation 
\begin{equation*}
\pi_L : \Cl(H) \to \BB(\F_L),
\end{equation*}
called the \emph{Fock representation},
which is uniquely determined by the properties
\begin{equation}
\label{DefininingPropertiesPiL}
\begin{aligned}
  \pi_L(v) \xi &= v \wedge \xi, & & \qquad v \in L, ~~ \xi \in \F_L \\
  \pi_L(\overline{v}) &= -\pi_L(v)^*, & & \qquad v \in H
\end{aligned}
\end{equation} 
on the subset $H \subset \Cl(H)$.
Each Fock representation is irreducible \cite[Thm.~2.4.2]{PlymenRobinson}.

If $g : H \to H^\prime$ is an orthogonal transformation between ``real'' Hilbert spaces and $L \subset H$ is a Lagrangian, then $gL \subset H^\prime$ is a Lagrangian in $H^\prime$.
The corresponding Fock representations $\pi_L$ and $\pi_{gL}$ of $\Cl(H)$ and $\Cl(H^\prime)$ are related by
\begin{equation}
\label{RelationFockRep}
\pi_{gL}(gv) = \Lambda_g \pi_L(v) \Lambda_g^*, \qquad v \in H,
\end{equation}
where $\Lambda_g : \F_L \to \F_{gL}$ is the map on Fock spaces induced by $g|_L : L \to gL$.
This relation can be easily  verified by checking that both sides satisfy the defining properties \eqref{DefininingPropertiesPiL}; 
it can be interpreted as saying that $\Lambda_g : \F_L \to \F_{gL}$ is an intertwiner along the $*$-isomorphism $\Cl_g : \Cl(H) \to \Cl(H^\prime)$ induced by $g$ via the universal property of the Clifford algebra.

\subsection{Equivalence and implementation}
\label{SectionEquivalence}

By definition, two Fock representations $\F_{L}$ and $\F_{L^\prime}$ are equivalent (as ungraded representations) if there exists a unitary isomorphism\footnote{$U$ is not necessarily grading preserving, see the discussion below.} $U : \F_{L} \to \F_{L^\prime}$ such that 
\begin{equation}
\label{UnitaryIntertwiner}
   \pi_{L^\prime}(v) = U \pi_{L}(v) U^*, \qquad v \in H.
\end{equation}
The  \emph{Segal-Shale equivalence criterion} states that two Fock representations $\F_{L}$ and $\F_{L^\prime}$ are equivalent in this sense if and only if the difference $P_{L} - P_{L^\prime}$ of the corresponding orthogonal projections is a Hilbert-Schmidt operator \cite[Thm.~3.4.1]{PlymenRobinson}.
This criterion motivates the following definition.

\begin{definition}[Equivalence of Lagrangians]
Two Lagrangians $L, L^\prime \subset H$ are \emph{equivalent} if the difference $P_L - P_{L^\prime}$ of orthogonal projections is a Hilbert-Schmidt operator.
A \emph{polarization} on a ``real'' Hilbert space $H$ is the choice of an equivalence class $\Lag$ of Lagrangians.
\end{definition}

It turns out that a unitary $U : \F_L \to \F_{L^\prime}$ satisfying \eqref{UnitaryIntertwiner} is always either parity-preserving or parity reversing \cite[\S3.5]{PlymenRobinson}.
More precisely, $U$ is parity-preserving if and only if $\dim(L \cap \overline{L}^\prime)$ is even and parity-reversing otherwise \cite[Thm.~3.5.2]{PlymenRobinson}, \cite[Thm.~1.22]{PratWaldron}.
This gives a grading to the Pfaffian lines defined below.

\begin{definition}[Pfaffian line] \label{DefinitionPfaffianLine}
Let $H$ be a ``real'' Hilbert space with a polarization $\Lag$.
The \emph{Pfaffian line} of two equivalent Lagrangians $L$, $L^\prime \in \Lag$ is the space of $\Cl(H)$-module homomorphisms
\begin{equation} \label{PfaffianLine}
 \Pf(L, L^\prime) = \underline{\Hom}(\F_{L}, \F_{L^\prime}),
\end{equation}
which can be described as the space of bounded linear maps $\Phi : \F_{L} \to  \F_{L^\prime}$ that are either grading preserving or grading reversing and satisfy the relation
\begin{equation}
\label{PfaffianRelation}
  \pi_{L^\prime}(v)\, \Phi = (-1)^{|\Phi|} \Phi\, \pi_{L}(v), \qquad v \in H.
\end{equation}
The Pfaffian line is graded by parity of its elements.
\end{definition}

\begin{remark}
\label{RemarkSignInPfaffianRelation}
The additional sign in \eqref{PfaffianRelation} comes from the Koszul rule; see \S\ref{SectionSuperBimodules} for the general definition of intertwiners between super bimodules.
A unitary map satisfying relation \eqref{PfaffianRelation} can be turned into a unitary map satisfying \eqref{UnitaryIntertwiner} by composing with the grading operator.
%
%
\end{remark}

There are the following additional structures on the Pfaffian lines.

\begin{enumerate}[(1)]

\item
As $L_1$ and $L_2$ are equivalent, their Pfaffian line is non-zero.
On the other hand, it follows from the irreducibility of the Fock representations that two elements of $\Pf(L_1, L_2)$ differ only by a scalar $\lambda \in \C$, in other words, $\Pf(L_1, L_2)$ is one-dimensional, hence indeed a complex line.
As there always exist \emph{unitary} intertwiners $U : \F_{L_1} \to \F_{L_2}$ between Fock representations for equivalent Lagrangians, which are either grading preserving or grading reversing, it follows that any element $\Phi \in \Pf(L_1, L_2)$ is necessarily a scalar multiple of a unitary.

\item
From the fact that its elements are multiples of unitaries, it follows that $\Pf(L_1, L_2)$ carries a natural inner product.
It satisfies 
\begin{equation}
\label{PfaffianInnerProduct}
\Phi^*\Phi^\prime = \langle \Phi, \Phi^\prime \rangle_{\Pf} \cdot \id_{\F_{L_1}}
\end{equation}
 and is uniquely determined by this relation.
The inner product induces an identification $\Pf(L_1, L_2)^* \cong \Pf(L_2, L_1)$ of the dual Pfaffian line.
%
%

\item
For any triple of Lagrangians $L_1, L_2, L_3 \in \Lag$, composition of operators provides a grading preserving isomorphism of super lines
\begin{equation}
\label{CompositionMap}
 \Pf(L_2, L_3) \otimes  \Pf(L_1, L_2) \longrightarrow  \Pf(L_1, L_3).
\end{equation}
\end{enumerate}

Let $(H, \Lag)$ and $(H^\prime, \Lag^\prime)$ be polarized ``real'' Hilbert spaces.
An orthogonal transformation $g: H \to H^\prime$ is \emph{restricted} if $gL \in \Lag^\prime$ for all $L \in \Lag$.
In particular, for $H^\prime = H$, one makes the following definition.

\begin{definition}[Restricted orthogonal group]
Let $H$ be a ``real'' Hilbert space with a polarization $\Lag$.
The corresponding \emph{restricted orthogonal group} is defined as
\begin{equation*}
  \O_{\res}(H) \defeq \{ g \in \O(H) \mid \forall L \in \Lag : gL \in \Lag\}.
\end{equation*}
\end{definition}

The Pfaffian lines enjoy the following equivariance with respect to $\O_{\res}(H)$:
If $g \in \O_{\res}(H)$ and $L, L^\prime \in \Lag$, there is a grading preserving isomorphism
\begin{equation}
\label{EquivariancePfaffianLines}
\Pf(L, L^\prime) \longrightarrow \Pf(gL, gL^\prime), \qquad \Phi \longmapsto \Lambda_g \Phi \Lambda_g^*.
\end{equation}

Fix a ``real'' Hilbert space $H$ with a polarization $\Lag$.
Historically, for $L, L^\prime \in \Lag$ the task of constructing an intertwiner $\F_L \to \F_{L^\prime}$ between Fock modules is known as the \emph{equivalence problem}. 
Closely related is the \emph{implementation problem}, which is the task of implementing the Bogoliubov automorphism corresponding to an orthogonal transformation on some Fock space.
Here an \emph{implementer} is a unitary operator $U$ on $\F_{L}$, which is either grading preserving or grading reversing and satisfies
\begin{equation}
\label{ImplementerRelation}
\pi_{L}(gv) = (-1)^{|U|} U \pi_{L}(v) U^*, \qquad v \in H,
\end{equation}
 for some $g \in \O_{\res}(H)$.
One says that $U$ \emph{implements} the Bogoliubov automorphism $\Cl_g$, or just that $U$ \emph{implements} $g$.
This leads to the following definition.

\begin{definition}[Implementer group]
Let $L \in \Lag$.
The group 
\begin{equation*}
\Imp_L \defeq \bigl\{ U \in \U(\F_L) \mid \exists g \in \O_{\res}(H) : U \text{ implements } g\bigr\}
\end{equation*}
of unitary transformations of $\F_{L}$ satisfying \eqref{ImplementerRelation} for \emph{some} $g \in \O_{\res}(H)$ is called the \emph{group of implementers} for $L$.
\end{definition}

\begin{remark}
The relation \eqref{ImplementerRelation} differs from the usual one by a sign in the case that $U$ is grading reversing, due to our general sign convention \eqref{IntertwiningRelation} for intertwiners.
Given an implementer satisfying \eqref{ImplementerRelation}, one obtains an intertwiner in the sense of, e.g., \cite[\S3.3]{PlymenRobinson} by composing with the grading operator.
\end{remark}

\begin{remark}
\label{RemarkParityImplementer}
An implementer $U$ for $g$ is grading preserving if $g$ lies in the identity component of $\O_{\res}(H)$ and grading reversing otherwise \cite[Thm.~3.5.1]{PlymenRobinson}.
\end{remark}

Any solution $U$ to the implementation problem \eqref{ImplementerRelation} gives a unitary element $U \Lambda_g^* \in \Pf(gL, L)$ of the Pfaffian line, i.e., a solution to the equivalence problem for $gL$ and $L$.
For details, see \cite[\S3.2]{PlymenRobinson}.

\subsection{The von Neumann completion of the Clifford algebra}

Let $H$ be a ``real'' Hilbert space.
For any Lagrangian $L \subset H$, the Fock representation $\pi_L$ induces a topology on the Clifford algebra $\Cl(H)$ by pulling back the ultraweak topology of $\BB(\F_L)$ along $\pi_L$; explicitly, this topology consists of the sets $\pi_L^{-1}(O)$, where $O \subseteq \BB(\F_L)$ ranges over the ultraweakly open sets.
Given a unitary transformation $U : \F_{L_1} \to \F_{L_2}$, a subset $O \subseteq \BB(\F_{L_1})$ is open for the ultraweak topology if and only if $UOU^* \subseteq \BB(\F_{L_2})$ is open.
Therefore, the relation \eqref{UnitaryIntertwiner} implies that two equivalent Lagrangians $L_1$, $L_2$ induce the same topology on $\Cl(H)$.
In other words, this topology only depends on the equivalence class.
Hence we can make the following definition.

\begin{definition}[Completion of Clifford algebra]
\label{DefinitionCompletionClifford}
Let $\Lag$ be an equivalence class of Lagrangians in $H$.
The \emph{completion} $\CC$ {of} $\Cl(H)$ \emph{with respect to} $\Lag$ is the abstract completion (in the sense of topological vector spaces) with respect to the ultraweak topology induced by the Fock representations $\pi_L$ for any $L \in \Lag$.
\end{definition}

Given a Lagrangian $L \in \Lag$, the Fock representation $\pi_L : \Cl(H) \to \BB(\F_L)$ extends by ultraweak continuity to an isomorphism of topological vector spaces from $\CC$ to the closure of the image $\pi_L(\Cl(H))$ inside $\BB(\F_L)$ with respect to the ultraweak topology.
By von Neumann's bicommutant theorem, this ultraweak closure is equal to both the closure with respect to the weak and the strong topology, and, moreover, coincides with the bicommutant $\pi_L(\Cl(H))^{\prime\prime} \subset \BB(\F_L)$.
We obtain that the completion $\CC$ is a super von Neumann algebra, and that, for any $L \in \Lag$, the Fock representation $\pi_L$ extends to a $*$-representation
\begin{equation}
\label{CanonicalIsoC}
 \pi_L :  \CC \longrightarrow \BB(\F_L).
\end{equation}
These extensions are $*$-isomorphisms, so we obtain that $\CC$ is a super factor of type I, of even kind (see \cite[\S2.1]{LudewigClifford}).

\begin{remark}
If $H$ is finite-dimensional (necessarily even to admit a Lagrangian), this is just the statement that the Clifford algebra of an even-dimensional complex vector space is isomorphic to a matrix algebra, where the isomorphism is non-canonical, but depends on the choice of a Lagrangian $L$ and an isomorphism $L \cong \C^{\dim(H)/2}$.
\end{remark}

If $(H, \Lag)$ and $(H^\prime, \Lag^\prime)$ are two polarized ``real'' Hilbert spaces and $g: H \to H^\prime$ is an orthogonal transformation sending $\Lag$ to $\Lag^\prime$, then the relation \eqref{RelationFockRep} implies that the induced isomorphism \eqref{InducedStarIso} extends to a $*$-isomorphism 
\begin{equation*}
  \Cl_g : \CC \longrightarrow \CC^\prime
\end{equation*}
between the von Neumann completions with respect to the polarizations.
This assignment is compatible with composition so that we obtain a functor from the category of polarized Hilbert spaces and restricted orthogonal transformations to the category of super von Neumann algebras and isomorphisms.
%
%
In particular, for $H^\prime = H$ and $\Lag = \Lag^\prime$, we obtain a group homomorphism
\begin{equation}
\label{GroupHomOresAutC}
  \O_{\res}(H) \longrightarrow \Aut(\CC).
\end{equation}

The category of polarized Hilbert spaces has an involutive endofunctor which sends a polarized Hilbert space $(H, \Lag)$ to the \emph{opposite} polarized Hilbert space $(-H, \Lag^\perp)$, where $-H$ is the opposite ``real'' Hilbert space and 
\begin{equation}
\label{ConjugateClassOfLagrangians}
\Lag^\perp = \{L^\perp = \overline{L} \mid L \in \Lag\}.
\end{equation}
Similarly, the category of super von Neumann algebras has an involutive endofunctor sending a super von Neumann algebra $A$ to its super opposite $A^{\op}$ (see \S\ref{SectionASuperVN}).
The following lemma states that the functor that assigns the completed Clifford algebra intertwines these involutions.

\begin{lemma}
\label{LemmaCanonicalIsomorphismOpposite}
Let $H$ be a ``real'' Hilbert space and let $\Lag$ be an equivalence class of Lagrangians in $H$.
The canonical isomorphism \eqref{CanonicalIsomorphismOpposite} extends to a $*$-isomorphism $\CC_- \cong \CC^{\op}$, where $\CC$ is the  completion of $\Cl(H)$ with respect to $\Lag$ and $\CC_-$ is the  completion of $\Cl(-H)$ with respect to the opposite equivalence class $\Lag^\perp$.
\end{lemma}

\begin{proof}
For a complex vector space $K$, we write $\overline{K}$ for its complex conjugate, the elements of which we denote by $\overline{v}$, for $v \in K$.
Let $\Gamma$ be the real structure of $H$.
Choose a Lagrangian $L \in \Lag$.
Then we have a canonical (complex-\emph{linear}) isomorphism $\F_{\overline{L}} \cong \overline{\F_L}$, given by sending $\xi \mapsto \overline{\Lambda_\Gamma \xi}$ (here $\Lambda_\Gamma : \F_{\overline{L}} \to \F_L$ denotes the map on Fock spaces induced by the real structure $\Gamma$).
This induces an isomorphism $\BB(\F_{\overline{L}}) \cong \overline{\BB(\F_L)}$, given by conjugation with $\Lambda_\Gamma$.
%

Denote by $\pi_{\overline{L}}^- : -H \to \BB(F_{\overline{L}})$ the Fock representation corresponding to the Lagrangian $\overline{L} \subset - H$, which is characterized by $\pi_{\overline{L}}^-(v)^* = \pi_{\overline{L}}(\overline{v})$ for $v \in -H$ and $\pi_{\overline{L}}(v) \xi = v \wedge \xi$ for $v \in \overline{L}$ and $\xi \in \F_{\overline{L}}$.
Consider the following diagram of $*$-algebras.
\begin{equation}
\label{DiagramOppositeNegative}
\begin{tikzcd}[column sep=1.3cm]
& \Cl(-H) 
	\ar[d, "\pi_{\overline{L}}^-"'] 
	\ar[dl, bend right=30, hookrightarrow] 
	\ar[rrr, rightsquigarrow] 
&[-0.6cm] & & 
\Cl(H) 
	\ar[d, "\pi_L"]
	\ar[dr, bend left=30, hookrightarrow]
	& \\
\CC_- 
\ar[r, "\pi_{\overline{L}}^-", "\cong"'] 
& 
\BB(\F_{\overline{L}}) 
\ar[r, "\substack{\text{conjugation} \\ \text{by } \Lambda_{\Gamma}}"]
& 
\overline{\BB(\F_L)}
\ar[r, rightsquigarrow, "\sharp"]
& 
\BB(\F_L)
\ar[r, "\substack{\text{conjugation} \\ \text{by } \Lambda_{-i}}"]
&
\BB(\F_L)
&
\CC \ar[l, "\pi_L"', "\cong"]
\end{tikzcd}
\end{equation}
where the curved arrows are the canonical inclusions into the completions,
 the top squiggly arrow is the canonical super anti-isomorphism obtained from \eqref{CanonicalIsomorphismOpposite} and the bottom squiggly arrow is the canonical super anti-homomorphism $\sharp$ given by
 \begin{equation}
\label{SharpIsomorphism}
  a^\sharp \defeq \begin{cases} a^* & a \text{ even} \\ i a^* & a \text{ odd}. \end{cases}
\end{equation}
See \S\ref{SectionASuperVN} for a more detailed discussion of opposite super algebras, super anti-homomorphisms and the homomorphism $\sharp$.
%
%

To show that the square in the middle of \eqref{DiagramOppositeNegative} commutes, it suffices to show that the two compositions agree on elements of $H \subset \Cl(-H)$.
The north east composition $\Cl(-H) \to \BB(\F_L)$ is given on elements of $H$ simply by $v \mapsto \pi_L(v)$ (as the top squiggly arrow is the identity on $H \subset \Cl(-H)$), while the south west composition is given by
\begin{equation*}
H \ni  v \longmapsto \pi^\prime(v) \defeq \Lambda_{-i} (\Lambda_\Gamma \pi_{\overline{L}}^-(v) \Lambda_\Gamma^*)^\sharp \Lambda_i
  = i \Lambda_{-i} \pi_L^-(v) \Lambda_{i}.
\end{equation*}
%
%
It is now straight forward to show that $\pi^\prime$ satisfies the defining properties \eqref{DefininingPropertiesPiL} of the Fock representation $\pi_L$, so that $\pi^\prime(v) = \pi_L(v)$ for all $v \in H$.
%
%
Therefore, the square commutes.
It follows that the composition of the bottom arrows is the desired $*$-isomorphism $\CC_- \to \CC^{\op}$ extending $\Cl(-H) \to \Cl(H)^{\op}$.
\end{proof}

%
%
%
%
%


\subsection{The Lagrangian Grassmannian and the Pfaffian line bundle}
\label{SectionBundleLagGrass}

Let $(H, \Lag)$ be a polarized Hilbert space and let $\CC$ be the completion of $\Cl(H)$ with respect to $\Lag$ (see Definition~\ref{DefinitionCompletionClifford}). 
We equip $\O_{\res}(H)$ with the coarsest topology that makes both the group homomorphism \eqref{GroupHomOresAutC} into $\Aut(\CC)$ and the inclusion into $\O(H)$ continuous.
Here we always equip the automorphism group of a (super) von Neumann algebra with Haagerup's u-topology (see \S\ref{SectionVNBundlesBimodules}).
It turns out that with this topology, $\O_{\res}(H)$ is a Banach Lie group \cite[\S6.2 \& \S12.4]{PressleySegal}.

For a Lagrangian $L \in \Lag$, any unitary $u \in \U(L)$ extends uniquely to an element of $\O_{\res}(H)$; this identifies $\U(L)$ with the closed subgroup of $g \in \O_{\res}(H)$ that commute with the complex structure $J_L = i(P_L - P_{L^\perp})$.

The obvious action of $\O_{\res}(H)$ on $\Lag$ is transitive, with stabilizer at $L \in \Lag$ the closed subgroup $\U(L) \subset \O_{\res}(H)$.
This turns $\Lag$ into a homogeneous space for $\O_{\res}(H)$, and there is a unique smooth structure making the action of $\O_{\res}(H)$ smooth \cite[Prop.~11 of {\S}III.11]{BourbakiLieGroups}.
Explicitly, this structure is characterized by the property that for each $L \in \Lag$, the map
\begin{equation}
\label{MapAlpha}
  \alpha_L : \O_{\res}(H) \longrightarrow \Lag, \qquad g \longmapsto gL
\end{equation}
descends to a diffeomorphism $\O_{\res}(H) / \U(L) \cong \Lag$.
%

\medskip

For any $L \in \Lag$, the group $\Imp_L$ of implementers on $L$ can be described as the pullback of the principal $\U(1)$-bundle $\U(\F_L) \to \Aut(\CC)$ along the group homomorphism \eqref{GroupHomOresAutC},
\begin{equation}
\label{ImplementerExtension}
\begin{tikzcd}
\Imp_L \ar[r, dashed] \ar[d, dashed]& \U(\F_L) \ar[d] \\
\O_{\res}(H) \ar[r] & \Aut(\CC).
\end{tikzcd}
\end{equation}
It turns out that also $\Imp_L$ is a Banach Lie group \cite[\S 3.5]{KristelWaldorf1}; 
hence we obtain a central extension
\begin{equation*}
  \U(1) \longrightarrow \Imp_{L} \longrightarrow \O_{\res}(H)
\end{equation*}
of Banach Lie groups.
Here the right map sends an implementer to the orthogonal transformation it implements.

\medskip

The Fock spaces $\F_L$, $L \in \Lag$ glue together to a continuous bundle $\F$ of super $\Cl(H)$-modules, with the property that the action of $\O_{\res}(H)$ given by
\begin{equation}
\label{FockSpaceAction}
\Lambda_g : \F_L \to \F_{gL}
\end{equation}
is continuous.
In other words, $\F$ is an equivariant bundle of Hilbert spaces for the $\O_{\res}(H)$-action on $\Lag$.
In fact, this determines the bundle structure completely. 
Observe here that for each $L \in \Lag$, the restriction of \eqref{FockSpaceAction} to the stabilizer $\U(L)$ is a continuous group action on $\F_L$, which is, however, not smooth in infinite dimensions (for the norm topology of $\U(\F_L); see $\cite[\S2.3]{Ottesen}). 
%
%

Notice that this equivariant structure does \emph{not} intertwine the Clifford actions (hence $\F$ is not equivariant as a $\Cl(H)$-module bundle); 
instead, by \eqref{RelationFockRep}, $\Lambda_g$ is an intertwiner along the ``Bogoliubov'' automorphism $\Cl_g$ of $\Cl(H)$.
All statements above stay true when $\Cl(H)$ is replaced by its completion $\CC$ with respect to $\Lag$.

\medskip

Over $\Lag \times \Lag$, we have the Pfaffian line bundle
\begin{equation*}
\Pf \defeq \underline{\Hom}(\F_1, \F_2),
\end{equation*}
where $\F_i$ denotes the pullback of $\F$ along the projection onto the $i$-th factor.
In infinite-dimensions, as $\F$ is only a continuous bundle, this only gives a continuous bundle structure at first.
However, it is straight forward to show that there is a unique smooth structure on the Pfaffian line bundle such that composition map \eqref{CompositionMap} is smooth and such that $\Pf$ is smoothly $\O_{\res}(H)$-equivariant for the action given by \eqref{EquivariancePfaffianLines}.
It has the property that for each $L \in \Lag$, the map
\begin{equation}
\label{IsoImpLPfL}
  \Imp_L \longrightarrow \Pf_L, \qquad U \longmapsto U \Lambda_g^* \in \Pf(gL, L),
\end{equation}
is smooth and descends to a diffeomorphism on the quotient of the left hand side by the subgroup $\U(L)$ when the right hand side is restricted to the subbundle of unitary intertwiners between Fock spaces.
In \eqref{IsoImpLPfL}, $U$ is an implementer for $g$. 

\begin{remark}
\label{RemarkParityPf}
The Lagrangian Grassmannian $\Lag$ has two connected components.
Indeed, since $\Lag$ is a homogenous space for $\O_{\res}(H)$ with isotropy $\U(L)$, it follows from the long exact sequence of homotopy groups and the contractibility of $\U(L)$ that $\Lag$ is homotopy equivalent to $\O_{\res}(H)$ (compare Prop.~12.4.2 in \cite{PressleySegal}).
That the restricted orthogonal group has two connected components is part of Thm.~6.3 in \cite{Araki1}.
Given $L, L^\prime \in \Lag$, the Pfaffian line $\Pf(L, L^\prime)$ is even if and only if $L, L^\prime$ lie in the same connected component and odd otherwise. 
This follows from the (grading preserving) isomorphism \eqref{IsoImpLPfL} and the corresponding fact for $\Imp_L$ \cite[Thm.~3.5.2]{PlymenRobinson}.
\end{remark}

\section{The spinor bundle on loop space}

In this section, we first define spinor bundles corresponding to bundles of polarized Hilbert spaces on arbitrary (infinite-dimensional) manifolds and then specialize to the example of the loop space.

\subsection{Spinor bundles and the Lagrangian gerbe}
\label{SectionLagrangianGerbe}

Let $M$ be a (possibly infinite-dimensional) manifold.

\begin{definition}[Bundle of polarized Hilbert spaces]
\label{DefinitionPolarizedHilbertSpace}
A \emph{smooth bundle of polarized Hilbert spaces} is a bundle of ``real'' Hilbert spaces $H$ over $M$ (see \S\ref{SectionvNbimodbundles}) such that each fiber $H_p$ carries a polarization $\Lag_p$.
We require that there exists a collection of local trivializations $\{\tau_i : H|_{O_i} \to O_i \times H_0\}_{i \in I}$ such that the following holds:
\begin{enumerate}[(i)]
\item 
The domains $O_i$, $i \in I$, form an open cover of $M$.
\item 
The typical fiber $H_0$ has a polarization $\Lag_0$ such that for each $i \in I$, the trivialization $\tau_i$ is fiberwise a restricted orthogonal transformation.
\item
By the previous requirement, the corresponding transition functions defined over two-fold overlaps of the cover $\{O_i\}_{i \in I}$ are fiberwise restricted orthogonal transformations.
We now require that the corresponding functions $O_i \cap O_j \to \O_{\res}(H_0)$ are smooth.
\end{enumerate}
\end{definition}

\begin{remark}
As bundles of Hilbert spaces are often not smooth, it may be useful in some situations to relax the above definition by dropping the smoothness requirement in (ii) and instead require continuity of the functions $O_i \cap O_j \to \O_{\res}(H)$ for the coarsest topology on $\O_{\res}(H_0)$ that makes the inclusion into $\Aut(\CC_0)$ continuous, where $\CC_0$ is the completion of $\Cl(H_0)$ with respect to $\Lag_0$ (see Definition~\ref{DefinitionCompletionClifford}).
This topology is coarser than the Lie group topology of $\O_{\res}(H_0)$ considered here, but finer than the strong topology induced from $\O(H)$.
With this weakened notion, topological versions of the statements below still hold; for example, the Lagrangian gerbe $\Lag$ is still defined as a topological bundle gerbe.
However, it turns out that the main example for a polarized Hilbert space bundle in this paper (which is essentially the completed tangent bundle of the smooth loop space of a Riemannian manifold) is smooth in the above sense, which allows to stay in the smooth setting.
\end{remark}

Given a bundle $H$ of polarized Hilbert spaces over $M$, the fiberwise completions $\CC_p$ of the Clifford algebras $\Cl(H_p)$, $p \in M$, glue together to a (continuous) bundle $\CC$ of super von Neumann algebras over $M$.
Local trivializations for $\CC$ are immediately obtained from those of $H$, using continuity of the group homomorphism \eqref{GroupHomOresAutC}.
%

\begin{definition}[Spinor bundle]
\label{DefinitionAbstractSpinorBundle}
Let $H$ be a bundle of polarized Hilbert spaces on $M$ and let $\CC$ be corresponding bundle of super von Neumann algebras.
A \emph{spinor bundle} for $H$ is a (continuous) bundle $\SB$ of irreducible super left modules over $M$ for the super von Neumann algebra bundle $\CC$.
\end{definition}

See Definition~\ref{DefinitionBimoduleBundle} for the general notion of a bundle of super left modules for a super von Neumann algebra bundle.
In particular, the definition entails that over each $p \in M$, the fiber $\SB_p$ is a super Hilbert space together with a grading preserving $*$-isomorphism $\CC_p \cong \BB(\SB_p)$.

\medskip

We now construct a geometric obstruction for the existence of a spinor bundle for a bundle $H$ of polarized Hilbert spaces, which is a super bundle gerbe.
Given a bundle of polarized Hilbert spaces, the fiber polarizations $\Lag_p$, $p \in M$, glue together to a smooth fiber bundle $\Lag \to M$, whose transition functions are obtained from those of $H$, using the smoothness of the action of $\O_{\res}(H_0)$ on $\Lag_0$ (here $(H_0, \Lag_0)$ is the typical fiber of $H$).
The Fock spaces $\F_L$, $L \in \Lag$, form a (continuous) bundle $\mathfrak{F}$ over $\Lag$, the fibers of which are super left modules for (the pullback to $\Lag$ of) the Clifford von Neumann algebra bundle $\CC$.
The Pfaffian line bundle
\begin{equation*}
  \Pf \defeq \underline{\Hom}(\F_1, \F_2)
\end{equation*}
is then a smooth super line bundle over the 2-fold fiber product $\Lag^{[2]}$.
Here, following Notation~\ref{NotationSubmersion}, $\F_i$ denotes the pullback of $\F$ along the $i$-th projection map $\Lag^{[2]} \to \Lag$.
(We remark here that since $\F$ is only a continuous bundle, this only gives a continuous line bundle at first.
However the smooth structure of the Pfaffian line bundle $\Pf^0$ over the homogeneous space $\Lag_0$ carries over to $\Pf$ by smoothness of the $\O_{\res}(H_0)$-action on $\Pf^0$.)
The composition map \eqref{CompositionMap} lifts to a smooth isomorphism of super line bundles
\begin{equation}
\label{MultiplicationBundleMorphism}
  \lambda : \Pf_{23} \otimes \Pf_{12} \longrightarrow \Pf_{13}
\end{equation}
over the three-fold fiber product $\Lag^{[3]}$ (see Notation~\ref{NotationSubmersion}).
These data form the Lagrangian gerbe for $H$, defined as follows.

\begin{definition}[Lagrangian gerbe]
\label{DefinitionLagrangianGerbe}
Let $H$ be a bundle of polarized Hilbert spaces over $M$.
The \emph{Lagrangian gerbe} for $H$ is the super bundle gerbe depicted as follows.
\begin{equation}
\label{LagrangianGerbe}
  \LagGrb_H =   \left[
\begin{tikzcd}
  & \Pf \ar[d,  dotted, -] & \substack{\text{composition} \\ \text{map}~\lambda} \ar[d, -, dotted] & \substack{\text{associativity} \\ \text{of composition}} \ar[d, -, dotted] 
  \\
  \Lag \ar[d]& \Lag^{[2]} \ar[l, shift left=1mm] \ar[l, shift right=1mm]&  \Lag^{[3]} \ar[l, shift left=2mm] \ar[l, shift right=2mm] \ar[l] & \Lag^{[4]} \ar[l, shift left=1mm] \ar[l, shift right=1mm] \ar[l, shift left=3mm] \ar[l, shift right=3mm]
  \\
  M
\end{tikzcd}
\right].
\end{equation}
Its cover is the bundle $\Lag$ of Lagrangians in $H$, its super line bundle is the Pfaffian line bundle $\Pf$ over $\Lag^{[2]}$ and its gerbe multiplication is the composition map \eqref{MultiplicationBundleMorphism}, which satisfies the necessary coherence \eqref{CoherenceGerbeMultiplication} over $\Lag^{[4]}$ by associativity of the composition of operators.
\end{definition}

This super bundle gerbe has been previously considered in \cite[\S5]{KristelWaldorf2} (in a somewhat different language), but without using the grading of its defining line bundle; see also \cite[Chapter 11]{Ambler}.

\begin{example}
\label{ExampleFiniteDimension}
Suppose that $M$ is a Riemannian manifold of finite dimension $d$.
Consider the vector bundle $H$ over $M$ given by
\begin{equation*}
  H = \begin{cases} T^\C M & \text{if $d$ is even} \\ T^\C M \oplus \C & \text{if $d$ is odd}. \end{cases}
\end{equation*}
It can be identified with the associated bundle $\O(M) \times_{\O(d)} H_0$, where $\O(M)$ is the principal $\O(d)$-bundle of orthogonal frames on $M$ and $H_0 = \C^d$ respectively $\C^{d+1}$.
Then a spinor bundle $\SB$ for $H$ exists if and only if $M$ has a $\Spin^c$ structure, and in that case, is just a usual complex spinor bundle.

In this situation, one can show that the characteristic classes of the Lagrangian gerbe are the first Stiefel-Whitney class $w_1(M) \in H^1(M, \Z_2)$ and the third integral Stiefel-Whitney class $W_3(M) \in H^3(M, \Z)$ of $M$, which are precisely the obstructions for the existence of a $\Spin^c$-structure.
\end{example}

\begin{example}
\label{ExampleLocalSectionsLag}
If the bundle $\Lag$ of Lagrangians in $H$ has a global section $L : M \to \Lag$, then a spinor bundle is just obtained by setting $\SB = F_L = L^*\F$.
However, this assumption is very restrictive: 
Indeed, the choice of a Lagrangian in $H_p$ is the same thing as the choice of an orthogonal complex structure in $H_p$ (where a Lagrangian $L$ corresponds to the complex structure $J_L = i(P_L-P_{\overline{L}})$).
Hence in the context of Example~\ref{ExampleFiniteDimension} with $d$ even, a global section of $\Lag$ is the same thing as an almost complex structure on $M$.
But many spin$^c$ manifolds do not admit an almost complex structure; for example, it is well-known that $M = S^{2n}$ only admits an almost complex structure when $n \in \{1, 2, 3\}$.
\end{example}

\begin{example}
\label{ExampleStrictTrivLag}
Compared to Example~\ref{ExampleLocalSectionsLag}, a less restrictive way to obtain a spinor bundle is to choose an open cover $(O_i)_{i \in I}$ such that $\Lag|_{O_i}$ admits sections $L_i$, together with grading preserving bundle isomorphisms $U_{ij} : \F_{L_i} \to \F_{L_j}$ defined over $O_i \cap O_j$.
These are then required to satisfy the cocycle condition
\begin{equation}
\label{CocycleConditionUij}
  U_{ik} = U_{jk}U_{ij},
\end{equation}
in order for the bundles $(\F_{L_i})_{i \in I}$ over the open cover to glue together to a spinor bundle $\SB$.
In fact, any spinor bundle on $M$ is isomorphic to one obtained this way.
\end{example}
%

\begin{example}
\label{ExampleSpinorBdlFromTrivialization}
The data $(O_i)_{i \in I}$ and $(U_{ij})_{ij \in I}$ from Example~\ref{ExampleStrictTrivLag} are a special case of a trivialization of the Lagrangian gerbe $\LagGrb_H$.
The above construction can be generalized as follows for an arbitrary trivialization $\mathfrak{t} = (\mathfrak{T}, \tau)$ of $\LagGrb_H$, where $\mathfrak{T}$ is a smooth super line bundle over $\Lag$ and $\tau : \mathfrak{T}_2 \otimes \Pf \longrightarrow \mathfrak{T}_1$ is an isomorphism of super line bundles over $\Lag^{[2]}$.
To obtain a spinor bundle $\SB$ from the trivialization $(\mathfrak{T}, \tau)$, we define a bundle $\tilde{\SB} \to \Lag$ of super left $\CC$-modules by the formula $\tilde{\SB} = \F \otimes \mathfrak{T}$.
Over $\Lag^{[2]}$, we have a canonical isomorphism of super vector bundles
\begin{equation*}
\begin{tikzcd}
\tilde{\SB}_2 =  \F_2 \otimes \mathfrak{T}_2 \cong \F_1  \otimes \Pf \otimes \mathfrak{T}_2 \cong \F_1 \otimes \mathfrak{T}_2 \otimes \Pf \ar[r, "\tau", "\cong"'] & \F_1 \otimes \mathfrak{T}_1 = \tilde{\SB}_1
\end{tikzcd}
\end{equation*}
given by $\tau$ and the canonical bundle isomorphism $\F_2 \cong \F_1 \otimes \Pf$ coming from the fiberwise application isomorphism
\begin{equation*}
\F_{L_1} \otimes \underline{\Hom}(\F_{L_1}, F_{L_2}) \otimes \F_{L_2}, \qquad (L_1, L_2) \in \Lag^{[2]}.
\end{equation*}
The three different pullbacks of this isomorphism to $\Lag^{[3]}$ satisfy the obvious cocycle condition, hence $\tilde{\SB}$ descends to a bundle $\SB$ of super left $\CC$-modules on $M$.
By construction, the fiber $\SB_p$ of $\SB$ is isomorphic to $\F_L$ for any $L \in \Lag_p$, hence irreducible as a $\CC$-module.
In other words, $\SB = \SB_{(\mathfrak{T}, \tau)}$ is a spinor bundle for $H$.
\end{example}

We may form the category $\textsc{SpinBdl}_H$ whose objects are spinor bundles $\SB$ for $H$ and whose morphisms are isomorphisms of super $\CC$-module bundles (see Definition~\ref{DefinitionBimoduleBundleIsomorphism}).
The construction from Example~\ref{ExampleSpinorBdlFromTrivialization} that associates to a trivialization of $\LagGrb_H$ a spinor bundle for $H$ can be easily upgraded to a functor
\begin{equation}
\label{TrivializationFunctor}
  \mathrm{Triv}(\LagGrb_H) \longrightarrow \textsc{SpinBdl}_H.
\end{equation}
%
%

\begin{theorem} 
\label{ThmEquivalenceFinite}
The above functor is essentially surjective and faithful. 
\end{theorem}

In particular, the above theorem shows that non-triviality of the Lagrangian gerbe $\LagGrb_H$ is the obstruction to the existence of a spinor bundle for $H$; compare also Theorem~5.3.7 of \cite{KristelWaldorf2}.

\begin{proof}
That the functor is faithful follows directly from the construction.

To see that the functor is essentially surjective, given a spinor bundle $\SB$ over $M$, consider the super line bundle 
\begin{equation}
\label{AssociatedSmoothnessLine}
  \mathfrak{N} = \underline{\Hom}(\F, \pi^*\SB),
\end{equation}
over $\Lag$, where $\pi^*\SB$ is the pullback along the footpoint projection $\pi: \Lag \to M$.
$\mathfrak{N}$ comes with a canonical grading preserving isomorphism of super line bundles
\begin{equation}
\label{CanonicalCompositionMap}
\nu : \mathfrak{N}_2 \otimes \Pf \rightarrow \mathfrak{N}_1,
\end{equation}
see Notation~\ref{NotationSubmersion}.
The tuple $(\mathfrak{N}, \nu)$ then forms a continuous trivialization of $\LagGrb_H$ and there is a canonical isomorphism $\SB \cong \SB_{(\mathfrak{N}, \nu)}$, where $\SB_{(\mathfrak{N}, \nu)}$ is the spinor bundle constructed in Example~\ref{ExampleSpinorBdlFromTrivialization}.

The category $\textsc{Triv}(\LagGrb_{LX})$ in the domain of \eqref{TrivializationFunctor} is the category of \emph{smooth} trivializations, but we may choose a smooth structure on $\mathfrak{N}$ such that the canonical composition map \eqref{CanonicalCompositionMap} is smooth.
This follows from the fact that every continuous trivialization of a bundle gerbe is (in the category of continuous trivializations) isomorphic to a smooth trivialization.
\end{proof}

\begin{remark}
Using the language of (super) 2-vector bundles (see \S\ref{SectionStringorBundle}), the above constructions can be phrased as follows.
Both $\LagGrb_H$ and the Clifford algebra bundle $\CC$ are examples of super 2-vector bundles over the manifold $M$ and the canonical isomorphism
\begin{equation*}
  \F_2 \otimes \Pf \longrightarrow \F_1
\end{equation*}
establishes that the Fock space bundle $\F$ over $\Lag$ defines an isomorphism of super 2-vector bundles $\mathfrak{f}: \LagGrb_H \to \CC$.
In the finite-dimensional case, this is discused in \S5 of \cite{kristel20212vector}.
The inverse of any trivialization is an isomorphism $\mathcal{I} \to \LagGrb_H$ (where $\mathcal{I}$ is the trivial super 2-vector bundle) and postcomposing with $\mathfrak{f}$ gives an isomorphism of super 2-vector bundles $\mathcal{I} \to \CC$, which is the same thing as an irreducible super left $\CC$-module.
\end{remark}

\subsection{Smoothing structures}
\label{SectionSmoothingStructures}

At this stage, we cannot upgrade the statement of Thm.~\ref{ThmEquivalenceFinite} to an equivalence of categories, because the functor constructed in the proof above is not full.
The problem is here that there is no smoothness requirement for isomorphisms in the category of spinor bundles.
This can be fixed as follows.

\begin{definition}[Smoothing structure]
\label{DefinitionSmoothingStructure}
Let $\SB$ be a spinor bundle for $H$.
A \emph{smoothing structure} on $\SB$ is a choice of smooth structure for the super line bundle $\mathfrak{N}$ over $\Lag$ defined in \eqref{AssociatedSmoothnessLine}, which makes the canonical composition map \eqref{CanonicalCompositionMap}  smooth.
\end{definition}

\begin{remark}
Following \cite{KristelWaldorf2}, a \emph{rigging} of $\SB$ is a smooth $\Cl(H)$-submodule bundle $\SB^\infty \subset \SB$ which is fiberwise isomorphic to the subspace $\F^\infty_L$ of smooth vectors for the action of $\U(L)$.
Given such a rigging, inclusion of intertwiners produces a line bundle isomorphism
\begin{equation*}
	\underline{\Hom}(\pi^*\SB^\infty, \F^\infty) \cong \underline{\Hom}(\pi^*\SB, \F) = \mathfrak{N}
\end{equation*}
Since the left hand side has a natural smooth structure (as both $\pi^*\SB^\infty$ and $\F^\infty$ are smooth vector bundles), we see that any rigging provides a canonical smoothing structure on $\SB$.
\end{remark}

We may form the category $\textsc{SpinBdl}_H^{\mathrm{sm}}$ of spinor bundles for $H$ with a smoothing structure.
Morphisms in this category consist of even bundle isomorphisms $\SB \to \SB^\prime$ intertwining the $\CC$-actions that make the induced line bundle homomorphism 
\begin{equation*}
\mathfrak{N} = \underline{\Hom}(\F, \pi^*\SB) \longrightarrow \underline{\Hom}(\F, \pi^* \SB^\prime) = \mathfrak{N}^\prime
\end{equation*}
smooth.

\begin{lemma}
\label{LemmaSmoothingStructureExistence}
Every spinor bundle $\SB$ admits a smoothing structure. 
Two spinor bundles with smoothing structures are isomorphic in $\textsc{SpinBdl}_H^{\mathrm{sm}}$ if and only if they are isomorphic in $\textsc{SpinBdl}_H$.
\end{lemma}

\begin{proof}
A smoothing structure for a spinor bundle $\SB$ can be obtained by identifying it with an element in the image of the functor \eqref{TrivializationFunctor} (which has a canonical smoothing structure).
As the functor is essentially surjective, this is always possible. 
The second statement follows from the fact that \eqref{TrivializationFunctor} is faithful.
\end{proof}

The proof of the following refinement of Thm.~\ref{ThmEquivalenceFinite} is now straight forward.

\begin{theorem}
The functor \eqref{TrivializationFunctor} refines to an equivalence of categories
\begin{equation}
\label{FunctorTrivSpin}
  \mathrm{Triv}(\LagGrb_H) \longrightarrow \textsc{SpinBdl}_H^{\mathrm{sm}}.
\end{equation} 
\end{theorem}

\begin{proof}
By Lemma~\ref{LemmaSmoothingStructureExistence}, the functor is still essentially surjective.
It is faithful by Thm.~\ref{ThmEquivalenceFinite}.
It is full as the compatibility with the smoothing structures of morphisms in $\textsc{SpinBdl}_H^{\mathrm{sm}}$ ensures that each isomorphism of spinor bundles with smoothing structures in the image of \eqref{FunctorTrivSpin} comes from an isomorphism of the corresponding trivializations of $\LagGrb_H$.
\end{proof}

\subsection{Lagrangians over the circle}
\label{SectionLagrangiansOverTheCircle}

When applying the general constructions from \S\ref{SectionLagrangianGerbe} to the loop space $M = LX$, where $X$ is some (oriented) Riemannian manifold, the Hilbert space bundle should $H$ should be a suitable completion of the tangent bundle $TLX$.
The crucial task, however, is to endow its fibers with a suitable polarization, i.e., a preferred equivalence class of Lagrangians.

To analyze this problem, consider the Hilbert space 
\begin{equation*}
H_E \defeq L^2(S^1, E), 
\end{equation*}
where $E$ is a ``real'' metric vector bundle $E$ on $S^1$.
$H_E$ acquires a real structure from the pointwise real structures in the fibers of $E$.

As $H_E$ is a Hilbert space of sections, it makes sense to consider Lagrangians $L \subset H_E$ whose orthogonal projection $P_L$ is a pseudodifferential operator (necessarily of order zero), at least up to a Hilbert-Schmidt perturbation.
Two equivalent spectral projections $P_L$ and $P_{L^\prime}$ of this type must necessarily have the same principal symbol $p$, which is a section of the bundle $\pi^*\End(E)$ over $T^*S^1$.
If we insist additionally that the principal symbol $p$ is  invariant under orthogonal transformations of the fibers of $E$, then (up to sign), $p$ must necessarily be given by
\begin{equation}
\label{SymbolOfProjection}
  p(t, \vartheta) \defeq \sign(\vartheta) \cdot \mathrm{id}_{E_t}, \qquad t \in S^1, ~\vartheta \in T^*_tS^1.
\end{equation}
It therefore seems natural to consider the equivalence class
\begin{equation}
\label{LagrangiansForE}
  \Lag_E \defeq \{ L \subset H_E \text{ Lagrangian}\mid P_L \text{ has principal symbol}~p\},
\end{equation}
consisting of those Lagrangians $L$ such that the corresponding orthogonal projection $P_L$ is (up to a Hilbert-Schmidt perturbation) a classical pseudodifferential operator with principal symbol given by \eqref{SymbolOfProjection}.
If $L$ and $L^\prime$ are two such Lagrangians, the difference $P_L - P_{L^\prime}$ 
has zero principal symbol which implies that it is a Hilbert-Schmidt operator, hence $L$ and $L^\prime$ are indeed equivalent.

However, there is an index-theoretic obstruction for this approach to work.
Recall that real vector bundles $E$ over $S^1$ are characterized by two invariants: Their dimension and their first Stiefel-Whitney class $w_1(E) \in \Z_2$. 
We now have the following.

\begin{theorem}
\label{TheoremConditionsOnE}
  $\Lag_E$ is non-empty if and only if $w_1(E) = \dim E \mod 2$.
\end{theorem}

\begin{proof}
Suppose that $\Lag_E$ is non-empty.
Then given a Lagrangian $L\in \Lag_E$, the operator $J_L = i(P_L - P_{L^\perp})$ is real skew-adjoint and therefore has an index $\ind J_L = \dim \ker J_L \mod 2 \in \Z_2$, which is independent of $L$ (as the difference $J_L - J_{L^\prime}$ is Hilbert-Schmidt for $L, L^\prime \in \Lag_E$) .
However,  by \cite[Thm.~2.3]{EllipticOperatorsV}, this index can be computed in terms of the symbol $p$.
Carrying out this calculation gives
\begin{equation*}
  \ind J_L = w_1(E) + \dim E \mod 2.
\end{equation*}
On the other hand, since $L$ is a Lagrangian, we have $H = L \oplus L^\perp$, which implies $\ker J_L = 0$.
Hence we must have $w_1(E) = \dim E \mod 2$.

For the converse direction, we construct a Lagrangian $L \in \Lag_E$, assuming $w_1(E) = \dim E \mod 2$.
To this end, choose a connection $\nabla^E$ on $E$.
Then covariant differentiation with respect to the canonical coordinate on $S^1$ gives a real skew-adjoint operator $\nabla^E_t$, whose mod 2 index is $w_1(E) + \dim E \mod 2 = 0$.
This implies that its kernel is even-dimensional.
Define now
\begin{equation}
\label{LPlusMinus}
  L_\pm = \bigoplus_{ \lambda > 0} \mathrm{Eig}(D, \pm\lambda),
\end{equation}
the (Hilbert space) direct sum of eigenspaces to positive (negative) eigenvalues of the self-adjoint operator $D = i\nabla^E_t$.
Then $\overline{L}_- = L_+$ and the orthogonal projection onto $L_-$ is a pseudodifferential operator with principal symbol \eqref{SymbolOfProjection} (see Prop.~14.2 in \cite{BoosBavnbek}).
Now, for any choice of Lagrangian $K$ in the finite-dimensional ``real'' Hilbert space $\ker(D) = \ker(\nabla^E_t) \subset H_E$, the sum $L= L_- + K$ is a Lagrangian in $\Lag_E$.
Such a Lagrangian $K$ exists as $\dim\ker(D) = \ind \nabla_t^E \mod 2$, which is zero.
\end{proof}

A real vector bundle $E$ satisfying the condition \eqref{TheoremConditionsOnE} can be written in the form $E = \mathbb{S} \otimes \C^d$, where $\mathbb{S}$ is the bounding spinor bundle on the circle, i.e., the rank one ``real'' bundle with nontrivial $w_1(\mathbb{S})$ (observe here that $w_1(\mathbb{S} \otimes \C^d) = d \cdot w_1(\mathbb{S}) = d \mod 2$).
We therefore set
\begin{equation}
\label{Hnot}
  H_0 \defeq L^2(S^1, \mathbb{S} \otimes \C^d)
\end{equation}
and let $\Lag_0$ be the equivalence class of Lagrangians defined by \eqref{LagrangiansForE} for this choice of $E$, i.e., the class of Lagrangians $L \subset H_0$ such that the corresponding projection $P_L$ is, up to a Hilbert-Schmidt perturbation, a pseudodifferential operator with principal symbol \eqref{SymbolOfProjection}.
If $d$ is even, we simply have $\C^d \cong \mathbb{S} \otimes \C^d$ but if $d$ is odd, Thm.~\ref{TheoremConditionsOnE} shows that we need the twist by $\mathbb{S}$ in order for $\Lag_0$ to be non-empty.

\medskip

It will be important that the Hilbert space $H_0$ has a canonical action of the loop group $L\SO(d)$, given by pointwise multiplication in the $\C^d$ factor.
This action is by orthogonal transformations, and as conjugation by elements of $L\SO(d)$ leaves the principal symbol \eqref{SymbolOfProjection} invariant, these preserve the polarization $\Lag_0$.
Hence we obtain a group homomorphism
\begin{equation}
\label{HomLSOdOres}
L\SO(d) \longrightarrow \O_{\res}(H_0), 
\end{equation}
which turns out to be smooth (see \cite[Prop.~12.5.1]{PressleySegal}, \cite[Prop.~3.23]{KristelWaldorf1}).
Given a Lagrangian $L$, we can pull back the implementer extension $\Imp_L$ along the homomorphism \eqref{HomLSOdOres} to obtain a central extension of $L\SO(d)$.
This extension can then be pulled back further along the group homomorphism $L\Spin(d) \to L\SO(d)$, obtaining a central extension of $L\Spin(d)$.
We obtain these central extensions the \emph{implementer extension} of $L\SO(d)$, respectively $L\Spin(d)$.
\begin{equation}
\label{DiagramImplementerExtension}
\begin{tikzcd}
  \widetilde{L\Spin(d)} \ar[r, dashed] \ar[d, dashed] & \widetilde{L\SO(d)} \ar[r, dashed] \ar[d, dashed] & \Imp_L \ar[d] \\
  L \Spin(d) \ar[r] & L\SO(d) \ar[r] & \O_{\res}(H_0)
\end{tikzcd}
\end{equation}

\begin{theorem}
\label{ThmImpBasic}
{\normalfont  \cite[Theorem~3.26]{KristelWaldorf1}} ~
If $d \geq 5$, then the implementer extension is a \emph{basic} central extension of $L \Spin(d)$, meaning that its first Chern class is a generator for $H^2(L\Spin(d), \Z) \cong \Z$.
\end{theorem}

This theorem is either proved by using that the first Chern class of $\Imp_L$ is a generator for $H^2(\O_{\res}(H), \Z) \cong \Z$ \cite[Prop.~1.2]{SperaWurzbacherSpinors} together with the fact that the map $\Omega \Spin(d) \to \O_{\res}(H_0)$ is $(d-2)$-connected \cite[Prop.~12.5.2]{PressleySegal}.
Alternatively, one can explicitly compute the group cocycle characterizing the extension (see \cite[Lemma~3.24]{KristelWaldorf1} or \cite[Thm.~8.3]{Borthwick}) and use \cite[Theorem 4.4.1 (iv) \& Proposition 4.4.6]{PressleySegal}.


\subsection{The loop space spinor bundle}
\label{SectionApplicationLoopSpace}

Let now $X$ be an oriented Riemannian manifold of dimension $d$ with loop space $LX = C^\infty(S^1, X)$.
The loop space has a canonical bundle of ``real'' Hilbert spaces $H_{LX}$, the fiber of which at a loop $\gamma \in LX$ is
\begin{equation}
\label{Hgamma}
 H_\gamma \defeq L^2(S^1, \mathbb{S} \otimes \gamma^*TX),
\end{equation}
with $\mathbb{S}$ the bounding spinor bundle on the circle.
The structure group of $H_{LX}$ is canonically reduced to $L\SO(d)$, in the sense that the looped frame bundle $L\SO(X)$ is a principal $L\SO(d)$-bundle with the property that the associated bundle $L\SO(X) \times_{L\SO(d)} H_0$ is canonically isomorphic to $H_{LX}$, where $H_0$ is the Hilbert space considered in \eqref{Hnot}.

As $X$ is oriented, we have $w_1(\gamma^*TX) = \gamma^* w_1(TX) = 0$, hence the canonical polarization 
\begin{equation}
\label{LagGamma}
  \Lag_\gamma \defeq \{ L \subset H_\gamma \text{ Lagrangian} \mid P_L \text{ has principal symbol } p\}
\end{equation}
of $H_\gamma$ described in \S\ref{SectionLagrangiansOverTheCircle} is non-empty.
We now provide $H_{LX}$ with the structure of a smooth bundle of polarized Hilbert spaces according to Definition~\ref{DefinitionPolarizedHilbertSpace}.
To this end, we remark that elements of the bundle $\SO(X)$ of orthogonal oriented frames can be viewed as oriented orthogonal transformations $\R^d \to T_xX$, hence by pointwise application, elements $g \in L\SO(X)$ of the looped frame bundle that lift a loop $\gamma \in LX$ give orthogonal transformations $g: H_0 \to H_\gamma$.
As such a $g$ intertwines the principal symbols \eqref{SymbolOfProjection} for the bundles $E = \mathbb{S} \otimes \C^d$ and $\mathbb{S} \otimes \gamma^*TX$, it is a restricted orthogonal transformation.
Therefore, a local section $g$ of $L\SO(X)$ over an open set $O \subseteq LX$ provides a local trivialization of $H_{LX}$.
If $g^\prime$ is another local section over an open set $O^\prime \subseteq LX$, then the transition function is obtained from the smooth function $g^\prime g^{-1} : O \cap O^\prime \longrightarrow L\SO(d)$ by postcomposing with the Lie group homomorphism \eqref{HomLSOdOres}.
Hence the transition functions are smooth.

As in \S\ref{SectionLagrangianGerbe}, the fiberwise completions $\CC_\gamma$ of the Clifford algebra $\Cl(H_\gamma)$ with respect to $\Lag_\gamma$ glue together to a continuous bundle $\CC$ of super von Neumann algebras over $LX$ (this bundle was previously considered in \cite{LudewigClifford}).
We repeat Definition~\ref{DefinitionAbstractSpinorBundle} for the present context.

\begin{definition}[Loop space spinor bundle]
\label{DefinitionLoopSpaceSpinorBundle}
A \emph{loop space spinor bundle} is a (continuous) bundle $\SB$ over $LX$ of irreducible super left modules for the super von Neumann algebra bundle $\CC$.
\end{definition}

The fiberwise polarizations $\Lag_\gamma$ glue together to a smooth fiber bundle $\Lag_{LX} \to LX$.
Over $\Lag$, we have the Fock bundle $\F$, a (continuous) bundle of super left modules for the (pullback to $\Lag_{LX}$ of the) Clifford algebra bundle $\Cl(H_{LX})$ and its completion $\CC$.

Over the 2-fold fiber product $\Lag^{[2]}$, we have the Pfaffian line bundle $\Pf$, which is the defining super line bundle for the Lagrangian gerbe $\LagGrb_{LX}$  for $H_{LX}$ (see Definition~\ref{DefinitionLagrangianGerbe}).
As seen in \S\ref{SectionLagrangianGerbe}, non-triviality of the Lagrangian gerbe is the obstruction to the existence of a loop space spinor bundle.
This statement is repeated in the following theorem.

\begin{theorem}
A loop space spinor bundle on $LX$ exists if and only if the Lagrangian gerbe $\LagGrb_{LX}$ admits a trivialization.
\end{theorem}

In the present situation, non-triviality of the Lagrangian gerbe $\LagGrb_{LX}$ can be characterized in geometric terms as follows.

\begin{theorem}
\label{ThmOrientationLoopSpace}
The structure group of $L\SO(X)$ can be reduced to the identity component $L\SO(d)_0 \subset L\SO(d)$ if and only if the orientation line bundle $\orclass(\LagGrb_{LX})$ is non-trivial.
\end{theorem}

Recall here that every super bundle 2-gerbe $\mathcal{G}$ has an associated principal $\Z_2$-bundle $\orclass(\mathcal{G})$, see Definition~\ref{DefinitionOrientationBundle}.

A sufficient condition for the reduction of structure groups from Thm.~\ref{ThmOrientationLoopSpace} to exist is the existence of a spin structure $\Spin(X)$ of $X$.
Indeed, the image $P$ of the map $L\Spin(X) \to L\SO(X)$ is a principal $L\SO(d)_0$-bundle.
Conversely, it was shown by McLaughlin \cite[Prop.~2.1]{McLaughlin} that in the case that $X$ has dimension $d \geq 4$ and is simply connected, the existence of a reduction of the structure group to $L\SO(d)_0$ implies that $X$ is spin.

\begin{proof}
Let $P \subset L\SO(X)$ be a reduction of the structure group to $L\SO(d)_0$, i.e., $P$ is a subbundle on which the subgroup $L\SO(d)_0$ acts fiberwise freely and transitively.
Fix some Lagrangian $L \in \Lag_0$ and consider the map of covers $\rho : P \to \Lag$, $g \mapsto gL$.
Since the fibers of $P$ are connected, we obtain that the pullback $\rho^*\Pf$ is purely even (see Remark~\ref{RemarkParityPf}).
This implies the orientation bundle of the refinement of $\LagGrb_{LX}$ along $\rho$ is trivial.

Conversely, suppose that the orientation bundle $\orclass(\LagGrb_{LX})$ is trivial and choose a section $s$ of $\orclass(\LagGrb_{LX})$ and a Lagrangian $L \in \Lag_0$.
Then
\begin{equation*}
 P \defeq \bigl\{ g \in L\SO(X)_\gamma \mid s(\gamma) = [gL, 0] \bigr\} \subset L\SO(X).
\end{equation*}
is a principal $L\SO(d)_0$-bundle.
\end{proof}

For the next result, recall that a central extension $\widetilde{L\Spin(d)}$ of $L\Spin(d)$ by $\U(1)$ is \emph{basic} if its first Chern class is a generator of $H^2(L\Spin(d), \Z) \cong \Z$.

\begin{theorem}
\label{ThmLagLifting}
Suppose that $X$ is a spin manifold of dimension $d \geq 5$.
Then $\LagGrb_{LX}$ admits a trivialization if and only if there exists a lift of the structure group of $LX$ from $L\Spin(d)$ to its basic central extension.
\end{theorem}

\begin{proof}
Suppose that $X$ is spin and let $\Spin(X)$ be the corresponding principal $\Spin(d)$-bundle.
For a given Lagrangian $L \in \Lag_0$, consider the map $\rho : L\Spin(X) \to \Lag$ given by $\tilde{g} \mapsto gL$, where $g$ is the image of $\tilde{g}$ in $L\SO(X)$ under the loop map of the double covering map $\Spin(X) \to \SO(X)$.
By Thm.~\ref{ThmOrientationLoopSpace}, $\LagGrb_{LX}$ has trivial orientation bundle, and refining $\LagGrb_{LX}$ along the map $\rho$ gives a new purely even super bundle gerbe $\mathcal{G}_L$ isomorphic to $\LagGrb_{LX}$ with cover $L\Spin(X)$.

We claim that the underlying $\U(1)$-bundle gerbe of $\mathcal{G}_L$ is strictly isomorphic to the lifting gerbe $\Lift_{\Spin(X)}$ for the implementer extension \eqref{DiagramImplementerExtension} of $L\Spin(d)$; compare \eqref{EqLiftingGerbe}.
Recall here that the Pfaffian line bundle has a canonical metric given by \eqref{PfaffianInnerProduct} and hence a canonical underlying $\U(1)$-bundle gerbe, given by passing to the norm one (i.e., unitary) elements.
Now, the line bundle of $\mathcal{G}_L$ at $(\tilde{g}_1, \tilde{g}_2) \in L\Spin(X)^{[2]}$ is $\Pf(g_1L, g_2L)$.
Using \eqref{IsoImpLPfL} and \eqref{EquivariancePfaffianLines}, we obtain a map
\begin{equation}
\label{ImpVsPfaffian}
  (\Imp_L){ g_2^*g_1} \longrightarrow \Pf(g_2^*g_1 L, L) \longrightarrow \Pf(g_1 L, g_2 L).
\end{equation}
This gives an isomorphism of the subbundle of $\rho^*\Pf$ consisting of unitary intertwiners with the pullback of the principal $\U(1)$-bundle $\Imp_L \to \O_{\res}(H_0)$ along the map
\begin{equation*}
  L\Spin(X)^{[2]} \stackrel{\delta}{\longrightarrow} L\Spin(d) \longrightarrow L\SO(d) \longrightarrow \O_{\res}(H_0), \qquad (\tilde{g}_1, \tilde{g}_2) \longmapsto g_2^*g_1.
\end{equation*}
 But with a view on \eqref{DiagramImplementerExtension}, this pullback is just $\delta^* \widetilde{L\Spin(d)}$.
 The isomorphism \eqref{ImpVsPfaffian} moreover intertwines the composition map \eqref{CompositionMap} with group multiplication in $\Imp_L$, so that $\mathcal{G}_L$ is indeed isomorphic to $\Lift_{\Spin(X)}$.
 
 By Thm.~\ref{ThmLiftingGerbe}, we therefore obtain that $\mathcal{G}_L$ and hence also $\LagGrb_{LX}$ are trivial if and only if the structure group of $LX$ admits a lift from $L\Spin(d)$ to the implementer extension.
On the other hand, by Thm.~\ref{ThmImpBasic} the assumption $d \geq 5$ guarantees that the implementer extension is $L\Spin(d)$.
\end{proof}

Recall that on a spin manifold $X$, there exists a characteristic class $\frac{1}{2}p_1(X)$, which has the property that twice this class equals the first Pontrjagin class \cite[Lemma~2.2]{McLaughlin}.
It was shown by Waldorf \cite[\S5.2]{WaldorfSpinString} that the Dixmier-Douady class of the obstruction gerbe to lifting the structure group $L\Spin(d)$ to its basic central extension is $\tau(\frac{1}{2}p_1(X))$.
Conversely, McLaughlin showed that if $X$ is 2-connected and $\dim(X) \geq 5$, then vanishing of $\tau(\frac{1}{2}p_1(X))$ implies the vanishing of $\frac{1}{2}p_1(X)$ itself.
Examples of manifolds $X$ with $\frac{1}{2}p_1(X) \neq 0$ but $\tau(\frac{1}{2}p_1(X)) = 0$ have been provided in \cite{PilchWarner}.

\section{The fusion product}
\label{SectionFusionProduct}

Let $X$ be a Riemannian manifold of dimension $d$.
In this section, we discuss \emph{fusion products} on the loop space spinor bundle.
For this, it will be important to cup loops into paths.
We write 
\begin{equation*}
PX \subset C^\infty([0, \pi], X)
\end{equation*}
 for the subset of those paths $\gamma : [0, \pi] \to X$ that are \emph{flat} near the end points, i.e., in some coordinate chart (hence all coordinate charts) the derivatives of $\gamma$ vanish to all orders at the end points.
$PX$ is an infinite-dimensional Fr\'echet manifold.

\begin{remark}
In the literature, two variations of this path space are often considered: First, one could drop the condition that the paths be flat at the end points, which also gives a nice Fr\'echet manifold of paths but has the disadvantage that after gluing two paths with the same end points to a loop with the map \eqref{CupMap} below, one does not obtain a smooth loop.
This can be repaired by considering paths with \emph{sitting instants} instead, i.e., paths that are constant near the end points.
But this has the disadvantage that one leaves the realm of manifolds and has to work with diffeological spaces instead.
\end{remark}

For $k \geq 2$, denote by $PX^{[k]}$ the $k$-fold fiber product of $PX$ with itself over the endpoint evaluation map to $X \times X$.
Explicitly, elements of $PX^{[k]}$ consist of tuples $(\gamma_1, \dots, \gamma_k)$ of paths $\gamma_i \in PX$ that each share both the same starting point and the same end point.

For $(\gamma_1, \gamma_2) \in PX^{[2]}$, we set
\begin{equation} \label{CupMap}
  (\gamma_1 \fuse \gamma_2)(t) = \begin{cases} \gamma_2(t) & t \in [0, \pi] \\ \gamma_1(2\pi - t) & t \in [\pi, 2\pi].\end{cases}
\end{equation}
Since $\gamma_1$ and $\gamma_2$ are flat at the endpoints, $\gamma_1 \fuse \gamma_2$ is a smooth loop in $X$, and sending $(\gamma_1, \gamma_2)$ to $\gamma_1 \fuse \gamma_2$ yields a smooth inclusion map $PX^{[2]} \to LX$.
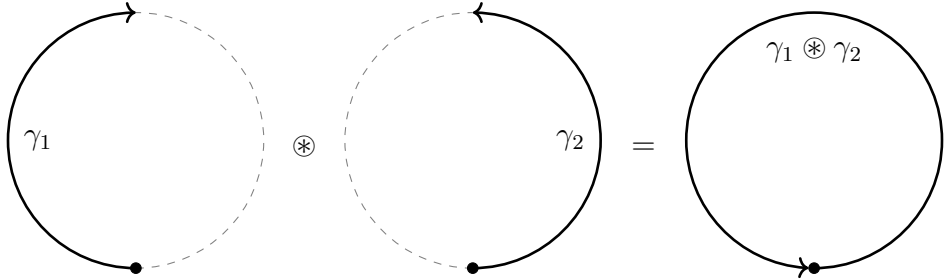
\begin{equation*}
\begin{aligned}
\begin{tikzpicture}


 \node at  (-1.3,0) {$\gamma_1$};
  \draw[->, line width=1pt] (0,-1.7) arc (-90:-270:1.7);
  \draw[dashed, color=gray] (0,-1.7) arc (-90:90:1.7);
  
  \coordinate (top) at (0,1.7);
    \node[above] at (top) {};

  \coordinate (bottom) at (0,-1.7);
  \filldraw[black] (bottom) circle (2pt);
  \node[below] at (bottom) {};

%

\end{tikzpicture}
\end{aligned}
~~\fuse~~
\begin{aligned}
\begin{tikzpicture}


 \node at  (1.3,0) {$\gamma_2$};
  \draw[->, line width=1pt] (0,-1.7) arc (-90:90:1.7);
  \draw[dashed, color=gray] (0,-1.7) arc (-90:-270:1.7);
  
  \coordinate (top) at (0,1.7);
    \node[above] at (top) {};

  \coordinate (bottom) at (0,-1.7);
  \filldraw[black] (bottom) circle (2pt);
  \node[below] at (bottom) {};

%

\end{tikzpicture}
\end{aligned}
~~= ~~
\begin{aligned}
\begin{tikzpicture}


 \node at  (0, 1.2) {$\gamma_1 \fuse \gamma_2$};
  \draw[->, line width=1pt] (0,-1.7) arc (-90:267.5:1.7);
  
  \coordinate (top) at (0,1.7);
    \node[above] at (top) {};

  \coordinate (bottom) at (0,-1.7);
  \filldraw[black] (bottom) circle (2pt);
  \node[below] at (bottom) {};

%

\end{tikzpicture}
\end{aligned}
\end{equation*}

\subsection{The von Neumann algebra bundle over the path space}
\label{SectionIntervalAlgebra}

In this section, we construct a bundle $\A$ of Clifford von Neumann algebras over the path space $PX$ of a Riemannian manifold $X$.
We first discuss the typical fiber of this bundle.
Consider the ``real'' Hilbert space
\begin{equation}
\label{Vnot}
V_0 \defeq L^2([0, \pi], \C^d).
\end{equation}
By choosing a metric preserving trivialization of the restriction $\mathbb{S}|_{[0, \pi]}$, we identify $V_0$ with a subspace of the ``real'' Hilbert space $H_0$ defined in \eqref{Hnot}.
We fix such a trivialization once and for all.
This induces an inclusion $*$-algebras $\Cl(V_0) \hookrightarrow \Cl(H_0) \subset \CC_0$, where $\CC_0$ is the completion of $\Cl(H_0)$ with respect to the polarization $\Lag_0$; see \eqref{LagrangiansForE}.
We define 
\begin{equation}
\label{DefinitionA0}
  \A_0 \defeq \Cl(V_0)^{\prime\prime} \subset \CC_0,
\end{equation}
as the completion of $\Cl(V_0)$ with respect to the ultraweak topology induced by this inclusion.

\begin{remark}
Observe that $\mathbb{S}|_{[0, \pi]}$ has precisely two trivializations (that preserve the fiber metrics), which differ by a sign. 
Hence the two inclusions $\Cl(V_0) \hookrightarrow \CC_0$ corresponding to these trivializations are intertwined by the grading automorphism of $\CC_0$.
We therefore see that both inclusion induce the same topology on $\Cl(V_0)$ and hence lead to the same completion $\A_0$.
\end{remark}

\begin{lemma}
\label{LemmaActionOnA0Cont}
The Bogoliubov automorphisms of $\Cl(V_0)$ induced by elements of $P\SO(d)$ extend to $\A_0$ by ultraweak continuity and the resulting group homomorphism
\begin{equation}
\label{PSOdToA0}
 P\SO(d) \longrightarrow \A_0
\end{equation} 
is continuous.
\end{lemma}

\begin{proof}
  Let $g \in P\SO(d)$, which induces a Bogoliubov automorphism $\Cl_g$ of $\Cl(V_0)$.
  To see that $\Cl_g$ is continuous, choose an element $g^\prime \in P\SO(d)$ with the same start and end point as $g$. 
  Then the element $g^\prime \fuse g \in \O_{\res}(H_0)$ induces a $*$-automorphism of $\CC_0$ restricting to $\Cl_g$ on $\Cl(V_0)$.
  So $\Cl_g$ is the restriction of a $*$-automorphism of $\CC_0$, hence ultraweakly continuous, therefore extends by continuity to $\A_0$.
  
  To see that \eqref{PSOdToA0} is continuous, we first observe that the composition 
  \begin{equation*}
  P\SO(d)^{[2]} \stackrel{\fuse}{\longrightarrow} L\SO(d) \longrightarrow \O_{\res}(H_0) \longrightarrow \Aut(\CC_0)
  \end{equation*}
  is continuous, where the middle one is \eqref{HomLSOdOres} and the last one is \eqref{GroupHomOresAutC}.
  Notice that the image of the composition is contained in the subgroup $\Aut(\CC_0)^\prime\subset\Aut(\CC_0)$ that consist of automorphisms of $\CC_0$ that preserve $\A_0$, hence we obtain a homomorphism $P\SO(d)^{[2]} \to \Aut^\prime(\CC_0)$.
  Precomposition with the diagonal map $P\SO(d) \to P\SO(d)^{[2]}$ and postcomposition with the map that restricts an automorphism of $\CC_0$ that preserves $\A_0$ to an automorphism of $\A_0$ gives a factorization of the homomorphism as the composition $P\SO(d) \to P\SO(d)^{[2]} \to \Aut(\CC_0)^\prime \to \Aut(\A_0)$ of continuous group homomorphisms.
\end{proof}

We now construct the desired bundle $\A$ with typical fiber $\A_0$ over the path space $PX$ of a Riemannian manifold $X$.
For $\gamma \in PX$, consider the Hilbert space
\begin{equation}
\label{Vgamma}
  V_\gamma \defeq L^2([0, \pi], \gamma^*TX).
\end{equation}
For $(\gamma_1, \gamma_2) \in PX^{[2]}$ we can form the loop $\gamma_1 \fuse \gamma_2$, and $V_{\gamma_2}$ is naturally a subspace of the Hilbert space $H_{\gamma_1 \fuse \gamma_2}$ defined in \eqref{Hgamma}.
$H_{\gamma_1 \fuse \gamma_2}$ supports the equivalence class $\Lag_{\gamma_1 \fuse \gamma_2}$ of Lagrangians defined in \eqref{LagGamma}, which gives rise to the  completion $\CC_{\gamma_1 \fuse \gamma_2}$ of the corresponding Clifford algebra.
We therefore have inclusions
\begin{equation}
\label{InclusionPathClifford}
  \Cl(V_{\gamma_2}) \subset \Cl(H_{\gamma_1 \fuse \gamma_2}) \subset \CC_{\gamma_1 \fuse \gamma_2}.
\end{equation}
To get rid of the dependence on $\gamma_1$, we use the following lemma.

\begin{lemma}
The ultraweak topology induced on $\Cl(V_{\gamma_2})$ by the inclusion \eqref{InclusionPathClifford} does not depend on the choice of $\gamma_1$.
\end{lemma}

\begin{proof}
Let $\gamma_1^\prime \in PX$ be another path with the same start and end points as $\gamma_2$.
Choose lifts $g \in L\SO(X)_{\gamma_1\fuse\gamma_2}$ and $g^\prime \in L\SO(X)_{\gamma_1^\prime \fuse\gamma_2}$ that agree on $[0, \pi]$ (i.e, over $\gamma_2$). We view these as orthogonal transformations 
\begin{equation*}
  g: H_0 \longrightarrow H_{\gamma_1\fuse\gamma_2}, \qquad g^\prime : H_0 \longrightarrow H_{\gamma_1^\prime \fuse\gamma_2}.
\end{equation*}
of polarized ``real'' Hilbert spaces.
The composition $g^\prime g^{-1} : H_{\gamma_1\fuse\gamma_2} \to H_{\gamma_1^\prime \fuse\gamma_2}$ is then also an orthogonal transformation of polarized ``real'' Hilbert spaces, hence induces an isomorphism $\Cl_{g^\prime g^{-1}}$ on the corresponding Clifford algebras which by construction intertwines the inclusions of $\Cl(V_{\gamma_2})$.
\begin{equation*}
\begin{tikzcd}[row sep=0.2cm]
  & \Cl(H_{\gamma_1\fuse\gamma_2}) \ar[r, hookrightarrow] \ar[dd, "\Cl_{g^\prime g^{-1}}"]& \CC_{\gamma_1\fuse\gamma_2} \ar[dd, "\Cl_{g^\prime g^{-1}}"] \\
   \Cl(V_{\gamma_2}) \ar[ur, hookrightarrow] \ar[dr, hookrightarrow] & & \\
  & \Cl(H_{\gamma_1^\prime\fuse\gamma_2}) \ar[r, hookrightarrow] & \CC_{\gamma_1^\prime\fuse\gamma_2} 
\end{tikzcd}
\end{equation*}
On the other hand, $\Cl_{g^\prime g^{-1}}$ extends by ultraweak continuity to the von Neumann algebra completions, which shows that both inclusions induce the same ultraweak topology.
\end{proof}

\begin{notation}
\label{NotationAgamma}
  For $\gamma \in PX$, we denote by $\A_\gamma$ the completion of $\Cl(V_\gamma)$ with respect to the ultraweak topology induced by any of the inclusions \eqref{InclusionPathClifford}.
\end{notation}

Each $\A_\gamma$, $\gamma \in PX$, is a super von Neumann algebra, and by construction, for any $(\gamma_1, \gamma_2) \in PX^{[2]}$, we have $*$-homomorphisms 
\begin{equation*}
\A_{\gamma_2} \hookrightarrow \CC_{\gamma_1\fuse\gamma_2}.
\end{equation*}
The $\A_\gamma$, $\gamma \in PX$, glue together to a bundle of von Neumann algebras (see Definition~\ref{DefinitionvNBundle}).
Indeed, if $g : O \to P\SO(X)$ is a local section of $P\SO(X)$ over an open set $O \subseteq PX$, we obtain a local trivialization $\Cl_g : \A|_O \longrightarrow \A_0 \times O$.
That the transition functions are continuous follows from Lemma~\ref{LemmaActionOnA0Cont}.

\subsection{Fock spaces as bimodules}
\label{SectionAbstractConnesFusion}

Let $H$ be a ``real'' Hilbert space together with a ``real'' subspace $V \subset H$, i.e., a complex subspace of $H$ that is preserved by the real structure.
Suppose now that there exists an orthogonal involution
\begin{equation*}
\sigma : H \to H
\end{equation*}
that exchanges $V$ with $V^\perp$.
The composition $i\sigma$ can be viewed as an orthogonal transformation $-H \to H$, hence we obtain an induced $*$-isomorphism 
\begin{equation}
\label{StarIsoClisigma}
\begin{tikzcd}
\Cl(V)^{\op} \cong \Cl(-H) \ar[r, "\Cl_{i\sigma}"] &
 \Cl(H),
 \end{tikzcd}
\end{equation}
 which sends $\Cl(-V)$ to $\Cl(V^\perp)$, the super commutant of $\Cl(V)$ inside $\Cl(H)$.
 
Given a Lagrangian $L \subset \Lag$, the corresponding Fock space $\F_L$ is a left super $\Cl(H)$-module via the Fock representation $\pi_L$. 
It also has an action of $\Cl(V) \subset \Cl(H)$ obtained from restricting $\pi_L$.
But the $*$-isomorphism \eqref{StarIsoClisigma} determined by $\sigma$ also induces a $*$-representation $\pi^\sigma_L$ of $\Cl(V)^{\op}$ on $\F_L$, 
explicitly given by
\begin{equation}
\label{RepresentationPiSigma}
  \pi_L^\sigma(v) \defeq \pi_L(i \sigma v), \qquad v \in V.
\end{equation}
As a Hilbert space with super commuting actions of $\Cl(V)$ and $\Cl(V)^{\op}$, the Fock space $\F_L$ acquires the structure of a super $\Cl(V)$-$\Cl(V)$-bimodule, where the right action is given on $V \subset \Cl(V)$ by
\begin{equation}
\label{DefinitionRightActionSigma}
  \xi \ract v  \defeq (-1)^{|\xi|} \pi_L^\sigma(v) \xi, \qquad v \in V.
\end{equation}

\medskip

We now lift the above construction to the von Neumann algebra completions.
Suppose we are given a polarization $\Lag$ on $H$ and let $\CC$ be the  completion of the Clifford algebra $\Cl(H)$ with respect to $\Lag$ (see Definition~\ref{DefinitionCompletionClifford}).
Let moreover
\begin{equation*}
\A = \Cl(V)^{\prime\prime} \subset \CC
\end{equation*}
be the ultraweak closure of the Clifford subalgebra $\Cl(V) \subset \Cl(H)$ in $\CC$.
We assume additionally that $\sigma$ sends the equivalence class $\Lag$ to the opposite equivalence class $\Lag^\perp$; see \eqref{ConjugateClassOfLagrangians}.
Then $\Cl_{i \sigma}$ extends by ultraweak continuity to a $*$-isomorphism $\CC_- \to \CC$, where $\CC_-$ denotes the  completion of $\Cl(-H)$ with respect to $\Lag^\perp$.
By Lemma~\ref{LemmaCanonicalIsomorphismOpposite}, also the canonical isomorphism $\Cl(H)^{\op} \cong \Cl(-H)$ extends continuously to a $*$-isomorphism $\CC^{\op} \cong \CC_-$.
Combining these two extensions, we obtain that \eqref{StarIsoClisigma} extends by ultraweak continuity to a $*$-isomorphism
\begin{equation}
\label{IsoWithOp}
\begin{tikzcd}
  \CC^{\op} \cong \CC_- \ar[r, "\Cl_{i\sigma}"] & \CC.
\end{tikzcd}
\end{equation}
As \eqref{StarIsoClisigma} sends $\Cl(V)^{\op}$ to $\Cl(V^\perp)$, the super commutant of $\Cl(V)$ in $\Cl(H)$, its continuous extension sends $\A^{\op}$ to the super commutant of $\A$ in $\CC$.

By the above, for any Lagrangian $L \in \Lag$,  the representations $\pi_L$ and $\pi_L^\sigma$ extend to supercommuting $*$-representations of $\A$, respectively $\A^{\op}$, on $\F_L$.
Turning the left $\A^{\op}$-action into a right $\A$-action using \eqref{DefinitionRightActionSigma} then turns the Fock space $\F_L$ into a super $\A$-$\A$-bimodule.
It is a non-trivial fact that the image of $\A^{\op}$ under $\pi_L^\sigma$ is actually \emph{equal} to the super commutant of $\pi_L(\A)$; this is the so-called ``twisted duality'' of the Clifford algebra, see \cite{TwistedDuality} or \cite[Thm.~13(iii)]{Wassermann}.
We obtain that the images of $\A$ under $\pi_L$ and $\pi_L^{\sigma}$ are each other's super commutant; 
this implies that the super $\A$-$\A$-bimodule $\F_L$ is invertible with respect to Connes fusion \cite[Prop.~5.5]{LandsmannBicategories}).

\begin{remark}
\label{RemarkDifferentSpaces}
The above constructions extend to the slightly more general setup.
Let $\tilde{V} \subset \tilde{H}$ be another subspace of a polarized Hilbert space $(\tilde{H}, \tilde{\Lag})$ and let $\tilde{\A} = \Cl(\tilde{V})^{\prime\prime} \subset \tilde{\CC}$ be the completion of the corresponding Clifford algebra in the  completion $\tilde{\CC}$ of $\Cl(\tilde{H})$ with respect to $\tilde{\Lag}$.
Then given an orthogonal transformation $\sigma : \tilde{H} \to H$ sending $\tilde{V}$ to $V^\perp$ and $\tilde{\Lag}{}^\perp$ to $\Lag$, the $*$-isomorphism 
\begin{equation*}
\begin{tikzcd}
\Cl(\tilde{H})^{\op} \cong \Cl(-\tilde{H}) \ar[r, "\Cl_{i \sigma}"] & \Cl(H)
\end{tikzcd}
\end{equation*}
 extends by ultraweak continuity to a $*$-isomorphism $\tilde{\CC}^{\op} \to \CC$ that takes $\tilde{\A}^{\op}$ to the super commutant of $\A$ in $\CC$.
 
Moreover, for any $L \in \Lag$, the representation $\pi^\sigma_L$ of $\Cl(\tilde{V})$ given by the formula \eqref{RepresentationPiSigma} extends continuously to a $*$-representation of $\tilde{\A}^{\op}$ on $\F_L$ such that $\pi_L^\sigma(\tilde{\A})$ is the super commutant of $\pi_L(\A)$.
This turns $\F_L$ into an invertible super $\A$-$\tilde{\A}$-bimodule.
\end{remark}

\medskip

Any Lagrangian $L$ in \emph{general position} with respect to $V$ (meaning that the intersections $L \cap V$ and $L \cap V^\perp$ are trivial) determines an orthogonal involution $\sigma = \sigma_L$ as above.
Namely, since $L$ is in general position, we have 
\begin{equation*}
L = \graph(T_L)
\end{equation*}
 for some (possibly unbounded) densely defined, invertible operator $T_L : V \to V^\perp$.
That $L$ is a Lagrangian entails that $T_L^* = -\overline{T}_L^{-1}$.
%
%
This relation implies that the operator $\sigma_L$, defined with respect to the direct sum decomposition $H = V \oplus V^\perp$ by the matrix
\begin{equation}
\label{DefinitionSigmaL}
  \sigma_L = 
  \begin{pmatrix} 
  0 & -i(T_L^*T_L)^{-1/2} T_L^* \\ 
  iT_L (T_L^*T_L)^{-1/2} 
  \end{pmatrix},
\end{equation}
is an orthogonal involution of $H$ that exchanges $V$ with $V^\perp$ and sends $L$ to $\overline{L}$ (and hence, in particular, sends $\Lag$ to $\Lag^\perp$).
%
%

\begin{theorem}
\label{TheoremFLStandard}
Let $L$ be a Lagrangian in general position and equip the Fock space $\F_L$ with the right action \eqref{DefinitionRightActionSigma} determined by the involution $\sigma_L$ given by \eqref{DefinitionSigmaL}.
Then as a super $\A$-$\A$-bimodule, $\F_L$ is isomorphic to the standard bimodule $L^2(\A)$, via a grading preserving unitary transformation.
\end{theorem}

\begin{proof}
By Lemma~3.3 and Prop.~3.4 in \cite{TwistedDuality}, the vacuum vector $\Omega \in \F_L$ is a cyclic and separating vector for $\pi_L(\A)$.
By Remark~\ref{RemarkCyclicVector}, there exists a canonical grading preserving unitary isomorphism $L^2(\A) \cong \F_L$, which intertwines the left $\A$-action on $L^2(\A)$ with $\pi_L$ and the right $\A$-action on $L^2(\A)$ with the right action on $\F_L$ given by
\begin{equation}
\label{OtherRightAction}
  \xi \ract a = J_\Omega \pi_L(a)^*J_\Omega \xi,
\end{equation}
where $J_\Omega$ the modular conjugation determined by $\Omega$ (see Remark~\ref{RemarkCyclicVector}).
By Thm.~5.6 of \cite{TwistedDuality}, the modular conjugation is given in this situation by
\begin{equation}
  J_\Omega = \mathrm{k}^* \Lambda_{\Gamma \sigma_L}, \qquad \text{where} \qquad \mathrm{k}\xi = \begin{cases} \xi & |\xi| \text{ even} \\ i\xi & |\xi| \text{ odd} \end{cases}
\end{equation}
is the Klein transformation and $\Gamma$ is the real structure of $H$.
%
%
With a view on \ref{RelationFockRep}, we find that on elements of $v \in V \subset \A$, the right action \eqref{OtherRightAction} is given by
\begin{equation*}
  \xi \ract v = (-1)^{|\xi|} \pi_L(i\sigma_L v),
\end{equation*}
which coincides with the right action \eqref{DefinitionRightActionSigma} given by $\sigma_L$.
%
%
As the right action is determined by $V \subset \A$, this shows that the two right actions agree.
We obtain that the grading preserving isomorphism $L^2(\F) \cong \F_L$ is in fact a bimodule homomorphism.
\end{proof}

We see that for any Lagrangian $L \in \Lag$ in general position with respect to $V$, there is even a \emph{canonical} unitary isomorphism $\F_L \cong L^2(\A)$, namely the canonical isomorphism determined by the vacuum vector $\Omega$, see Remark~\ref{RemarkCyclicVector}.
It is a bimodule isomorphism if $\F_L$ carries the right action determined by the involution $\sigma_L$ 
For general Lagrangians, we have the following result.

\begin{corollary}
\label{CorollaryBimoduleIso}
Let $\sigma$ be an orthogonal involution of $H$ such that $\sigma = \sigma_{L}$ for some Lagrangian ${L} \in \Lag$ in general position with respect to $V$.
Then for any other Lagrangian $L^\prime \in \Lag$, the $\A$-$\A$-bimodule $\F_{L^\prime}$, with the right action induced by $\sigma$, is  isomorphic to the standard bimodule $L^2(\A)$, via a unitary intertwiner that is either grading preserving or grading reversing.
\end{corollary}

\begin{proof}
As $L^\prime$ and $L$ are equivalent, there exists a unitary intertwiner $U : \F_{L^\prime} \to \F_{L}$ for the left action of $\Cl(H)$ (equivalently $\CC$) that is either grading preserving or grading reversing, depending on the parity of $\dim(L^\prime \cap \overline{L})$.
Our convention \eqref{IntertwiningRelation} is such that any such intertwiner is at the same time an isomorphism of super $\A$-$\A$-bimodules.
%
%
But by Thm.~\ref{TheoremFLStandard}, $\F_{L}$ is isomorphic to the standard bimodule.
\end{proof}

We now apply the above general results to the circle Hilbert space $H_0$, given in \eqref{Hnot}, with the equivalence class $\Lag_0$ as in \eqref{LagrangiansForE}.
Let $V_0$ be the ``real'' Hilbert space given in \eqref{Vnot}, viewed as a subspace of $H_0$ using the fixed trivialization of $\mathbb{S}|_{[0, \pi]}$.

\begin{definition}[Spin involution]
\label{DefinitionSpinInvolution}
A \emph{spin involution}  is a ``real'' metric-preserving bundle automorphism $s$ of $\mathbb{S}$ covering the flip diffeomorphism $t \mapsto -t$ of $S^1$.
\end{definition}

A spin involution exists, because the pullback of $\mathbb{S}$ under the flip is isomorphic to $\mathbb{S}$ (notice both $\mathbb{S}$ and its pullback have the same dimension and Stiefel-Whitney class), and is unique up to sign.
As the map $t \mapsto -t$ is orientation-reversing, $s$ sends the principal symbol \eqref{SymbolOfProjection} to its negative; hence a spin involution $s$ of $H_0$ exchanges $\Lag_0$ with its opposite class $\Lag_0^\perp$.

A spin involution $s$ induces an orthogonal involution $\sigma$ of $H_0$, given by 
\begin{equation}
\label{InducedInvolution}
(\sigma f)(t) = (s \otimes \mathrm{id}_{\C^d})f(-t).
\end{equation}

\begin{theorem}
\label{ThmSpinvolution}
There exists a Lagrangian $L \in \Lag_0$ in general position with respect to $V_0$ such that the orthogonal involution $\sigma_L$ given by \eqref{DefinitionSigmaL} is induced by a spin involution as in \eqref{InducedInvolution}.
\end{theorem}

\begin{proof}
Such a Lagrangian is explicitly given as follows.
Let $D = i \nabla_t$ be the covariant derivative with respect to the canonical ``real'' connection on $\mathbb{S} \otimes \C^d$.
The involution $\sigma$ induced by a spin involution exchanges the positive and negative spectral subspaces $L_+$ and $L_-$ of $D$, defined in \eqref{LPlusMinus}.
Combining this with the fact that $D$ has trivial kernel, we obtain that the subspace 
\begin{equation}
\label{StandardLagrangian}
 L_- = \bigoplus_{\lambda < 0} \mathrm{Eig}(D, \lambda)
\end{equation}
 is a Lagrangian.
$L_-$ is in general position with respect to $V_0$, as is not hard to see.
By Prop.~14.2 in \cite{BoosBavnbek}, the corresponding orthogonal projection $P_{L_-}$ has principal symbol \eqref{SymbolOfProjection}.
It is then a non-trivial calculation that the corresponding orthogonal involution $\sigma_{L_-}$ is then induced by a spin involution; 
this result can be found in \cite[Lemma~8]{JanssensDraft}, \cite[pp.~57-58]{HenriquesNotes} or \cite[Appendix B]{KristelWaldorf1} and is a variation on a result of Wassermann, see \cite[Thm.~14(c)]{Wassermann}.
\end{proof}

\begin{signDiscussion}
\label{SignDiscussion}
There are two spin involutions for $\mathbb{S}$, which differ by a sign.
By Thm.~\ref{ThmSpinvolution} and Corollary~\ref{CorollaryBimoduleIso}, there exists a spin involution such that with the right action induced by the corresponding orthogonal involution $\sigma$ of $H_0$, all Fock spaces are isomorphic to $L^2(\A_0)$.
Replacing the spin involution (and hence $\sigma$) by its negative, this modifies the bimodule structure by twisting the right action with the grading operator of $\A_0$.
As one can show that the grading isomorphism is not inner for $\A_0$, the Fock spaces with this modified bimodule structure will \emph{not} be isomorphic to $L^2(\A_0)$, neither with a grading preserving not with a grading reversing intertwiner.
Hence Thm.~\ref{ThmSpinvolution} singles out a spin involution $s$, such that the corresponding orthogonal involution makes all Fock bimodules isomorphic to $L^2(\A_0)$.
\end{signDiscussion}

According to Sign Discussion~\ref{SignDiscussion}, we fix, once and for all, the orthogonal involution $\sigma$ of $H_0$ induced by a spin involution singled out by Thm.~\ref{ThmSpinvolution}.
Then with the right action determined by $\sigma$, the Fock spaces $\F_L$, $L \in \Lag_0$, become super $\A_0$-$\A_0$-bimodules, where $\A_0$ is the completion of $\Cl(V_0)$ defined in \eqref{DefinitionA0}.
Combining Thm.~\ref{ThmSpinvolution} with Corollary~\ref{CorollaryBimoduleIso}, we obtain the following result.

\begin{corollary}
\label{CorollaryFockSpacesIsomorphic}
With the bimodule structure described above, the Fock spaces $\F_L$, $L \in \Lag_0$ are all isomorphic, as $\A_0$-$\A_0$-bimodules, to the standard bimodule $L^2(\A_0)$, via a unitary intertwiner that is either grading preserving or grading reversing.
\end{corollary}

\begin{remark}
\label{RemarkPreferredConnectedComponent}
The above construction determines a connected component $\Lag^+_0$ of $\Lag_0$ defined by the property that for each $L \in \Lag^+_0$, the super $\A$-$\A$-bimodule isomorphism $\F_L \cong L^2(\A_0)$ is grading preserving.
In particular, the Lagrangian $L_-$ defined in \eqref{StandardLagrangian} is contained in $\Lag_0^+$.
The other connected component consists of those $L \in \Lag_0$ such that the isomorphism $\F_L \cong L^2(\A_0)$ is grading reversing.
\end{remark}

\subsection{Fusion of Fock modules on loop space}
\label{SectionFusionFockLoop}

Let $X$ be an oriented Riemannian manifold.
For a loop $\gamma_1 \fuse \gamma_2$ in LX coming from $(\gamma_1, \gamma_2) \in PX^{[2]}$, we can consider the Hilbert space $H_{\gamma_1 \fuse \gamma_2}$ defined in \eqref{Hgamma} and its subspace $V_{\gamma_2}$ given in \eqref{Vgamma}.
We fix a spin involution as in Sign Discussion~\ref{SignDiscussion}.
Together with this spin involution, the flip diffeomorphism of $S^1$ induces an orthogonal transformation $\sigma$ sending $H_{\gamma_2 \fuse \gamma_1}$ to $H_{\gamma_1 \fuse \gamma_2}$. 
Composing with multiplication by $i$ (which for any ``real'' Hilbert space $H$ is an orthogonal transformation $H \to -H$), we obtain an orthogonal transformation
\begin{equation*}
i\sigma : -H_{\gamma_2 \fuse \gamma_1} \longrightarrow H_{\gamma_1 \fuse \gamma_2}, 
\end{equation*}
which restricts to an isomorphism $-V_{\gamma_1} \to V_{\gamma_2}^\perp$ and sends the equivalence class $\Lag_{\gamma_2 \fuse \gamma_1}$ to $\Lag_{\gamma_1 \fuse \gamma_2}^\perp$ (recall that we also fixed, once and for all, a trivialization of $\mathbb{S}|_{[0, \pi]}$ that realizes $V_{\gamma_2}$ as a subspace of $H_{\gamma_1 \fuse \gamma_2}$).
As explained in \S\ref{SectionAbstractConnesFusion} (see Remark~\ref{RemarkDifferentSpaces}), the corresponding isomorphism of Clifford algebras
\begin{equation*}
\begin{tikzcd}
  \Cl(H_{\gamma_2\fuse\gamma_1})^{\op} \cong\Cl(- H_{\gamma_2\fuse\gamma_1}) \ar[r, "\Cl_{i\sigma}"] & \Cl(H_{\gamma_1 \fuse \gamma_2})
\end{tikzcd}
\end{equation*}
sends $\Cl(V_{\gamma_1})^{\op} \subset \Cl(H_{\gamma_2 \fuse \gamma_1})^{\op}$ to $\Cl(V_{\gamma_2}^\perp) \subset \Cl(H_{\gamma_1 \fuse \gamma_2})$ and
extends to a $*$-isomorphism 
\begin{equation*}
  \CC_{\gamma_2 \fuse \gamma_1}^{\op} \longrightarrow \CC_{\gamma_1 \fuse \gamma_2}
\end{equation*}
that sends $\A_{\gamma_1}^{\op}$ to the super commutant of $\A_{\gamma_2}$ in $\CC_{\gamma_1 \fuse \gamma_2}$ (see Notation~\ref{NotationAgamma}).
For each $L \in \Lag_{\gamma_1 \fuse \gamma_2}$, this equips the Fock space $\F_L$ with the structure of an invertible super $\A_{\gamma_2}$-$\A_{\gamma_1}$-bimodule.
The following result is crucial.

\begin{theorem}
\label{ThmFockIso}
Let $(\gamma_1, \gamma_2, \gamma_3) \in PX^{[3]}$ and let $L_{ij} \in \Lag_{\gamma_i \fuse \gamma_j}$, $i<j$.
Then there exists a unitary isomorphism of $\A_{\gamma_3}$-$\A_{\gamma_1}$-bimodules
\begin{equation*}
\F_{L_{23}} \boxtimes_{\A_{\gamma_2}} \F_{L_{12}}  \cong
 \F_{L_{13}},
\end{equation*}
which is either grading preserving or grading reversing, and this isomorphism is unique up to a scalar.
\end{theorem}

\begin{proof}
Choose a lift $(g_1, g_2, g_3) \in P\SO(X)^{[3]}$ of $(\gamma_1, \gamma_2, \gamma_3)$ and consider the orthogonal transformations $g_{ij} = g_i \fuse g_j : H_0 \to H_{\gamma_i \fuse \gamma_j}$.
Then $L_{ij}^\prime = g_{ij}^*L_{ij}$ are Lagrangians in $H_0$, contained in the equivalence class $\Lag_0$.
We obtain $*$-isomorphisms $\Cl_{g_i} : \A_0 \to \A_{\gamma_i}$ and unitary isomorphisms $\Lambda_{g_{ij}}: \F_{L_{ij}^\prime} \to \F_{L_{ij}}$, which are intertwining along $\Cl_{g_j}$ and $\Cl_{g_i}$.
Moreover, the fusion product of $\Lambda_{g_{23}}$ and $\Lambda_{g_{12}}$ provides a unitary isomorphism
\begin{equation*}
  \Lambda_{g_{23}} \boxtimes \Lambda_{g_{12}} : \F_{L_{23}^\prime} \boxtimes_{\A_0} \F_{L_{12}^\prime} \longrightarrow \F_{L_{23}} \boxtimes_{\A_{\gamma_2}} \F_{L_{12}}
\end{equation*}
that intertwines the bimodule actions along $\Cl_{g_3}$ and $\Cl_{g_1}$.
This reduces claim to the statement that there exists an isomorphism
\begin{equation*}
  \F_{L_{23}^\prime} \boxtimes_{\A_0} \F_{L_{12}^\prime} \cong \F_{L_{13}^\prime}
\end{equation*}
of $\A_0$-$\A_0$-bimodules.
But by Corollary~\ref{CorollaryBimoduleIso}, each of the super $\A_0$-$\A_0$-bimodules $\F_{L_{ij}^\prime}$ is isomorphic to the standard bimodule $L^2(\A_0)$ (although possibly via a grading reversing intertwiner), so the existence of the isomorphism follows from the fact that the standard bimodule is the identity with respect to Connes fusion (see Example~\ref{ExampleModifiedActionConnesFusion}).
The statement about uniqueness follows from the fact that for any von Neumann algebra $A$, the only bimodule automorphisms of $L^2(A)$ are multiples of the identity.
\end{proof}

\subsection{Fusion for the loop space spinor bundle}
\label{SectionFusionSpinorBundle}

Let $X$ be an oriented Riemannian manifold of dimension $d$ and assume that there exists a spinor bundle $\SB$ on the loop space $LX$, according to Definition~\ref{DefinitionLoopSpaceSpinorBundle}.
At every loop $\gamma \in LX$, $\SB_\gamma$ is a super left module for the  completion $\CC_\gamma$ of $\Cl(H_\gamma)$.

For $(\gamma_1, \gamma_2) \in PX^{[2]}$, we have the von Neumann subalgebra $\A_{\gamma_2} \subset \Cl_{\gamma_1 \fuse \gamma_2}$ and, as seen in \S\ref{SectionFusionFockLoop}, its super commutant is identified with $\A_{\gamma_1}^{\op}$.
Hence $\SB_{\gamma_1 \fuse \gamma_2}$ attains the structure of a super $\A_{\gamma_2}$-$\A_{\gamma_1}$-bimodule.
That $\SB_{\gamma_1 \fuse \gamma_2}$ is irreducible as a left $\Cl_{\gamma_1 \fuse \gamma_2}$-module implies that this bimodule is invertible with respect to Connes fusion.
Recall from \S\ref{SectionIntervalAlgebra} that over $PX$, the von Neumann algebras $\A_\gamma$ glue together to a super von Neumann algebra bundle $\A$.
Varying fibers, we obtain that the restriction of $\SB$ to $PX^{[2]}$ (or, more formally, its pullback along the cup map \eqref{CupMap}) is a super $\A_2$-$\A_1$-bimodule bundle, where $\A_i$ denotes the pullback of the von Neumann algebra bundle $\A$ along the projection $PX^{[2]} \stackrel{i}{\to} PX$.
This is according to Notation~\ref{NotationSubmersion}, which we use throughout from now on.

\begin{definition}[Fusion product]
\label{DefinitionFusionProduct}
A \emph{(continuous) fusion product} for a spinor bundle $\SB$ on $LX$ is a grading preserving unitary isomorphism
\begin{equation*}
 \Upsilon :  \SB_{23} \boxtimes_{\A_2} \SB_{12} \longrightarrow \SB_{13}
\end{equation*}
of super $\A_3$-$\A_1$-bimodule bundles over $PX^{[3]}$.
This isomorphism is required to make the diagram
\begin{equation*}
  \begin{tikzcd}[column sep=2cm]
     \SB_{34} \boxtimes_{\A_3} \SB_{23} \boxtimes_{\A_2} \SB_{12} \ar[r, "\Upsilon_{234} \otimes \id"] \ar[d, "\id\otimes \Upsilon_{123}"'] &
  \SB_{24} \boxtimes_{\A_2} \SB_{12} \ar[d, "\Upsilon_{124}"] \\
  \SB_{34} \boxtimes_{\A_3} \SB_{13}\ar[r, "\Upsilon_{134}"] &
  \SB_{14}
  \end{tikzcd}
\end{equation*}
over $PX^{[4]}$ commutative.
\end{definition}

\begin{remark}
 As a spinor bundle $\SB$ is locally isomorphic to a Fock space bundle $\F_L$ for a section $L$ of $\Lag$, it follows from Thm.~\ref{ThmFockIso} that fusion products always exist locally.
 However, as we will see below, their global existence is obstructed unless $X$ admits a string structure.
\end{remark}

The rest of this section will discuss how to refine this definition to build in a notion of smoothness of the fusion product.
The super von Neumann algebra bundles $\A_1$ and $\A_2$ over $PX^{[2]}$ can be further pulled back to bundles over the restriction to $PX^{[2]}$ of the Lagrangian fibration $\Lag$ over $LX$ (again denoted by the same symbols).

It is straightforward to show the following lemma.

\begin{lemma}
The bimodule structure defined above turns the bundle of Fock spaces $\F$ over $\Lag$ into a super $\A_2$-$\A_1$-bimodule bundle in the sense of Definition~\ref{DefinitionBimoduleBundle}.
\end{lemma}

We denote the pullback of the Lagrangian fibration over $LX$ to $PX^{[2]}$ along the cup map again by $\Lag$.
Let $\Lag_{ij}$ be its pullback along the projection map $PX^{[3]} \to PX^{[2]}$ on the indicated factors.
Denote by
\begin{equation*}
  \Lag^{[1, 3]} \defeq \Lag_{23} \times_{PX^{[2]}} \Lag_{12} \times_{PX^{[2]}} \Lag_{13}
\end{equation*}
the fiber product of all these covers.
We have three von Neumann algebra bundles $\A_1$, $\A_2$, $\A_3$ over $\Lag^{[1, 3]}$, obtained as the pullback of the von Neumann algebra bundle $\A$ along the map $\Lag^{[1, 3]} \to PX^{[3]} \stackrel{i}{\to} PX$.
For $1 \leq i < j \leq 3$, we have the $\A_j$-$\A_i$-bimodule bundle $\F_{ij}$, which is the pullback of the Fock space bundle over $\Lag_{ij}$ to $\Lag^{[1, 3]}$.

\begin{definition}[Fusion line bundle]
\label{DefinitionFusionLineBundle}
The \emph{fusion line bundle} is the line bundle of bimodule homomorphisms
\begin{equation}
\label{FusionLineBundle}
  \Fus \defeq \underline{\Hom}(\F_{23} \boxtimes_{\A_2} \F_{12}, \F_{13})
\end{equation}
over $\Lag^{[1, 3]}$.
It is graded by parity of its elements.
\end{definition}

Explicitly, $(\gamma_1, \gamma_2, \gamma_3) \in PX^{[3]}$, then the fiber of $\Fus$ at a triple $(L_{23}, L_{12}, L_{13})$ of Lagrangians such that $L_{ij} \in \Lag_{\gamma_i \fuse \gamma_j}$ is the super line
\begin{equation*}
  \Fus(L_{23}, L_{12}, L_{13}) = \underline{\Hom}(\F_{L_{23}} \boxtimes_{\A_{\gamma_2}} \F_{L_{12}},
 \F_{L_{13}}).
\end{equation*}
graded according to whether its elements are grading preserving or grading reversing.
By irreducibility of the Fock representations, it is one-dimensional.
As both $\F_{23} \boxtimes_{\A_2} \F_{12}$ and $\F_{13}$ are continuous super $\A_3$-$\A_1$-bimodule bundles, the corresponding homomorphism space the structure of a continuous super line bundle.
It turns out that $\Fus$ has a canonical smooth structure.
We will construct this smooth structure at the end of the section, but for the moment take it for granted in order to define what a \emph{smooth} fusion product is.

\medskip 

Suppose that we are given a smoothing structure on our spinor bundle $\SB$, which, according to Definition~\ref{DefinitionSmoothingStructure} is a smooth structure on the associated super line bundle $\mathfrak{N} = \underline{\Hom}(\F, \pi^*\SB)$ over $\Lag$, satisfying a compatibility condition.
Then a fusion product $\Upsilon$ for $\SB$ induces an isomorphism of line bundles
\begin{equation}
\label{IsomorphismFusion}
\begin{aligned}
 \mathfrak{N}_{13}^{-1} \otimes  \mathfrak{N}_{23} \otimes \mathfrak{N}_{12}  
  &\longrightarrow \Fus , \\
  \Phi_{13} \otimes \Phi_{23} \otimes \Phi_{13} &\longmapsto \Phi_{13} \circ \Upsilon \circ (\Phi_{23} \boxtimes \Phi_{13})
\end{aligned}
\end{equation}
over $\Lag^{[1, 3]}$, which in general is only continuous. 
We can now formulate the following definition.

\begin{definition}[Smooth fusion product]
\label{DefinitionSmoothFusionProduct}
Suppose that the loop space spinor bundle $\SB$ carries a smoothing structure.
We say that a fusion product $\Upsilon$ for $\SB$ is \emph{smooth} if the canonical isomorphism \eqref{IsomorphismFusion} is smooth.
\end{definition}

We shall turn to the question of existence of smooth fusion products for a given spinor bundle in \S\ref{SectionFusion2Gerbe}, see in particular Thm.~\ref{TheoremFusionProductFusion2Gerbe}.

\medskip

We will now construct a smooth structure on $\Fus$ by writing it as an associated bundle.
To this end, consider the fusion line bundle over $\Lag_0 \times \Lag_0 \times \Lag_0$ (also denoted by $\Fus$), whose fiber at $(L_{23}, L_{12}, L_{13})$ is the space of bimodule homomorphisms $\F_{L_{23}} \boxtimes_{\A_0} \F_{L_{12}} \to \F_{L_{13}}$.
Here $\Lag_0$ is the canonical polarization of the Hilbert space $H_0$ defined in \eqref{Hnot}.

\begin{lemma}
\label{LemmaSmoothStructureOnFus0}
There exists a unique smooth structure on the fusion line bundle $\Fus$ over  $\Lag_0 \times \Lag_0 \times \Lag_0$ that such that the composition map
\begin{equation*}
  \Pf(L_{13}, L_{13}^\prime) \otimes \Fus(L_{23}, L_{12}, L_{13}) \otimes \Pf(L_{23}^\prime, L_{23}) \otimes \Pf(L_{12}^\prime, L_{12})\longrightarrow\Fus(L_{23}^\prime, L_{12}^\prime, L_{13}^\prime)
\end{equation*}
is smooth in all arguments.
\end{lemma}

\begin{proof}
  Clearly, for each fixed choice of $L_{12}$, $L_{23}$, $L_{13}$, there is a unique smooth structure on $\Fus$ (as a line bundle in the variables $L_{12}^\prime$, $L_{23}^\prime$, $L_{13}^\prime$) such that the above map is smooth, using the smooth structure on the Pfaffian line bundle $\Pf$ over $\Lag_0 \times \Lag_0$.
  That all these smooth structures coincide follows from the smoothness of the composition map for $\Pf$ over $\Lag_0 \times \Lag_0 \times \Lag_0$.
\end{proof}
%

There is an action of $P\SO(d)^{[3]}$ on $\Lag_0 \times \Lag_0 \times \Lag_0$ where 
\begin{equation*}
(q_1, q_2, q_3) \lact (L_{23}, L_{12}, L_{23}) \defeq (q_{23}L_{23}, q_{12}L_{12}, q_{13} L_{13}), \qquad q_{ij} = q_i \fuse q_j.
\end{equation*}
The fusion line bundle $\Fus$ is equivariant for this action, with the action lifted by the formula
\begin{equation}
\label{PSO(d)ActsOnFus}
  (q_1, q_2, q_3) \lact \Upsilon \defeq \Lambda_{q_{13}} \circ \Upsilon \circ (\Lambda_{q_{23}} \boxtimes \Lambda_{q_{12}})^*.
\end{equation}

\begin{proposition}
\label{PropActionOnFusSmooth}
The action \eqref{PSO(d)ActsOnFus} is smooth with respect to the smooth structure on $\Fus$ from Lemma~\ref{LemmaSmoothStructureOnFus0}.
\end{proposition}

Given this proposition, the fusion line bundle $\Fus$ on the fiber bundle $\Lag^{[1, 3]}$ over $PX^{[3]}$ now acquires a smooth structure from the local trivializations of $\Lag$, just as the Pfaffian line bundle over $\Lag^{[2, 2]}$.
Here Prop.~\ref{PropActionOnFusSmooth} is used to check that the corresponding transition functions are smooth.

\begin{proof}[of Prop.~\ref{PropActionOnFusSmooth}]
By the characterizing property of the smooth structure and the smoothness of the action of $\O_{\res}(H)$ on the Pfaffian line bundle $\Pf$, it suffices to show smoothness of the action restricted to the orbit of the triple $(L, L, L)$ for a fixed Lagrangian $L \in \Lag_0$.
Here we have have canonical isomorphisms
\begin{equation*}
\begin{aligned}
  \Fus(q_{23}L, q_{12}L, q_{13}L) &\cong
   \Pf(L, q_{13}L) \otimes \Fus(L, L, L )  \otimes \Pf(q_{23}L, L) \otimes \Pf(q_{12}L, L),
\end{aligned}
\end{equation*}
which are smooth in $q_1, q_2, q_3$.
If now $\Upsilon \in \Fus(L, L, L)$ is a unitary and $U_{ij} \in \Imp_L$ are implementers for $q_{ij}$, then under this isomorphism, the element $\Lambda_{q_{13}} \circ \Upsilon \circ (\Lambda_{q_{23}} \boxtimes \Lambda_{q_{12}})^*$ is sent to
\begin{equation*}
 (-1)^{|U_{23}||U_{12}|} (\Lambda_{q_{13}}U_{13}^*) \otimes \bigl( U_{13} \circ \Upsilon \circ (U^*_{23} \boxtimes U^*_{12}) \bigr) \otimes (U_{23} \Lambda_{q_{23}}^*) \otimes (U_{12} \Lambda_{q_{12}}^*);
\end{equation*}
here we use the relation \eqref{FunctorialityConnesFusion}. 
By smoothness of the map \eqref{IsoImpLPfL}, the factors contained in the Pfaffian lines depend smoothly on $q_{i}$ and $U_{ij}$.
Hence the smoothness of the whole term is equivalent to smoothness of the middle term in $\Fus(L, L, L)$, which in turn is equivalent to smoothness of the element
\begin{equation*}
\Upsilon \circ (U_{23} \boxtimes U_{12}) \circ \Upsilon^* \in \Imp_L
\end{equation*}
 in $U_{ij}$ and $q_i$.
 Let us choose $L \in \Lag_0^+$, so that $\Fus(L, L, L)$ is even and pick an even isomorphism $\F_L \cong L^2(\A_0)$ (see Remark~\ref{RemarkPreferredConnectedComponent}).
We may choose $\Upsilon$ to correspond, under this isomorphism, to the canonical unitor from Example~\ref{ExampleModifiedActionConnesFusion}.
A straightforward calculation then gives the result
\begin{equation}
\label{IdentityFusion}
\Upsilon \circ (U_{23} \boxtimes U_{12}) \circ \Upsilon^*
= U_{23}L^2(\Cl_{q_2})^* U_{12},
\end{equation}
where $L^2(\Cl_{q_2})^*$ denotes the unitary on $L^2(\A_0)$ induced by the $*$-automorphism $\Cl_{q_2}$; see Remark~\ref{RemarkL2phi}.
The right hand side of \eqref{IdentityFusion} is just the fusion product of the implementer extension, see \eqref{FusionProductImplementers} below.
\end{proof}

\subsection{The stringor bundle}
\label{SectionStringorBundle}

In their preprint \cite{StolzTeichnerSpinorBundle}, Stolz and Teichner argue that a fusion product $\Upsilon$ makes a spinor bundle $\SB$ on the loop space $LX$ behave locally in $X$, and should therefore be called the \emph{stringor bundle} on $X$.
In this section, we make this idea rigorous using (a von Neumann algebra analog of) the language of 2-vector bundles developed in \cite{kristel20212vector}.

As a warm-up, we start by explaining the notion of finite-dimensional 2-vector bundles and then explain how to modify this notion to meet the requirements of the present context.
A \emph{super 2-vector bundle} $\mathcal{V}$ over a manifold $X$, given in terms of a cover $Y \to X$, consists of a super algebra bundle $\A$ over $Y$, a super $\A_2$-$\A_1$-bimodule bundle $\mathfrak{M}$ over $Y^{[2]}$ and a grading preserving isomorphism
\begin{equation*}
  \lambda : \mathfrak{M}_{23} \otimes_{\A_2} \mathfrak{M}_{12} \longrightarrow \mathfrak{M}_{13}
\end{equation*}
of super $\A_3$-$\A_1$-bimodule bundles over $Y^{[3]}$ that is associative over $Y^{[4]}$ (here $Y^{[k]}$ is the $k$-fold iterated fiber product of $Y \to X$; see Notation~\ref{NotationSubmersion}).
Similar to bundle gerbes, these data and conditions can be organized in a diagram as follows.
\begin{equation}
\label{General2-VectorDiagram}
  \mathcal{V} =   \left[
\begin{tikzcd}
  \A \ar[d,  dotted, -] & \mathfrak{M} \ar[d,  dotted, -] & \substack{\text{2-vector} \\ \text{bundle} \\ \text{product}~\lambda} \ar[d, -, dotted] & \substack{\text{associativity} \\ \text{of } \lambda} \ar[d, -, dotted] 
  \\
  Y \ar[d]& Y^{[2]} \ar[l, shift left=1mm] \ar[l, shift right=1mm]&  Y^{[3]} \ar[l, shift left=2mm] \ar[l, shift right=2mm] \ar[l] & Y^{[4]} \ar[l, shift left=1mm] \ar[l, shift right=1mm] \ar[l, shift left=3mm] \ar[l, shift right=3mm]
  \\
  X
\end{tikzcd}
\right]
\end{equation}
Observe that this notion is a common generalization of super algebra bundles over $X$ and bundle gerbes over $X$: Any algebra bundle over $X$ gives rise to a 2-vector bundle over $X$, given in terms of the trivial cover, and any super bundle gerbe over $X$ is a super 2-vector bundle for which the algebra bundle $\A$ is the trivial algebra bundle.

For each $X$, there is a bicategory $\sTwoVect(X)$ of 2-vector bundles over $X$;
varying the base manifold $X$, super 2-vector bundles form a 2-stack over the site of manifolds.
This 2-stack can be obtained by applying the plus construction of Nikolaus and Schweigert \cite{NikolausSchweigert} to the pre-sheaf of bicategories $\sAlg$ of super algebra bundles, super bimodule bundles and even intertwiners described in \cite{InsideousBicategory},
\begin{equation}
\label{PlusConstruction2Vect}
  \sTwoVect \defeq \sAlg^+.
\end{equation}

In \cite{kristel20212vector}, we described a 2-stack of super 2-vector bundles whose typical fiber $\A$ is a finite-dimensional super algebra.
In the present context, we have to allow super 2-vector bundles whose typical fiber is a super von Neumann algebra.
In view of the simple construction \eqref{PlusConstruction2Vect}, such a stack is easily obtained by applying the plus construction to the presheaf of bicategories $\svNAlg$ of super von Neumann algebra bundles, super (Hilbert) bimodule bundles and even unitary intertwiners sketched in Appendix~\ref{SectionVNBundlesBimodules}.
Explicitly, such super 2-vector bundles can be visualized also by the diagram \eqref{General2-VectorDiagram}, but now $\A$ is a super von Neumann algebra bundle (Definition~\ref{DefinitionvNBundle}) and $\M$ is a (Hilbert) super $\A_2$-$\A_1$-bimodule bundle (Definitions~\ref{DefinitionBimoduleBundle}) whose typical fiber is implementing (Definition~\ref{DefinitionImplementingBimodule}).
As the composition of morphisms in this bicategory is the fiberwise Connes fusion product, the 2-vector bundle product in this case is a grading preserving unitary intertwiner
\begin{equation*}
  \lambda : \mathfrak{M}_{23} \boxtimes_{\A_2} \mathfrak{M}_{12} \longrightarrow \mathfrak{M}_{13}.
\end{equation*}

Using this notion of super 2-vector bundle, it is completely straight forward to define the stringor bundle of Stolz and Teichner.
Let $X$ be an oriented Riemannian manifold and let $\SB$ be a loop space spinor bundle (Definition~\ref{DefinitionLoopSpaceSpinorBundle}), together with a fusion product $\Upsilon$ (Definition~\ref{DefinitionFusionProduct}).
Observe that by Thm.~\ref{TheoremFusionProduct} below, the existence of these structures is equivalent to the manifold $X$ being string.

\begin{definition}[Stringor bundle]
The \emph{stringor bundle} $\mathcal{S}$ corresponding to the data $(\SB, \Upsilon)$ is the super 2-vector bundle over $X \times X$ given in terms of the cover $PX$, with algebra bundle the von Neumann algebra bundle $\A$ described in \S\ref{SectionIntervalAlgebra}, super $\A_2$-$\A_1$-bundle over $PX^{[2]}$ the pullback of $\SB$ along the cup map \eqref{CupMap} and 2-vector bundle product over $PX^{[3]}$ given by the fusion product $\Upsilon$.
\begin{equation*}
  \mathcal{S} =   \left[
\begin{tikzcd}
  \A \ar[d, dotted, -] & \SB \ar[d,  dotted, -] & \substack{ \text{fusion} \\ \text{product}~\Upsilon} \ar[d, -, dotted] & \substack{\text{associativity of } \\ \text{fusion product}} \ar[d, -, dotted] 
  \\
  PX \ar[d]& PX^{[2]} \ar[l, shift left=1mm] \ar[l, shift right=1mm]&  PX^{[3]} \ar[l, shift left=2mm] \ar[l, shift right=2mm] \ar[l] & PX^{[4]} \ar[l, shift left=1mm] \ar[l, shift right=1mm] \ar[l, shift left=3mm] \ar[l, shift right=3mm]
  \\
  X \times X
\end{tikzcd}
\right]
\end{equation*}
\end{definition}

For any point $x \in X$, pullback along the inclusion $X \cong \{x\} \times X \hookrightarrow X \times X$ gives a stringor bundle over $X$.

\section{The Lagrangian 2-gerbe} \label{SectionLagrangian2Gerbe}

In this section, we discuss the existence of fusion products for loop space spinor bundles.
Our main player here is the \emph{Lagrangian 2-gerbe} $\LagTwoGrb_X$, a super bundle 2-gerbe obtained by degression of the Lagrangian gerbe on the loop space.
We show that $\LagTwoGrb_X$ is ungraded if and only if the manifold is spin.
In this case, we construct a further canonical isomorphism from the Lagrangian 2-gerbe to a certain lifting 2-gerbe, whose (non-)triviality is well known to be equivalent to the string condition.

On the other hand, we show show that the existence of a fusion product on a spinor bundle $\SB$ is obstructed by a certain other super bundle 2-gerbe $\FusTwoGrb(\SB)$ which we call the \emph{fusion 2-gerbe} of $\SB$.
We then show that $\FusTwoGrb(\SB)$ is isomorphic to $\LagTwoGrb_X$.

Throughout this section, we freely use the material of \S\ref{SectionBundle2Gerbes}, \S\ref{SectionIso2Gerbes} \& \S\ref{SectionClassificationBundle2Gerbes}.

\subsection{The Lagrangian 2-gerbe} \label{SubSectionLagrangian2Gerbe}

Let $X$ be an oriented Riemannian manifold and let $\LagGrb_{LX}$ be the Lagrangian gerbe over $LX$ introduced in \S\ref{SectionApplicationLoopSpace}.
In this section, we construct a super bundle 2-gerbe whose defining super bundle gerbe is $\LagGrb_{LX}$.
We call this super bundle gerbe the \emph{Lagrangian 2-gerbe} and denote it by $\LagTwoGrb_X$.

$\LagTwoGrb_X$ is most naturally constructed as a bundle gerbe over $X \times X$, defined in terms of the cover $Y = PX$, mapping to $X \times X$ via the end point evaluations; a super bundle 2-gerbe on $X$ can be obtained by pullback along the second factor inclusion $X \hookrightarrow X \times X$, given the choice of a basepoint.
As all these embeddings are homotopic, the super bundle 2-gerbes obtained this way for two different choices of basepoints are isomorphic.

\medskip


We will start by describing the partial semi-bisimplicial diagram for $\LagTwoGrb_X$, as in \eqref{CoverDiagram}, which consists of various pullbacks and fiber products taken of the restriction (or rather pullback along the cup map \eqref{CupMap}) of the Lagrangian fibration $\Lag$ over $LX$ to $PX^{[2]}$. 
Abusing notation, we will denote this restriction also by $\Lag$ throughout.

Let $n \geq 2$. 
In view of Notation~\ref{NotationSubmersion}, the pullback of $\Lag$ along the various projections $PX^{[n]} \to PX^{[2]}$ provides covers $\Lag_{ij}$ of $PX^{[n]}$.
Similarly, we can pull back the $m$-fold fiber products $\Lag^{[m]}$ of $\Lag$ with itself over $PX^{[2]}$, to obtain covers $\Lag^{[m]}_{ij}$.
We denote the fiber product of all these covers by
\begin{equation}
\label{DefinitionLagmn}
  \Lag^{[m, n]} ~~\defeq~~ \underbrace{\Lag_{12}^{[m]} \times_{PX^{[n]}} \cdots \times_{PX^{[n]}} \Lag_{ij}^{[m]} \times_{PX^{[n]}} \cdots \qquad}_{\substack{\text{fiber product over $PX^{[n]}$} \\ \text{of all $\Lag_{ij}^{[m]}$, where $1 \leq i < j \leq n$}}}.
\end{equation}
In particular, $\Lag^{[m, 2]} = \Lag^{[m]}$, the $m$-fold fiber product of $\Lag$ over $PX^{[2]}$.
These spaces fit into the following partial semi-bisimplicial diagram.
\begin{equation} 
\label{CoverDiagramLag}
\begin{aligned}
\begin{tikzcd}[column sep=0.7cm]
&
 	\vdots
	\ar[d]
 	\ar[d, shift left=2mm]
 	\ar[d, shift right=2mm]
	& 
	\vdots 
	\ar[d]
 	\ar[d, shift left=2mm]
 	\ar[d, shift right=2mm]
	&
	 	\vdots
	\ar[d]
 	\ar[d, shift left=2mm]
 	\ar[d, shift right=2mm]
			\\
&
 	\Lag^{[2, 2]}
 	\ar[d, shift left=1mm]
 	\ar[d, shift right=1mm]
	&
		\Lag^{[2, 3]}
			\ar[l]
 			\ar[l, shift left=2mm]
 			\ar[l, shift right=2mm]
 		\ar[d, shift left=1mm]
 		\ar[d, shift right=1mm]
		&
			\Lag^{[2, 4]}
 			\ar[d, shift left=1mm]
 			\ar[d, shift right=1mm]
			\ar[l, shift left=1mm] 
		\ar[l, shift right=1mm]
			\ar[l, shift left=3mm] 
		\ar[l, shift right=3mm]
		& 
		\ar[l]
			\ar[l, shift left=2mm]
 			\ar[l, shift right=2mm]
 			\ar[l, shift left=4mm]
 			\ar[l, shift right=4mm]
			\cdots
			\\
&
	\Lag^{[1, 2]} 
	\ar[d]
	& 
		\Lag^{[1, 3]}
			\ar[l]
 			\ar[l, shift left=2mm]
 			\ar[l, shift right=2mm]
		\ar[d]
		&
			\Lag^{[1, 4]}
 			\ar[d]
			\ar[l, shift left=1mm] 
		\ar[l, shift right=1mm]
			\ar[l, shift left=3mm] 
		\ar[l, shift right=3mm]
		&
		\ar[l]
			\ar[l, shift left=2mm]
 			\ar[l, shift right=2mm]
 			\ar[l, shift left=4mm]
 			\ar[l, shift right=4mm]
			\cdots
			\\
PX \ar[d]
 &
	PX^{[2]}
	\ar[l]
	&
		PX^{[3]}
			\ar[l]
 			\ar[l, shift left=2mm]
 			\ar[l, shift right=2mm]
			&
			PX^{[4]}
 			\ar[l, shift left=1mm]
 			\ar[l, shift right=1mm]
 			\ar[l, shift left=3mm]
 			\ar[l, shift right=3mm]
			&
			\ar[l]
			\ar[l, shift left=2mm]
 			\ar[l, shift right=2mm]
 			\ar[l, shift left=4mm]
 			\ar[l, shift right=4mm]
			\cdots
			\\
X \times X
\end{tikzcd}
\end{aligned}
\end{equation}
Explicitly, elements of $\Lag^{[m, n]}$ over an element $(\gamma_1, \dots, \gamma_n)$ of $PX^{[n]}$ consists of collections $(L_{ij}^a)_{1 \leq i < j \leq n}^{1 \leq a \leq m}$ such that $L_{ij}^a \in \Lag_{\gamma_i \fuse \gamma_j}$.
We can identify $\Lag^{[m, n]} = (\Lag^{[1, n]})^{[m]}$, the exterior fiber product taken over $PX^{[n]}$, which gives maps 
\begin{equation*}
\pi^i : \Lag^{[m, n]} \to \Lag^{[m-1, n]}, \qquad i=1, \dots, m.
\end{equation*}
These are the vertical maps in \eqref{CoverDiagramLag}.
The horizontal maps in \eqref{CoverDiagramLag} are
\begin{equation*}
\pi_k : \Lag^{[m, n]} \to \Lag^{[m, n-1]}, \qquad k=1, \dots, n
\end{equation*} 
covering the $k$-th projection $PX^{[n]} \to PX^{[n-1]}$. 
Explicitly,
\begin{equation*}
 \pi_k\Bigl((L_{ij}^a)^{1 \leq a \leq m}_{1 \leq i < j \leq n}\Bigr) = (\tilde{L}_{ij}^a)^{1 \leq a \leq m}_{1 \leq i < j \leq n-1}, \qquad \text{where} \qquad \tilde{L}_{ij}^a = \begin{cases} L_{ij}^a & i< j < k \\ L_{i, j+1}^a & i < k \leq j \\ L_{i+1, j+1}^a & k \leq i < j.\end{cases}
\end{equation*}

In relation to \eqref{CoverDiagramLag}, the data defining $\LagTwoGrb_X$ can be arranged as follows.
\begin{equation} 
\label{LagTwoGrbBigDiagram}
\begin{aligned}
%
\LagTwoGrb_X =  \left[
\begin{tikzcd}
 & 
 	\lambda 
	\ar[d,  dotted, -]
	& 
			\\
 & 
 	\Pf
 	\ar[d,  dotted, -]
	& 
		\mu
			\ar[l,  dotted, -]
 		\ar[d,  dotted, -]
			\\
&
	{\color{white}. }
	\ar[d,  dotted, -]
	& 
		\Fus
			\ar[l,  dotted, -]
		\ar[d,  dotted, -]
		&
			\alpha
			\ar[d,  dotted, -]
			\ar[l,  dotted, -]
			\\
PX \ar[d]
&
	PX^{[2]}
		\ar[l, shift left=1mm] 
		\ar[l, shift right=1mm]
	&
		PX^{[3]}
			\ar[l]
 			\ar[l, shift left=2mm]
 			\ar[l, shift right=2mm]
		&
			PX^{[4]}
			\ar[l, shift left=1mm] 
		\ar[l, shift right=1mm]
			\ar[l, shift left=3mm] 
		\ar[l, shift right=3mm]
			\\
			X \times X
\end{tikzcd}
\right]
\end{aligned}
\end{equation}
Explicitly, the objects over $\Lag^{[i,j]}$ are the following.
\begin{enumerate}
\item[{[2, 2]}]
$\Pf$ is the Pfaffian line bundle, introduced in \S\ref{SectionApplicationLoopSpace}.
\item [{[3, 2]}]
$\lambda$ is the corresponding composition map \eqref{MultiplicationBundleMorphism}.
\item[{[2, 3]}]
$\Fus$ is the Fusion line bundle from Definition~\ref{DefinitionFusionLineBundle}.
\item[{[3, 3]}]
Denoting horizontal pullbacks by lower indices and vertical pullbacks by upper indices (following Notation~\ref{NotationBisimplicialPullbacks}), $\mu$ is the canonical isomorphism of super line bundles
\begin{equation}
\label{CanonicalIsomorphismNu}
\begin{aligned}
\mu : \Fus^2 \otimes \Pf_{23} \otimes \Pf_{12} &\longrightarrow \Pf_{13} \otimes \Fus^1. \\
\end{aligned}
\end{equation}
over $\Lag^{[2, 4]}$, which at two triples $(L_{12}, L_{23}, L_{13})$ and  $(L_{12}^\prime, L_{23}^\prime, L_{13}^\prime)$, both lying over $(\gamma_1, \gamma_2, \gamma_3)$, is defined as the cospan
\begin{equation*}
  \begin{tikzcd}[row sep=0.3cm]
  \Fus(L_{23}^\prime, L_{12}^\prime, L_{13}^\prime) \otimes \Pf(L_{23}, L_{23}^\prime) \otimes \Pf(L_{12}, L_{12}^\prime) \ar[rd, bend left=10] \ar[dd, dashed, "\mu"] & \\
  & \Fus(L_{23}, L_{12}, L_{13}^\prime).
  \\
  \Pf(L_{13}, L_{13}^\prime) \otimes \Fus(L_{23}, L_{12}, L_{13}) \ar[ur, bend right=10]
  \end{tikzcd}
\end{equation*}
The diagonal maps are the evident isomorphisms given by composition;
here we observe that a bounded linear map $\Phi : \F_{L_{ij}} \to \F_{L_{ij}^\prime}$ intertwines the left $\CC_{\gamma_i \fuse \gamma_j}$ actions if and only if it intertwines the $\A_{\gamma_j}$-$\A_{\gamma_i}$-bimodule actions.
Hence the elements of the Pfaffian line can be viewed as bimodule homomorphisms.

\item[{[1, 4]}]
To define the isomorphism $\alpha$, we define the \emph{triple fusion line bundle} over $\Lag^{[1, 4]}$ by
\begin{equation*}
  \Fus^{(3)} = \underline{\Hom}(\F_{34} \boxtimes_{\A_3} \F_{23} \boxtimes_{\A_2} \F_{12}, \F_{14}),
\end{equation*}
where, generalizing Notation~\ref{NotationSubmersion} in an obvious way, $\F_{ij}$ and $\A_i$ denote the pullbacks to $\Lag^{[1, 4]}$ of the Fock space bundle $\F$ and the von Neumann algebra bundle $\A$ along the appropriate maps in the diagram \eqref{CoverDiagramLag}.
We then have canonical isomorphisms
\begin{equation}
\label{FusLineToFusLine3}
\begin{aligned}
&\Fus_{124} \otimes \Fus_{234} \longrightarrow \Fus^{(3)}, \qquad & &\Phi \otimes \Psi \longmapsto \Phi \circ (\id_{\F_{12}} \boxtimes \Psi) \\
&\Fus_{134} \otimes \Fus_{123} \longrightarrow \Fus^{(3)}, \qquad & &\Phi \otimes \Psi \longmapsto \Phi \circ (\Psi\boxtimes\id_{\F_{34}}  ),\\
\end{aligned}
\end{equation}
over $\Lag^{[1, 4]}$ which enable defining $\alpha$ as the cospan
\begin{equation}
\label{DefinitionAlphaTilde}
  \begin{tikzcd}
    \Fus_{124} \otimes \Fus_{234} \ar[r, "\cong"'] \ar[rr, bend left=20, dashed, "\alpha"]&  \Fus^{(3)} & \Fus_{134} \otimes \Fus_{123}. \ar[l, "\cong"]
  \end{tikzcd}
\end{equation}
Here we implicitly use the associativity isomorphism of the Connes fusion product to identify $(\F_{34} \boxtimes_{\A_3} \F_{23}) \boxtimes_{\A_2} \F_{12} \cong \F_{34} \boxtimes_{\A_3} ( \F_{23} \boxtimes_{\A_2} \F_{12})$.
There is a unique smooth structure on $\Fus^{(3)}$ turning the isomorphisms \eqref{FusLineToFusLine3} into smooth bundle isomorphisms.
\end{enumerate}

\begin{theorem}
\label{TheoremCoherenceLagTwoGrb}
The data specified above constitute the structure of a super bundle 2-gerbe.
\end{theorem}

\begin{proof}
We have to verify four coherence and compatibility conditions.
The first coherence condition is just the statement that $\LagGrb_{LX}$ is a super bundle gerbe and follows from associativity of composition.
The compatibility condition \eqref{CoherenceNu} for $\lambda$ and $\mu$ over $\Lag^{[2, 3]}$ is straight forward to check.

To see the compatibility condition \eqref{CoherenceAlphaTilde} for $\alpha$ and $\mu$, we observe that there is a canonical isomorphism
\begin{equation}
\label{NuThree}
\mu^{(3)} : \Fus^{(3), 2} \otimes \Pf_{34} \otimes \Pf_{23} \otimes \Pf_{12} \longrightarrow \Pf_{14} \otimes \Fus^{(3), 1}, 
\end{equation}
defined in a similar fashion as $\mu$.
Using the identity \eqref{FunctorialityConnesFusion} for the Connes fusion product, it is then straightforward to check that the two diagrams
\begin{equation*}
\begin{tikzcd}[column sep=1cm]
\Fus_{124}^2 \otimes \Fus_{234}^2 \otimes \Pf_{34} \otimes \Pf_{23} \otimes \Pf_{12} 
	\ar[r] 
	\ar[d, "\id \otimes \mu_{234} \otimes \id"']
& 
\Fus^{(3), 2} \otimes \Pf_{34} \otimes \Pf_{23} \otimes \Pf_{12}
	\ar[ddd, "\mu^{(3)}"]
\\
\Fus_{124}^2 \otimes \Pf_{24} \otimes \Fus^1_{234} \otimes  \Pf_{12}
	\ar[d, equal]
&
\\
\Fus_{124}^2 \otimes \Pf_{24} \otimes  \Pf_{12} \otimes \Fus^1_{234}
	\ar[d, "\id \otimes \mu_{124}"']
&
\\
 \Pf_{14} \otimes  \Fus_{124}^1 \otimes \Fus^1_{234}
 	\ar[r]
	&
	\Pf_{14} \otimes \Fus^{(3),1}
\end{tikzcd}
\end{equation*}
and
\begin{equation*}
\begin{tikzcd}[column sep=1cm]
\Fus^{(3), 2} \otimes \Pf_{34} \otimes \Pf_{23} \otimes \Pf_{12}  
	\ar[ddd, "\mu^{(3)}"']
& 
\Fus_{134}^2 \otimes\Fus_{123}^2 \otimes \Pf_{34} \otimes \Pf_{23} \otimes \Pf_{12}
	\ar[d, equal]
	\ar[l]
\\
&
\Fus_{134}^2 \otimes \Pf_{34} \otimes\Fus_{123}^2 \otimes \Pf_{23} \otimes \Pf_{12}
	\ar[d, "\id \otimes \id \otimes \mu_{123} "]
\\
&
\Fus_{134}^2 \otimes \Pf_{34} \otimes \Pf_{13}  \otimes\Fus_{123}^1 
	\ar[d, "\mu_{134} \otimes \id"]
\\
 \Pf_{14} \otimes  \Fus^{(3), 1}
	&
	\Pf_{14} \otimes \Fus_{134}^1 \otimes\Fus_{123}^1
	\ar[l]
\end{tikzcd}
\end{equation*}
over $\Lag^{[2, 4]}$ commute, where each of the horizontal arrows is one of the canonical isomorphisms \eqref{FusLineToFusLine3}.
Joining the first with the second along the common morphism $\mu^{(3)}$, we obtain \eqref{CoherenceAlphaTilde}, where the vertical maps give the defining cospan \eqref{DefinitionAlphaTilde} for $\alpha$.

To verify commutativity of the diagram \eqref{CoherenceAlpha} over $\Lag^{[1, 5]}$, we introduce the \emph{quadruple fusion line bundle}
\begin{equation*}
\Fus^{(4)} = \underline{\Hom}(\F_{45} \boxtimes_{\A_4} \F_{34} \boxtimes_{\A_3} \F_{23} \boxtimes_{\A_2} \F_{12}, \F_{15}),
\end{equation*}
a super line bundle over $\Lag^{[1, 5]}$.
There is a unique smooth structure on $\Fus^{(4)}$ such that each of the canonical isomorphisms
\begin{equation}
\label{CanonicalIsomorphismsFus4}
\begin{aligned}
&\Fus^{(3)}_{1235} \otimes \Fus_{345} \longrightarrow \Fus^{(4)},
& \qquad  &\Omega \otimes \Upsilon \longmapsto \Omega \circ (\Upsilon \boxtimes \id_{\F_{23} \boxtimes_{\A_2} \F_{12}})
\\
 &\Fus_{125} \otimes \Fus^{(3)}_{2345} \longrightarrow \Fus^{(4)},
 & & \Upsilon \otimes \Omega \longmapsto \Upsilon \circ (\Omega \boxtimes \id_{\F_{12}})
\\
&\Fus^{(3)}_{1245} \otimes \Fus_{234}  \longrightarrow \Fus^{(4)},
& & \Omega \otimes \Upsilon \longmapsto \Omega \circ (\id_{\F_{45}} \boxtimes \Upsilon \boxtimes \id_{\F_{12}})
\\
&\Fus_{145} \otimes \Fus_{1234}^{(3)} \longrightarrow \Fus^{(4)},
& & \Upsilon \otimes \Omega \longmapsto \Upsilon \circ (\id_{\F_{45}} \boxtimes \Omega)
\\
&\Fus^{(3)}_{1345} \otimes \Fus_{123} \longrightarrow \Fus^{(4)},
& & \Omega \otimes \Upsilon \longmapsto  \Omega \circ (\id_{\F_{23} \boxtimes_{\A_2} \F_{12}} \boxtimes \Upsilon )
\end{aligned}
\end{equation}
is smooth.
We now fill the diagram \eqref{CoherenceAlpha} as follows.
\begin{equation*}
\begin{footnotesize}
\begin{tikzcd}[column sep=-1.6cm]
&[+1cm] & &
\Fus_{125} \otimes \Fus_{235} \otimes \Fus_{345} 
\ar[dl] 
\ar[dr] 
\ar[dlll]
\ar[drrr]
&[-0.5cm] & &
\\
\Fus_{125} \otimes \Fus_{245} \otimes \Fus_{234} 
\ar[dd]
\ar[rr]
\ar[dr]
& &
 \Fus_{125} \otimes \Fus^{(3)}_{2345}
 \ar[dr]
& & 
\Fus^{(3)}_{1235} \otimes \Fus_{345}
\ar[dl]
& &
\Fus_{135} \otimes \Fus_{123} \otimes \Fus_{345} 
\ar[ll]
\ar[dd, equal] 
\\
&
\Fus^{(3)}_{1245} \otimes \Fus_{234} 
\ar[rr]
& &
\Fus^{(4)}
& &
{\color{white} ......................................}
&
\\
\Fus_{145} \otimes \Fus_{124} \otimes \Fus_{234} 
\ar[rr]
\ar[ur]
\ar[drrr]
& &
\Fus_{145} \otimes \Fus_{1234}^{(3)}
\ar[ur]
& &
\Fus^{(3)}_{1345} \otimes \Fus_{123}
\ar[ul]
& & 
\Fus_{135} \otimes \Fus_{345} \otimes \Fus_{123} 
\ar[dlll]
\ar[ll]
\\
& & &
\Fus_{145} \otimes \Fus_{134} \otimes \Fus_{123} 
\ar[ul]
\ar[ur]
\end{tikzcd}
\end{footnotesize}
\end{equation*}
Here the arrows going inward from the outer nodes are (pullbacks of) the canonical isomorphisms \eqref{FusLineToFusLine3}, while the arrows arriving at the middle node are the five canonical isomorphisms \eqref{CanonicalIsomorphismsFus4}.
The outer triangles of this hexagonal diagram are the defining triangles for $\alpha$, and commutativity of the tetragons and of the pentagon having one vertex $\Fus^{(4)}$ is straighforward to check case by case.
This shows commutativity of \eqref{CoherenceAlpha} and finishes the construction of the Lagrangian 2-gerbe $\LagTwoGrb_X$.
\end{proof}

\subsection{The fusion 2-gerbe} 
\label{SectionFusion2Gerbe}

Let $X$ be an oriented Riemannian manifold.
Assume that the Lagrangian gerbe $\LagGrb_{LX}$ over $LX$ is trival, so that we may construct a loop space spinor bundle; see Thm.~\ref{ThmEquivalenceFinite}.
Suppose we are given such a choice of loop space spinor bundle $\SB$, together with the choice of a smoothing structure in the sense of Definition~\ref{DefinitionSmoothingStructure}, in other words, a smooth structure on the line bundle $\mathfrak{N} = \underline{\Hom}(\F, \pi^*\SB)$ over $\Lag$.

In this situation the pullback of $\SB$ to $PX^{[2]}$ along the cup map \eqref{CupMap} is an $\A_2$-$\A_1$-bimodule bundle, where $\A_i$ denotes the pullback of the Clifford von Neumann algebra bundle $\A \to PX$ along the $i$-th projection $PX^{[2]} \to PX$ (recall Notation~\ref{NotationSubmersion}).
Over $PX^{[3]}$, we may therefore define the super line bundle
\begin{equation}
\label{DefinitionLFusionLineBundleS}
\M \defeq \underline{\Hom}(\SB_{23} \boxtimes_{\A_2}\SB_{12}, \SB_{13}).
\end{equation}
As pointwise, $\SB$ is isomorphic to a Fock space, Thm.~\ref{ThmFockIso} implies that $\M$ is indeed a complex line; as usual, it is graded according to whether its elements are parity-preserving or -reversing.
As $\SB$ is a continuous bundle of Hilbert spaces, $\M$ is a continuous line bundle over $PX^{[3]}$.
However, a smooth structure is specified by requiring that the canonical isomorphism $\beta$ over $\Lag^{[1, 3]}$, given as the cospan
\begin{equation}
\label{CanonicalIsomorphismBeta}
\begin{tikzcd}[column sep=2cm]
\pi^*\mathfrak{M} \otimes \mathfrak{N}_{23} \otimes \mathfrak{N}_{12} 
\ar[r, "\text{composition}"'] 
\ar[rr, bend left=15, "\beta", dashed]
& 
\underline{\Hom}(\F_{23} \boxtimes_{\A_2} \F_{12}, \pi_{13}^*\SB)
&\mathfrak{N}_{13} \otimes \Fus,
\ar[l, "\text{composition}"]
\end{tikzcd}
\end{equation}
is smooth.
%
%
We assume from now on that such a smoothing structure is given, in order to stay in the setting of smooth bundle gerbes.

\medskip

The line bundle $\mathfrak{M}$ determines a super bundle 2-gerbe $\FusTwoGrb(\SB)$ over $X \times X$, which we call the \emph{fusion 2-gerbe}.
Just as $\LagTwoGrb_X$, it is given in terms of the cover $PX$.
Its cover diagram \eqref{CoverDiagram} is trivial, in the sense that $V^{[m, n]} = PX^{[n]}$ for each $m$ and $n$.
In particular, its super bundle gerbe over $PX^{[2]}$ is trivial.
Using the notation of Definition~\ref{DefinitionSuper2Gerbe}, the other data are the following.
\begin{enumerate}
\item[{[1, 3]}]
Its super line bundle over $PX^{[3]}$ is $\M$, defined in \eqref{DefinitionLFusionLineBundleS}.
\item[{[2, 4]}]
Its morphism $\mu$ is the identity on $\M$.
\item[{[1, 4]}]
 The isomorphism
\begin{equation}
\label{AssociatorFusionGerbe}
\tilde{\alpha} : \mathfrak{M}_{124} \otimes \mathfrak{M}_{234} \longrightarrow \mathfrak{M}_{134} \otimes \mathfrak{M}_{123},
\end{equation}
over $PX^{[4]}$ is defined as the cospan
\begin{equation*}
  \begin{tikzcd}
    \M_{124} \otimes \M_{234} \ar[r] \ar[rr, bend left=20, dashed, "\tilde{\alpha}"]&  \M^{(3)} & \M_{134} \otimes \M_{123}, \ar[l]
  \end{tikzcd}
\end{equation*}
involving the line bundle
\begin{equation*}
  \M^{(3)} = \underline{\Hom}(\SB_{34} \boxtimes_{\A_3} \SB_{23} \boxtimes_{\A_2} \SB_{12}, \SB_{14})
\end{equation*}
over $PX^{[4]}$ and the canonical isomorphisms
\begin{equation*}
\begin{aligned}
&\M_{124} \otimes \M_{234} \longrightarrow \M^{(3)}, \qquad & &\Upsilon_{124} \otimes \Upsilon_{234} \longmapsto \Upsilon_{124} \circ (\id_{\SB_{12}} \boxtimes \Upsilon_{234}) \\
&\M_{134} \otimes \M_{123} \longrightarrow \M^{(3)}, \qquad & &\Upsilon_{134} \otimes \Upsilon_{123} \longmapsto \Upsilon_{134} \circ (\Upsilon_{123}\boxtimes\id_{\SB_{34}}).
\end{aligned}
\end{equation*}
\end{enumerate}

\begin{theorem}
 The above data constitute the structure of a super bundle 2-gerbe.
\end{theorem}

\begin{proof}
It is only necessary to verify the cocycle condition \eqref{CoherenceAlpha}.
This is checked entirely analogously to the case of the cocycle condition for the isomorphism of $\alpha$ of the Lagrangian 2-gerbe $\LagTwoGrb_X$; see the proof of Thm.~\ref{TheoremCoherenceLagTwoGrb}.
\end{proof}

The fusion 2-gerbe $\FusTwoGrb(\SB)$ relevant to answer the question whether a given loop space spinor bundle admits a (smooth) fusion product according to Definitions~\ref{DefinitionFusionProduct} and \ref{DefinitionSmoothFusionProduct}.

\begin{theorem}
\label{TheoremFusionProductFusion2Gerbe}
If the spinor bundle $\SB$ admits a fusion product, then the fusion 2-gerbe $\FusTwoGrb(\mathfrak{S})$ is trivializable.
  Conversely, if $\FusTwoGrb(\mathfrak{S})$ admits a trivialization, then there exists a line bundle $\T$ over $LX$ together with a smooth fusion product for the modified spinor bundle $\mathfrak{S}^\prime = \mathfrak{S} \otimes \T$.
\end{theorem}

\begin{proof}
By definition, a fusion product for $\mathfrak{S}$ is an even section $\Upsilon$ of $\M$ such that
\begin{equation*}
  \Upsilon_{134} \circ (\mathbf{1} \otimes \Upsilon_{123}) = \Upsilon_{124} \circ (\Upsilon_{234} \otimes \mathbf{1}).
\end{equation*}
This identity means precisely that $\Upsilon_{134} \otimes \Upsilon_{123}$ is sent to $\Upsilon_{124} \otimes \Upsilon_{234}$ under the associator \eqref{AssociatorFusionGerbe}, hence $\Upsilon$ determines a trivialization of $\FusTwoGrb(\mathfrak{S})$.
If $\Upsilon$ is not smooth, this is only a continuous trivialization, but any bundle 2-gerbe is continuously trivial if and only if it is smoothly trivial.

Conversely, suppose we are given a trivialization $\mathcal{t} = (\mathcal{T}, \mathfrak{t}, t)$ of $\FusTwoGrb(\SB)$, given by a bundle gerbe $\mathcal{T}$ over $PX$, a morphism of bundle gerbes $\mathfrak{t} : \mathcal{T}_2 \to \mathcal{T}_1$ and a 2-morphism
\begin{equation*}
\begin{tikzcd}[column sep=1.5cm]
  \mathcal{T}_3
   \ar[r, "\mathfrak{M}"] \ar[d, "\mathfrak{t}_{23} \otimes 1"'] 
  & \mathcal{T}_{3} 
    \ar[dd, " \mathfrak{t}_{13}"] 
    \ar[ddl, Rightarrow, "t"', shorten=0.8cm]
  \\
  \mathcal{T}_2 
  \ar[d, "1 \otimes \mathfrak{t}_{12}"']
  &
  \\
  \mathcal{T}_1 
  \ar[r, equal]
  & \mathcal{T}_1,
\end{tikzcd}
\end{equation*}
where the super line bundle $\mathfrak{M}$ is viewed as an automorphism of the super bundle gerbe $\mathcal{T}_3$ over $PX^{[3]}$.
(This uses the fact that for super bundle gerbes $\tilde{\mathcal{G}}$, $\mathcal{G}$ over a manifold $M$, the groupoid of isomorphisms $\textsc{Iso}(\tilde{\mathcal{G}}, \mathcal{G})$ is a torsor for the groupoid $\sLine(M)$ of super line bundles over $M$; moreover if $\tilde{\mathcal{G}} = \mathcal{G}$, then the identity isomorphism of $\mathcal{G}$ provides a base point, and hence an equivalence $\textsc{Iso}(\mathcal{G}, \mathcal{G}) = \sLine(M)$.)

Since the evaluation-at-zero map $\mathrm{ev}_0 : PX \to X$ is a homotopy equivalence, there is a super bundle gerbe $\mathcal{T}^\prime$ over $X$ such that $\mathcal{T} \cong \mathrm{ev}_0^* \mathcal{T}^\prime$.
Upon possibly modifying the isomorphism $\mathfrak{t}$ and the 2-isomorphism $t$ from the definition of the trivialization $\mathcal{t}$ (thus obtaining a new trivialization of $\FusTwoGrb(\SB)$), we may assume that $\mathcal{T}$ is of this form.
Under this assumption, the pullbacks $\mathcal{T}_1$ and $\mathcal{T}_2$ of $\mathcal{T}$ to $PX^{[2]}$ are actually \emph{equal}, hence $\mathfrak{t}$ is an \emph{auto}morphism of super bundle gerbes over $PX^{[2]}$ and hence may be identified with a line bundle $\mathfrak{T}$ over $PX^{[2]}$.
Under this identification, the 2-morphism $t$ is just a grading-preserving isomorphism of super line bundles
\begin{equation*}
  \tau : \mathfrak{M} \otimes \mathfrak{T}_{13} \longrightarrow \mathfrak{T}_{23} \otimes \mathfrak{T}_{12}.
\end{equation*}
Since $PX^{[2]}$ is smoothly homotopy equivalent to $LX$, $\mathfrak{T}$ is the restriction of a line bundle over $LX$, which we also denote by $\mathfrak{T}$.
%
%
Inserting the definition \eqref{DefinitionLFusionLineBundleS} of the line bundle $\mathfrak{M}$,
$\tau$ corresponds to an even element $\Upsilon$ of
\begin{equation*}
   \underline{\Hom}\bigl((\mathfrak{S}_{23} \otimes \mathfrak{T}_{23}) \boxtimes_{\A_2}(\mathfrak{S}_{12} \otimes \mathfrak{T}_{12}), \mathfrak{S}_{13} \otimes \mathfrak{T}_{13}\bigr);
\end{equation*}
see Example~\ref{ExampleTensorProduct}.
But this is precisely a fusion product for $\mathfrak{S} \otimes \T$, where the coherence of $\tau$ provides associativity of $\Upsilon$ over $P X^{[4]}$.
\end{proof}

\begin{theorem}
\label{TheoremIsoFusLag}
For any loop space spinor bundle $\SB$ with smoothing structure, there is a canonical isomorphism of super bundle 2-gerbes
\begin{equation*}
\mathcal{h}_{\SB} : \LagTwoGrb_X \longrightarrow \FusTwoGrb(\SB).
\end{equation*}
\end{theorem}

\begin{proof}
Both bundle gerbes are presented in terms of the cover $PX$ of $X\times X$, so that we can take $Z = PX$ as the common cover for $\mathcal{h}_{\SB}$. 
Moreover, as the cover diagram of $\FusTwoGrb(\SB)$ is trivial, we may take the cover diagram \eqref{CoverDiagramMorphism} of $\mathcal{h}$ to coincide with the cover diagram \eqref{CoverDiagramLag}.
The bundle gerbe $\mathcal{H}$ over $Z = PX$ is then trivial trivial; in other words, both $\mathfrak{H}$ and $\eta$ in \eqref{TwoMorphismDiagram} are trivial.

All remaining data of $\mathcal{h}_{\SB}$ have been already considered above; they are the following.
\begin{enumerate}
\item[{[1, 2]}]
The line bundle $\mathfrak{N}$ over $\Lag^{[1, 2]} = \Lag$ is just the line bundle \eqref{AssociatedSmoothnessLine}, which acquires its smooth structure from the smoothing structure of $\SB$.
\item[{[2, 2]}]
The bundle isomorphism $\nu$ over $\Lag^{[2, 2]}$ is the bundle homomorphism \eqref{CanonicalCompositionMap}.
Its smoothness is part of the assumption of the smoothing structure of $\SB$.
\item[{[2, 3]}]
The bundle isomorphism $\beta$ over $\Lag^{[2, 3]}$ is given by \eqref{CanonicalIsomorphismBeta}.
\end{enumerate}

We have to check three conditions.
The diagram \eqref{CoherenceEta} becomes
\begin{equation*}
\begin{tikzcd}
 \mathfrak{N}^3 \otimes \Pf^{23} \otimes \Pf^{12} 
 \ar[r, "1 \otimes \lambda"] 
 \ar[d, "\nu^{23} \otimes 1"']
 	&
	\mathfrak{N}^3 \otimes \Pf^{13}
	\ar[dd, "\nu^{13}"] 
	\\
\mathfrak{N}^2 \otimes \Pf^{12} 
\ar[d, "\nu^{12}"']
&
\\
\mathfrak{N}^1
\ar[r, equal]
&
\mathfrak{N}^1,
\end{tikzcd}
\end{equation*}
which is obviously commutative, as fiberwise, $\nu$ is just the composition homomorphism
\begin{equation*}
  \underline{\Hom}(\F_{L^\prime}, \SB_\gamma) \otimes \underline{\Hom}(\F_L, \F_{L^\prime}) \longrightarrow \underline{\Hom}(\F_L, \SB_\gamma).
\end{equation*}
The diagram \eqref{CoherenceBetaTilde} collapses to
\begin{equation*}
\begin{tikzcd}[column sep=1cm]
  \mathfrak{N}_{13}^2 \otimes \Fus^2 \otimes \Pf_{23} \otimes \Pf_{12}
  \ar[r, "\beta^2\otimes 1 \otimes 1"]
  \ar[d, "1 \otimes 1 \otimes \tilde{\mu}"']
  	&  
	\mathfrak{M}^2 \otimes \mathfrak{N}_{23}^2 \otimes \mathfrak{N}_{12}^2 \otimes \Pf_{23} \otimes \Pf_{12}
		\ar[d, equal]
  \\
  \mathfrak{N}_{13}^2 \otimes \Pf_{13} \otimes \Fus^1
  	\ar[d, "\nu_{13}\otimes 1"']
   &
     	\mathfrak{M}^2 \otimes \mathfrak{N}_{23}^2 \otimes \Pf_{23} \otimes \mathfrak{N}_{12}^2  \otimes \Pf_{12}
		\ar[d, "1\otimes \nu_{23} \otimes \nu_{12}"]
  \\
\mathfrak{N}_{13}^1  \otimes \Fus^1
	\ar[r, "1 \otimes 1 \otimes \beta^1"]
	 &
	  \mathfrak{M}^1  \otimes \mathfrak{N}_{23}^1 \otimes \mathfrak{N}_{12}^1.
\end{tikzcd}
\end{equation*}
Its commutativity is straight forward to check.
Finally, the commutativity of \eqref{CocycleConditionb} is easy to check as well.
\end{proof}

\subsection{The Lagrangian gerbe and the spin condition}
\label{SectionSpinCondition}

Let $X$ be an oriented Riemannian manifold and let $\LagTwoGrb_X$ be the corresponding Lagrangian 2-gerbe, as introduced in \S\ref{SubSectionLagrangian2Gerbe}.
As to any super bundle 2-gerbe, one may form a corresponding $\Z_2$-bundle gerbe $\orclass(\LagTwoGrb_X)$, called the \emph{orientation gerbe}, see Definition~\ref{DefinitionOrientationGerbe}.
The purpose of this section is to prove the following theorem.

\begin{theorem}
\label{ThmLagTwoGerbSpin}
The orientation gerbe $\orclass(\LagTwoGrb_X)$ admits a trivialization if and only if $X$ admits a spin structure.
\end{theorem}

We start with the following definition, see \cite[\S2]{WaldorfSpinString}.

\begin{definition}[Loop space orientation bundle]
The \emph{orientation bundle} of the loop space $LX$ of a Riemannian manifold $X$ is
\begin{equation*}
\hat{L}X \defeq L\SO(X) \times_{L\SO(d)} \Z_2,
\end{equation*}
a principal $\Z_2$-bundle over $LX$. 
Here the homomorphism $L\SO(d) \to \Z_2$ is the one labeling the connected components.
\end{definition}

The orientation bundle $\hat{L}X$ comes with a \emph{fusion product} map
\begin{equation}
\label{FusionProductOrientationBundle}
\begin{aligned}
\hat{\Upsilon} : \hat{L}X_{23} \otimes_{\Z_2} \hat{L}X_{12} &\longrightarrow \hat{L}X_{13} \\
[g_2 \fuse g_3, \epsilon_{23}] \otimes [g_1 \fuse g_2, \epsilon_{12}] &\longmapsto [g_1 \fuse g_3, \epsilon_{23} + \epsilon_{12}]
\end{aligned}
\end{equation}
over $PX^{[3]}$, coming from the fusion of loops (here, as usual, we use Notation~\ref{NotationSubmersion}). 
The existence of the fusion product $\hat{\Upsilon}$ allows to degress $\hat{L}X$ to a $\Z_2$-bundle gerbe $\mathcal{O}_X$ on $X \times X$, which we call the \emph{loop space orientation gerbe}.
Explicitly, this gerbe is given in terms of the cover $PX$ (with the endpoint projections), its principal $\Z_2$-bundle over $PX^{[2]}$ is the pullback of $\hat{L}X$ along \eqref{CupMap} and its bundle gerbe product is $\hat{\Upsilon}$.
\begin{equation*}
  \mathcal{O}_X =   \left[
\begin{tikzcd}
&
\fuse^*\hat{L}X \ar[d]
  & \hat{\Upsilon} \ar[d, -, dotted]  
  \\
  PX \ar[d]& PX^{[2]}  \ar[l, shift left=1mm] \ar[l, shift right=1mm]&  PX^{[3]} \ar[l, shift left=2mm] \ar[l, shift right=2mm] \ar[l]
  \\
  X
\end{tikzcd}
\right]
\end{equation*}

Let $\Lift_{\SO(X)}$ be the spin lifting gerbe over $X$, i.e., the obstruction gerbe to lifting the structure group $\SO(d)$ of the frame bundle $\SO(X)$ to the spin group $\Spin(d)$, which can be depicted as follows, see \S\ref{SectionBundleGerbes}.
\begin{equation*}
  \Lift_{\SO(X)} =   \left[
\begin{tikzcd}
  \Spin(d) \ar[d] & 
  \\
   \SO(d) & \delta^*\Spin(d) \ar[d, dashed] \ar[ul, dashed] 
  & \substack{\text{group} \\ \text{multiplication}} \ar[d, -, dotted] 
  & \substack{\text{associativity} \\ \text{of group} \\ \text{multiplication} \\ \text{in } \Spin(d)} \ar[d, -, dotted] 
  \\
  \SO(X) \ar[d]& \SO(X)^{[2]}  \ar[ul, "\delta"] \ar[l, shift left=1mm] \ar[l, shift right=1mm]&  \SO(X)^{[3]} \ar[l, shift left=2mm] \ar[l, shift right=2mm] \ar[l] & \SO(X)^{[4]} \ar[l, shift left=1mm] \ar[l, shift right=1mm] \ar[l, shift left=3mm] \ar[l, shift right=3mm]
  \\
  X
\end{tikzcd}
\right]
\end{equation*}
We then have the following result.

\begin{lemma}
\label{LemmaOrGerbeLiftingGerbe}
There is a canonical isomorphism
\begin{equation*}
  \mathfrak{h} : \mathcal{O}_X \longrightarrow \pr_2^*\Lift_{\SO(X)} \otimes \mathrm{pr}_1^*\Lift_{\SO(X)}^{-1}.
\end{equation*}
of $\Z_2$-bundle gerbes over $X \times X$. 
In particular, the loop space orientation gerbe is trivial if and only if $X$ is spin.
\end{lemma}

A section $s$ of $\hat{L}X$ is called \emph{fusion preserving} if the identity $\hat{\Upsilon}(s_{23} \otimes s_{12}) = s_{13}$ holds over $PX^{[3]}$.
It was shown by Stolz and Teichner \cite[Thm.~9]{StolzTeichnerSpinorBundle} that fusion-preserving sections of $\hat{L}X$ are in 1-1 correspondence with spin structures on $X$; see also \cite[Thm.~2.5]{WaldorfSpinString}.
As isomorphism classes of trivializations of $\Lift_{\SO(X)}$ are in 1-1 correspondence with spin structures on $X$ (Thm.~\ref{ThmLiftingGerbe}), this follows directly from the above proposition, as a fusion preserving sections is essentially the same as trivializations of $\mathcal{O}_X$.
As in \cite{StolzTeichnerSpinorBundle} a different description of the orientation bundle $\hat{L}X$ is used, we give a proof of Lemma.~\ref{LemmaOrGerbeLiftingGerbe} in the current setting.

\begin{proof}[of Lemma~\ref{LemmaOrGerbeLiftingGerbe}]
The isomorphism $\mathfrak{h}$ will be a span of refinements, along the two legs of the common cover
\begin{equation}
\label{CommonRefinementfrakh}
\begin{tikzcd}[column sep=2cm]
 PX & \ar[l, "\substack{\text{footpoint} \\ \text{projection}}"'] \ar[r, "\substack{\text{endpoint} \\ \text{evaluation}}"] P\SO(X) & \SO(X) \times \SO(X).
\end{tikzcd}
\end{equation}
To begin with, we have to construct an isomorphism
\begin{equation}
\label{IsoVarphiOrBundleToSpin}
  \eta :  \hat{L}X \longrightarrow \pr_2^*\delta^*\Spin(d) \otimes \pr_1^*\delta^*\Spin(d)
\end{equation}
of principal $\Z_2$-bundles over $P\SO(X) \times_{X \times X} P\SO(X)$ (where we suppress in notation the pullback along the legs of \eqref{CommonRefinementfrakh}).

Let $(h_1, h_2) \in P\SO(X) \times_{X \times X} P\SO(X)$, with footpoint curves $(\gamma_1, \gamma_2) \in PX^{[2]}$ (we emphasize that $h_1$, $h_2$ do not necessarily have the same end points in $P\SO(X)$, but the end points of their footpoint curves coincide in $X$).
Let moreoever $(g_1, g_2) \in P\SO(X)^{[2]}$, where the fiber product is taken over $X\times X$, so that $g_1$ and $g_2$ have the same end points.
Then 
\begin{equation*}
q_i \defeq g_i^{-1}h_i , \qquad i =1, 2,
\end{equation*}
are two elements of $P \SO(d)$.
Let $\tilde{q}_i \in P\Spin(d)$, $i=1, 2$, be lifts of $q_1$ and $q_2$.
As $g_1$ and $g_2$ have matching end points, we have for $t \in \{0, \pi\}$ that
\begin{equation*}
q_2(t)^{-1} q_1(t) = h_2(t)^{-1} g_2(t)  g_1(t)^{-1} h_1(t) = h_2(t)^{-1} h_1(t) = \delta\bigl( \gamma_1(t), \gamma_2(t) \bigr),
\end{equation*}
where $\delta$ is the difference map \eqref{DifferenceMap}.
Hence $\tilde{q}_2(t)^{-1} \tilde{q}_1(t)$ defines an element of $\delta^*\Spin(d)$ at $(\gamma_1(t), \gamma_2(t))$.
For each $i=1, 2$, there are precisely two choices of lift $\tilde{q}_i$ of $q_i$, which differ by a sign.
Hence
\begin{equation*}
\tilde{\eta}(g_1, g_2, h_1, h_2) \defeq \tilde{q}_2(\pi)^{-1} \tilde{q}_1(\pi) \otimes_{\Z_2} \tilde{q}_2(0)^{-1} \tilde{q}_1(0)
\end{equation*}
is a well-defined element of $\pr_2^*\delta^*\Spin(d) \otimes_{\Z_2} \pr_1^*\delta^*\Spin(d)$ at $(h_1(0), h_2(0); h_2(\pi), h_2(\pi))$.
The map \eqref{IsoVarphiOrBundleToSpin} at a point $(h_1, h_2)$ is now given by
\begin{equation*}
\eta([g_1 \fuse g_2, \epsilon]) \defeq (-1)^{\epsilon} \tilde{\eta}(g_1, g_2, h_1, h_2).
\end{equation*}
As the fibers of the covering $\Spin(d) \to \SO(d)$ are discrete, $\tilde{\eta}$ only depends on the path component of $(h_1, h_2)$ in the fibers of $P\SO(X) \to \SO(X) \times \SO(X)$ and the path component of $g_1 \fuse g_2$ in the fibers of $P\SO(X)^{[2]} \to PX^{[2]}$.
Therefore $\eta$ is well defined.
It is moreover straight forward to verify that $\eta$ intertwines the bundle gerbe products.
\end{proof}

In view of Lemma~\ref{LemmaOrGerbeLiftingGerbe}, Thm.~\ref{ThmLagTwoGerbSpin} now follows directly from the following result.

\begin{lemma}
  The orientation gerbe $\orclass(\LagTwoGrb_X)$ of the Lagrangian 2-gerbe is canonically isomorphic to the loop space orientation gerbe $\mathcal{O}_{X}$.
\end{lemma}

\begin{proof}
Observe first that after fixing a given Lagrangian $L \in \Lag_0$, there is a canonical isomorphism
\begin{equation}
\label{DefinitionVarphi}
  \varphi : \hat{L} X \longrightarrow \orclass(\LagGrb_{LX}), \qquad [g, \epsilon] \longmapsto [gL, \epsilon + |L|]
\end{equation}
of principal $\Z_2$-bundles over $LX$, where $\orclass(\LagGrb_{LX})$ is the orientation bundle of $\LagGrb_{LX})$, see Definition~\ref{DefinitionOrientationBundle}.
Here $|L| = 0$ if $L \in \Lag_0^+$ and $|L| = 1$ otherwise (see Remark~\ref{RemarkPreferredConnectedComponent}).
This additional sign ensures that $\varphi$ does not depend on the choice of $L$.

We just need to verify that this isomorphism intertwines the fusion product maps over $PX^{[3]}$.
For $\orclass(\LagGrb_{LX})$, this is the map 
\begin{equation*}
\begin{tikzcd}[column sep=2cm]
  \orclass(\LagGrb_{23}) \otimes \orclass(\LagGrb_{12}) \cong \orclass(\LagGrb_{23} \otimes \LagGrb_{12}) 
  \ar[r, "\orclass(\mathfrak{m})"] 
  &
   \orclass(\LagGrb_{13}).
\end{tikzcd}
\end{equation*}
Given a lift $(g_1, g_2, g_3) \in P\SO(X)^{[3]}$ of $(\gamma_1, \gamma_2, \gamma_3) \in PX^{[3]}$, any element of the left hand side can be taken to be of the form $[g_{23}L, \epsilon_{23}] \otimes [g_{12}L, \epsilon_{12}]$, while any element of the right hand side can be taken to have the form $[g_{13}L, \epsilon_{13}]$ (here $g_{ij} = g_i \fuse g_j$).
Suppose that
\begin{equation*}
\orclass(\mathfrak{m}) \bigl([g_{23}L, \epsilon_{23}] \otimes [g_{12}L, \epsilon_{12}]\bigr) = [g_{13}L, \epsilon_{13}].
\end{equation*}
By definition of $\orclass(\mathfrak{m})$, this implies that the parity of the fusion line $\Fus(g_{23}L, g_{12}L, g_{13}L)$ equals $\epsilon_{23} + \epsilon_{12} + \epsilon_{13}$ (here, as usual, $g_{ij} = g_i \fuse g_j$).
As 
\begin{equation*}
\Fus(L, L, L) \longrightarrow \Fus(g_{23}L, g_{12}L, g_{13}L), \qquad  \Upsilon \longmapsto \Lambda_{g_{13}} \circ \Upsilon \circ (\Lambda_{g_{23}} \boxtimes \Lambda_{g_{12}})^*
\end{equation*}
is a grading preserving isomorphism, both fusion lines have the same parity.
As the parity of $\Fus(L, L, L)$ is even if $L \in \Lag_0^+$ and odd otherwise, we get the relation
\begin{equation}
\label{EpsilonRelation}
\epsilon_{23} + \epsilon_{12}  + \epsilon_{13} = |L|.
\end{equation}
A simple calculation shows now that the map $\varphi$ given in \eqref{DefinitionVarphi} intertwines the fusion products.
%
\end{proof}

\begin{remark}
Any loop space spinor bundle $\SB$ gives a section of $\orclass(\LagGrb_{LX})$ given by $s(\gamma) = [L, 0]$, where $L \in \Lag_\gamma$ is is such that the isomorphism $\F_L \cong \SB_\gamma$ is even.
This section is fusion-preserving if and only if the line bundle $\mathfrak{M}$ over $PX^{[3]}$ defined in \eqref{DefinitionLFusionLineBundleS} is even.
We therefore see that if a spinor bundle on loop space exists, then $X$ is spin if and only if a grading preserving fusion product exists projectively.
\end{remark}

\subsection{Lifting 2-gerbes}
\label{SectionCS2gerbes}

Let $G$ be a Lie group.
Recall that a \emph{multiplicative bundle gerbe} over $G$ is consists of a bundle gerbe $\mathcal{G}$ over $G$ together with an isomorphism
\begin{equation}
\label{MorphismMultGerbe}
  \mathfrak{m} : \mathrm{pr}_1^*\mathcal{G} \otimes \mathrm{pr}_2^*\mathcal{G} \longrightarrow \mathrm{m}^*\mathcal{G}
\end{equation}
of bundle gerbes over $G \times G$ (called the \emph{multiplicative structure}) and an associator 2-morphism over $G \times G \times G$ satisfying certain higher cocycle conditions over $G \times G \times G \times G$; see \cite[\S5]{CJMSW} or \cite{WaldorfMultiplicative}.
Here $\mathrm{m} : G \times G \to G$ is the group multiplication map of $G$.

\begin{construction}
\label{ConstructionMultGerbsFromCentExt}
One way multiplicative bundle gerbes over $G$ arise is from central extensions of the loop group $LG$.
Let
\begin{equation*}
\U(1) \longrightarrow \widetilde{LG} \longrightarrow LG
\end{equation*}
be a central extension by $\U(1)$, meaning that $\smash{\widetilde{LG}}$ is a central extension of Fr\'echet Lie groups which is locally trivial principal as $\U(1)$-bundle. 
Recall \cite[Def.~3.4]{WaldorfSpinString} that a \emph{(multiplicative) fusion product} on $\smash{\widetilde{LG}}$  is a group homomorphism
\begin{equation}
\label{FusionProduct}
 \mu : \widetilde{LG}_{23} \otimes_{\U(1)} \widetilde{LG}_{12} \longrightarrow \widetilde{LG}_{13}
\end{equation}
of Lie groups over $PG^{[3]}$ that is the identity on the copy of $\U(1)$ canonically contained on both sides and that is associative over $PG^{[4]}$.
Here $\smash{\widetilde{LG}_{ij}}$ denotes the pullback of the central extension along the map $PG^{[3]} \to PG^{[2]} \to LG$, where the second arrow is the cup map \eqref{CupMap} and the first map is the projection onto the $i$-th and $j$-th factor.

Any central extension with a fusion product yields a multiplicative bundle gerbe $\mathcal{G}$ over $G$, 
given in terms of the cover $P_e G \to G$, where $P_eG$ is the group of paths in $G$ starting at the identity element $e \in G$, together with the endpoint projection to $G$.
Its line bundle over $P_eG^{[2]}$ is the pullback of the line bundle $\widetilde{LG} \times_{\U(1)} \C$ over $LG$ associated to the central extension along the cup map $P_eG^{[2]} \to LG$, and its bundle gerbe product is given by the fusion product \eqref{FusionProduct}.
The multiplicative structure $\mathfrak{m}$ of $\mathcal{G}$ is the image under \eqref{InclusionRefinements} of the refinement of bundle gerbes along the map
\begin{equation}
\label{RefinementForMmultGerbe}
\begin{tikzcd}
  \pr_1^*PG \times_{G \times G} \pr_2^*P_eG \longrightarrow \mathrm{m}^*PG, \qquad (q_1, q_2) \mapsto q_1q_2,
\end{tikzcd}
\end{equation}
whose line bundle isomorphism is given by pointwise group multiplication in $\widetilde{LG}$.
By associativity of the group multiplication, the pullbacks of this refinement assemble to a commutative diagram in the category $\textsc{Gerb}_{\mathrm{ref}}(G \times G \times G)$ of bundle gerbes of $G \times G \times G$ with refinements, hence the associator is provided by the functor \eqref{InclusionRefinements} and the corresponding cocycle condition is automatic.
%
%
\end{construction}

\begin{remark}
Restricting a fusion product to the diagonal $\{(q, q, q) \mid q \in PX\} \subset PX^{[3]}$, we obtain a trivialization $LG_{q \fuse q} \cong \U(1)$ and hence a smooth group homomorphism
\begin{equation*}
  i : PG \longrightarrow \widetilde{LG},
\end{equation*}
called \emph{fusion factorization}.
Conversely, given any smooth group homomorphism $PG \to \smash{\widetilde{LG}}$ with the property that $i(q)$ lies over $q \fuse q$, we obtain a fusion product by the formula
\begin{equation}
\label{GeneralFormulaFusionProduct}
  \mu (U, V) = U i(q^\prime)^{-1}V,
\end{equation}
whenever $U, V \in \widetilde{LG}$ lie over $q^\prime \fuse q^{\prime\prime}$, respectively $q\fuse q^\prime$.
If $G$ is simply connected and semisimple, there exists a unique fusion factorization \cite[Thm.~3.3.5]{LudewigWaldorfLoopGroups}, and hence a unique fusion product.
\end{remark}

\begin{construction}
\label{ConstructionMultGerbsFromCentExtDbl}
As our investigations are centered around the free loop space instead of the based loop space, it will be relevant below that a central extension $\widetilde{LG}$ also defines a multiplicative bundle gerbe $\mathcal{G}^{\mathrm{dbl}}$ over $G \times G$ with cover $PG \to G \times G$, defined in a completely analogous fashion to Construction~\ref{ConstructionMultGerbsFromCentExt}.
We call $\mathcal{G}^{\mathrm{dbl}}$ the \emph{doubled} multiplicative bundle gerbe determined by $\smash{\widetilde{LG}}$.
If $G$ is semisimple and simply connected, there exists a canonical isomorphism
\begin{equation}
\label{DoubleIsoGrb}
  \mathcal{G}^{\mathrm{dbl}} \cong \mathrm{pr}_2^*\mathcal{G} \otimes \mathrm{pr}_1^*\mathcal{G}^{-1}
\end{equation}
of multiplicative bundle gerbes over $G \times G$; see \cite[\S5.2 and \S5.3]{Bigerbes}.
\end{construction}

\begin{example}
\label{ExampleImplementerGerbe}
Consider $G = \Spin(d)$, where $d \geq 5$.
As this is a simply connected and simple Lie group, any central extension of $L\Spin(d)$ admits a unique (multiplicative) fusion product \cite[Thm.~3.3.5]{LudewigWaldorfLoopGroups}.
A particular example of such a central extension is the \emph{implementer extension} \eqref{DiagramImplementerExtension}.
Applying Construction~\ref{ConstructionMultGerbsFromCentExt} to this central extension yields a multiplicative bundle gerbe $\mathcal{G}_{\Imp}$ over $\Spin(d)$.
Construction~\ref{ConstructionMultGerbsFromCentExtDbl} yields multiplicative bundle gerbe $\mathcal{G}_{\Imp}^{\mathrm{dbl}}$ over $\Spin(d) \times \Spin(d)$.

We now describe the fusion factorization and the fusion product explicitly for this example.
By the characterization \eqref{ImplementerExtension} of the  extension $\Imp_L$ of $\O_{\res}(H_0)$, the implementer extension of $L\Spin(d)$ can be written as the double pullback 
\begin{equation*}
  \begin{tikzcd}
    \widetilde{L\Spin(d)} \ar[r, dashed] \ar[d, dashed] & \Imp_L \ar[r] \ar[d] & \U(\F_L) \ar[d] \\
    L\Spin(d) \ar[r] & \O_{\res}(H_0) \ar[r] & \Aut(\CC_0).
  \end{tikzcd}
\end{equation*}
We choose a Lagrangian $L \in \Lag_0^+$ and fix an even isomorphism $\F_L \cong L^2(\A_0)$ (see Remark~\ref{RemarkPreferredConnectedComponent} and Corollary~\ref{CorollaryFockSpacesIsomorphic}).
Then under this isomorphism, the fusion factorization $i$ is given by Haagerups canonical implementation (see Remark~\ref{RemarkL2phi}); explicitly
\begin{equation*}
  i(\tilde{q}) = L^2(\Cl_q),
\end{equation*}
where $q \in P\SO(d)$ is the image of $\tilde{q} \in P\Spin(d)$.
Indeed, $i$ is characterized as the map
\begin{equation*}
  \begin{tikzcd}
  P\Spin(d) \ar[d, "\text{diagonal}"'] \ar[r]  \ar[dr, dashed, "i"]& \Aut(\A_0) \ar[dr, bend left=25, "\substack{\text{canonical} \\ \text{implementation}}"]\\
 P\Spin(d)^{[2]} \ar[dr, bend right=25, "\fuse"'] &  \widetilde{L\Spin(d)} \ar[r, ] \ar[d, ] & \U(L^2(\A_0)) \ar[d] \\
  &  L\Spin(d) \ar[r] & \Aut(\CC_0),
  \end{tikzcd}
\end{equation*}
provided by the universal property of the pullback, hence is continuous.
It then follows that, as a homomorphism between locally exponential Lie groups, $i$ is automatically smooth.
It must therefore be the unique fusion factorization (see \cite[Thm.~3.3.5]{LudewigWaldorfLoopGroups}).
In view of \eqref{GeneralFormulaFusionProduct}, the fusion product \eqref{GeneralFormulaFusionProduct} becomes
\begin{equation}
\label{FusionProductImplementers}
  \mu (U, V) = U L^2(\Cl_{q^\prime})^* V,
\end{equation}
where $U, V \in \Imp_L$ implement $q^\prime \fuse q^{\prime\prime}$, respectively $q \fuse q^{\prime}$.
Here and throughout, we use that the image of the map $P\Spin(d)^{[2]} \to P\SO(d)^{[2]} \to L\SO(d)$ is contained in the identity component of $L\SO(d)$.
Hence implementers for  $q_1 \fuse q_2$ are always even if $(q_1, q_2) \in P\SO(d)^{[2]}$ lifts to $(\tilde{q}_1, \tilde{q}_2) \in P\Spin(d)^{[2]}$.
\end{example}

The underlying $\U(1)$-bundle gerbe of a multiplicative bundle gerbes $\mathcal{G}$ over a Lie group $G$ can be viewed as a 2-group extension of $G$ by $B\U(1)$.
The obstruction for lifting the structure group of a principal $G$-bundle $P$ over a manifold to this 2-group is then a bundle 2-gerbe, defined as follows \cite{NikolausWaldorfLifting, WaldorfStringConnections}. 

\begin{definition}[Lifting 2-gerbe]
\label{DefinitionCS2gerbe}
Let $\mathcal{G}$ be a multiplicative bundle gerbe over $G$ and let $P$ be a principal $G$-bundle over a manifold $X$.
The corresponding \emph{lifting 2-gerbe} $\LiftTwoGrb_P$ is the bundle 2-gerbe over $X$ whose bundle over over $P^{[2]}$ is the pullback of $\mathcal{G}$ along the difference map 
\begin{equation*}
  \delta : P^{[2]} \longrightarrow G, \qquad (p_1, p_2) \longmapsto p_2^{-1} p_1,
\end{equation*}
and the bundle 2-gerbe product is given by the multiplicative structure of $\mathcal{G}$.
\end{definition}

\begin{remark}
The lifting 2-gerbe $\LiftTwoGrb_P$ is also often called \emph{Chern-Simons 2-gerbe} (see \cite{CJMSW, WaldorfStringConnections}, because its holonomy is related to the Chern-Simons action.
However, in this context, it seems conceptually more clear to call it lifting 2-gerbe.
\end{remark}

\begin{remark}
\label{RemarkStringGroup}
It follows from the results of \cite{LudewigWaldorf2Group} together with Thm.~\ref{ThmImpBasic} that the 2-group corresponding by the multiplicative bundle gerbe $\mathcal{G}_{\Imp}$ is precisely the string 2-group $\String(d)$, which is the central extension
\begin{equation*}
\begin{tikzcd}
	B\U(1) \ar[r] & \String(d) \ar[r] & \Spin(d).
\end{tikzcd}
\end{equation*}
Hence if $P$ is a principal $\Spin(d)$-bundle, the corresponding lifting gerbe $\LiftTwoGrb_P$ represents the geometric obstruction to lifting the structure group of $P$ to $\String(d)$.
The same is true when the central extension $\Imp$ is replaced by the dual extension $\Imp^*$.
The corresponding 2-groups are isomorphic via an isomorphism that is inversion on the center $B\U(1)$.
\end{remark}

Let $\widetilde{LG}$ be a central extension with a fusion product of the loop group $LG$ for a Lie group $G$.
Given a principal $G$-bundle $P$ over $X$, we can then consider its \emph{doubled} lifting 2-gerbe $\LiftTwoGrb^{\mathrm{dbl}}_P$, defined as the lifting 2-gerbe corresponding to $P \times P$, for the doubled multiplicative bundle gerbe $\mathcal{G}^{\mathrm{dbl}}$ determined by $\widetilde{LG}$ according to Construction~\ref{ConstructionMultGerbsFromCentExtDbl}.
It follows from \eqref{DoubleIsoGrb} that if $G$ is simply connected and semisimple, then we have an isomorphism
\begin{equation}
\label{DoublingIsomorphism}
\LiftTwoGrb_{P}^{\mathrm{dbl}} \cong \mathrm{pr}_2^*\LiftTwoGrb_P \otimes \mathrm{pr}_1^*\LiftTwoGrb_P^{-1},
\end{equation}
in particular $\LiftTwoGrb_{P}^{\mathrm{dbl}}$ is trivial if and only if $\LiftTwoGrb_P$ is trivial.

\medskip

For later use, we note that cover diagram of $\LiftTwoGrb_P^{\mathrm{dbl}}$ is
\begin{equation} 
\label{CoverDiagramCS}
\begin{aligned}
\begin{tikzcd}[column sep=0.7cm]
&
 	\vdots
	\ar[d]
 	\ar[d, shift left=2mm]
 	\ar[d, shift right=2mm]
	& 
	\vdots 
	\ar[d]
 	\ar[d, shift left=2mm]
 	\ar[d, shift right=2mm]
	&
	 	\vdots
	\ar[d]
 	\ar[d, shift left=2mm]
 	\ar[d, shift right=2mm]
			\\
&
 	W^{[2, 2]}
 	\ar[d, shift left=1mm]
 	\ar[d, shift right=1mm]
	&
		W^{[2, 3]}
			\ar[l]
 			\ar[l, shift left=2mm]
 			\ar[l, shift right=2mm]
 		\ar[d, shift left=1mm]
 		\ar[d, shift right=1mm]
		&
			W^{[2, 4]}
 			\ar[d, shift left=1mm]
 			\ar[d, shift right=1mm]
			\ar[l, shift left=1mm] 
		\ar[l, shift right=1mm]
			\ar[l, shift left=3mm] 
		\ar[l, shift right=3mm]
		& 
		\ar[l]
			\ar[l, shift left=2mm]
 			\ar[l, shift right=2mm]
 			\ar[l, shift left=4mm]
 			\ar[l, shift right=4mm]
			\cdots
			\\
&
	W^{[1, 2]} 
	\ar[d]
	& 
		W^{[1, 3]}
			\ar[l]
 			\ar[l, shift left=2mm]
 			\ar[l, shift right=2mm]
		\ar[d]
		&
			W^{[1, 4]}
 			\ar[d]
			\ar[l, shift left=1mm] 
		\ar[l, shift right=1mm]
			\ar[l, shift left=3mm] 
		\ar[l, shift right=3mm]
		&
		\ar[l]
			\ar[l, shift left=2mm]
 			\ar[l, shift right=2mm]
 			\ar[l, shift left=4mm]
 			\ar[l, shift right=4mm]
			\cdots
			\\
P \times P\ar[d]
 &
	(P\times P)^{[2]}
	\ar[l]
	&
		(P\times P)^{[3]}
			\ar[l]
 			\ar[l, shift left=2mm]
 			\ar[l, shift right=2mm]
			&
			(P\times P)^{[4]}
 			\ar[l, shift left=1mm]
 			\ar[l, shift right=1mm]
 			\ar[l, shift left=3mm]
 			\ar[l, shift right=3mm]
			&
			\ar[l]
			\ar[l, shift left=2mm]
 			\ar[l, shift right=2mm]
 			\ar[l, shift left=4mm]
 			\ar[l, shift right=4mm]
			\cdots
			\\
			X \times X
\end{tikzcd}
\end{aligned}
\end{equation}
where
\begin{equation*}
\begin{aligned}
W^{[m, n]} = \left\{ \left.
\left(
\begin{pmatrix} 
p_1 & \cdots & p_n \\ 
\tilde{p}_1 & \cdots & \tilde{p}_n
\end{pmatrix};
\begin{pmatrix} 
q^1_{ij} \\ \vdots \\ q^m_{ij}
\end{pmatrix}_{ i < j}
\right)
~\right|~ 
\begin{matrix}
q_{ij}^a(\pi) = \delta (p_i, p_j) \\
q_{ij}^a(0) = \delta (\tilde{p}_i, \tilde{p}_j) 
\end{matrix}
\right\}.\\
\text{with}\qquad
(\tilde{p}_1, \dots, \tilde{p}_n), (p_1, \dots, p_n) \in (P\times P)^{[n]}, 
\quad q^a_{ij} \in PG.
\end{aligned}
\end{equation*}
The line bundle over $W^{[2, 2]}$ is, essentially, the line bundle associated to central extension $\smash{\widetilde{L\Spin}}(d)$, the corresponding bundle gerbe product over $W^{[3, 2]}$ is the fusion product, and the third column is filled with the multiplicative structure $\mathfrak{m}$.

\subsection{Comparison of the Lagrangian gerbe and the lifting 2-gerbe}
\label{SectionComparisonLagCS}

Let $X$ be a spin manifold and let $\LagTwoGrb_X$ the corresponding Lagrangian 2-gerbe over $X \times X$, as discussed in \S\ref{SubSectionLagrangian2Gerbe}.
By Thm.~\ref{ThmLagTwoGerbSpin}, that $X$ is spin implies that $\LagTwoGrb_X$  has trivial orientation gerbe hence is isomorphic to an ordinary, ungraded bundle 2-gerbe.

On the other hand, consider the following lifting 2-gerbe:
Let $\mathcal{G}_{\Imp^*}$ be the multiplicative $\U(1)$-bundle gerbe over $\Spin(d)$ associated to the dual $\Imp^*$ of the implementer extension $\Imp$ of $L\Spin(d)$ through Construction~\ref{ConstructionMultGerbsFromCentExt}; see also Example~\ref{ExampleImplementerGerbe}.
According to Construction~\ref{ConstructionMultGerbsFromCentExtDbl}, this multiplicative bundle gerbe together with the principal $\Spin(d)$-bundle $\Spin(X)$ provided by the spin structure determines a bundle 2-gerbe $\LiftTwoGrb_{\Spin(X)}^{\mathrm{dbl}}$ over $X \times X$. 
In view of Remark~\ref{RemarkStringGroup}, this bundle 2-gerbe geometrically represents the obstruction to lift the structure group of $X$ to the string 2-group.
 
\begin{theorem}
\label{ThmIsoCSLag}
There exists a canonical isomorphism of bundle 2-gerbes 
\begin{equation}
\label{IsomorphismCSLag}
  \mathcal{h} : \LiftTwoGrb_{\Spin(X)}^{\mathrm{dbl}}  \longrightarrow \LagTwoGrb_X.
\end{equation}
\end{theorem}

The proof of this theorem will be given at the end of this section.
We first employ this theorem to prove the main result of this article.
To this end, we use the following notion of a string structure; see \cite[Definition~1.1.5]{WaldorfStringConnections}.

\begin{definition}
 A \emph{string structure} on a spin manifold $X$ of dimension $d \geq 5$ is a trivialization of the lifting 2-gerbe $\LiftTwoGrb_{\Spin(X)}$ over $X$ for $\mathcal{G}_{\Imp^*}$.
\end{definition}

Recall from Example~\ref{ExampleImplementerGerbe} that the Dixmier-Douady class of $\mathcal{G}_{\Imp}$ (and hence also of its dual) is a generator for $H^3(\Spin(d), \Z) \cong \Z$.
By Thm.~1.1.3  in\cite{WaldorfStringConnections}, admits a string structure in the above sense if and only if the characteristic class $\frac{1}{2}p_1(X)$ vanishes.
The following theorem repeats Thm.~B from the introduction.

\begin{theorem}
\label{TheoremFusionProduct}
  Let $X$ be an oriented Riemannian manifold of dimension $d \geq 5$.
  Then there exists a spinor bundle $\SB$ with fusion product over $X$ if and only if $X$ admits a string structure.
\end{theorem}

\begin{proof}
By \eqref{DoublingIsomorphism}, triviality of $\LiftTwoGrb_{\Spin(X)}$ is equivalent to triviality of the doubled lifting 2-gerbe $\smash{\LiftTwoGrb_{\Spin(X)}^{ \mathrm{dbl}}}$.
This implies that $X$ admits a string structure if and only if the doubled lifting 2-gerbe $\smash{\LiftTwoGrb_{\Spin(X)}^{\mathrm{dbl}}}$ for the dual implementer extension, given in \eqref{DataCSImp}, admits a trivialization.

Now, first suppose that there exists a spinor bundle $\SB$ with fusion product over $LX$.
Then, by Lemma~\ref{LemmaSmoothingStructureExistence}, there exists a smoothing structure for $\SB$ and a corresponding fusion 2-gerbe $\FusTwoGrb(\SB)$ (see \S\ref{SectionFusion2Gerbe}).
By Thm.~\ref{TheoremFusionProductFusion2Gerbe}, the existence of a fusion product implies that $\FusTwoGrb(\SB)$ admits a trivialization.
But by Thm.~\ref{TheoremIsoFusLag}, $\FusTwoGrb(\SB)$ is isomorphic to $\LagTwoGrb_X$.
Hence the existence of a spinor bundle with fusion product implies triviality of $\LagTwoGrb_X$, which in view of the isomorphism \eqref{IsomorphismCSLag} and the previous discussion implies that $X$ admits a string structure.

Conversely, suppose that $X$ admits a string structure.
In particular, $X$ is spin, and the string condition implies that the structure group $L\Spin(d)$ of the $LX$ can be lifted to its basic central extension \cite{KillingbackWorldSheet, WaldorfSpinString}. 
By Thm.~\ref{ThmLagLifting}, this implies that the Lagrangian gerbe $\LagGrb_{LX}$ of the loop space admits a trivialization, hence there exists a spinor bundle $\SB^\prime$ with a smoothing structure over $LX$ (Thm.~\ref{FunctorTrivSpin}).
But since $X$ admits a string structure, the doubled lifting 2-gerbe $\LiftTwoGrb_{\Spin(X)}^{\mathrm{dbl}}$ for the dual implementer extension and hence (by Theorems \ref{ThmIsoCSLag} \& \ref{TheoremIsoFusLag}) also the fusion 2-gerbe $\FusTwoGrb(\SB^\prime)$ are trivializable.
Therefore, by Thm.~\ref{TheoremFusionProductFusion2Gerbe}, there exists a line bundle $\mathfrak{T}$ over $LX$ such that $\SB = \SB^\prime \otimes \mathfrak{T}$ admits a smooth fusion product.
\end{proof}

We finish by proving Thm.~\ref{ThmIsoCSLag}.
Before we start the proof, we describe %
 the lifting 2-gerbe $\LiftTwoGrb_{P}^{\mathrm{dbl}}$ in more detail in the case that $G = \Spin(d)$ and the multiplicative bundle gerbe $\mathcal{G}$ is the bundle gerbe obtained by applying Construction~\ref{ConstructionMultGerbsFromCentExt} to the central extension of $L\Spin(d)$ \emph{dual} to the implementer extension, see Example~\ref{ExampleImplementerGerbe}.

We will make the following modification:
By definition, every loop group central extension with fusion product has two products: The group multiplication and the (partially defined) fusion product.
In relation to the cover diagram \eqref{CoverDiagramCS}, the fusion product lives over $W^{[3, 2]}$, while group multiplication lives over $W^{[2, 3]}$.
Below, we will replace the line bundle defining the central extension $\smash{\widetilde{L\Spin}}(d)$ by an isomorphic one, which roughly has the effect that fusion product and group multiplication swap places.
As in the case of the Lagrangian 2-gerbe, multiplication of implementers lives over $\Lag^{[3, 2]}$ while their fusion product lives over $\Lag^{[2, 3]}$, this will facilitate the construction of the isomorphism in \S\ref{SectionComparisonLagCS}.

Throughout, we fix a Lagrangian $L \in \Lag_0^+$, together with a cyclic and separating vector in $\F_L$, which allows to identify $\F_L \cong L^2(\A_0)$ (see Remark~\ref{RemarkPreferredConnectedComponent}).
For $(q, q^\prime) \in P\SO(d)^{[2]}$, we then set
\begin{equation*}
  \ImpLine(q, q^\prime) \defeq \underline{\Hom}\bigl(L^2(\A_0)_{q}, L^2(\A_0)_{q^\prime}\bigr),
\end{equation*}
where the subscript $q$ denotes a modification of the right action by the automorphism $\Cl_q$ of $\A_0$.
The adjoint operation is a $\U(1)$-anti-equivariant map identifying $\Imp(q \fuse q^\prime)^*$ with $\Imp(q^{-1} \fuse (q^\prime)^{-1})$, which implies that the $\U(1)$-equivariant map
\begin{equation}
\label{IsoVdash}
  \Imp_L(q \fuse q^\prime)^* \cong \Imp(q^{-1} \fuse (q^\prime)^{-1}) \longrightarrow \ImpLine(q, q^\prime), \qquad U \longmapsto L^2(\Cl_{q^\prime})U, 
\end{equation}
identifies $\ImpLine(q, q^\prime)$ with the complex line associated to the $\U(1)$-torsor dual to $\Imp_L(q \fuse q^\prime)$.
%
%
Replacing $\smash{\widetilde{L\Spin}}(d) \times_{\U(1)} \C$ by $\ImpLine$ in Example~\ref{ExampleImplementerGerbe}, we obtain a multiplicative line bundle gerbes $\mathcal{G}_{\Imp^*}$ and $\smash{\mathcal{G}_{\Imp^*}^{\mathrm{dbl}}}$ over $\Spin(d)$, respectively $\Spin(d) \times \Spin(d)$.
We now describe the data of the corresponding lifting 2-gerbe $\smash{\LiftTwoGrb_P^{\mathrm{dbl}}}$ explicitly.
Its cover diagram is given by \eqref{CoverDiagramCS}.
All data depend on $W^{[m, n]}$ only through the elements $q_{ij}^a$.

\begin{enumerate}
\item[{[2, 2]}]
The line bundle over $W^{[2, 2]}$ is the implementer bundle $\ImpLine$.
\item[{[3, 2]}]
It turns out that the isomorphism \eqref{IsoVdash} intertwines the fusion product of $\Imp_L^*$ with composition of operators.
The bundle gerbe product over $W^{[3, 2]}$ is therefore given by
\begin{equation*}
 \tilde{\lambda} : \ImpLine(q^\prime, q^{\prime\prime}) \otimes \ImpLine(q, q^\prime) \longrightarrow \ImpLine(q, q^{\prime\prime}), \qquad \Phi^\prime \otimes \Phi \longmapsto \Phi^\prime\Phi.
\end{equation*}
\end{enumerate}
The bundle isomorphism of the refinement along \eqref{RefinementForMmultGerbe}, given by group multiplication in $\Imp^*$, translates, via the isomorphism \eqref{IsoVdash}, to
\begin{equation}
\label{TransformedMultiplication}
\begin{aligned}
 \ImpLine(q_{23}, q_{23}^\prime) \otimes \ImpLine(q_{12}, q_{12}^\prime) &\longrightarrow \ImpLine(q_{23}q_{12}, q_{23}^\prime q_{12}^\prime), \\
  \Phi_{23} \otimes \Phi_{12} &\longmapsto L^2(\Cl_{q_{23}^\prime}) \Phi_{12} L^2(\Cl_{q_{23}^\prime})^* \Phi_{23}.
\end{aligned}
\end{equation}
%
%
To describe the data in the third column of the cover diagram \eqref{CoverDiagramCS}, one has to calculate the image $\mathfrak{m}$ of this refinement under the functor \eqref{InclusionRefinements} explicitly.
The result is the following.
\begin{enumerate}
\item[{[1, 3]}]
The line bundle over $W^{[1, 3]}$ is the \emph{multiplication line} $\Mul$, defined as the pullback of $\ImpLine$ along the map $W^{[1, 3]} \to W^{[1, 2]}$ given by
\begin{equation*}
  \left(
\begin{pmatrix} 
p_1 & p_2 & p_3 \\ 
\tilde{p}_1 & \tilde{p_2} & \tilde{p}_3
\end{pmatrix}; 
q_{23}, q_{12}, q_{13}
\right)
\longmapsto
  \left(
\begin{pmatrix} 
p_1  & p_3 \\ 
\tilde{p}_1 & \tilde{p}_3
\end{pmatrix};
 q_{23}q_{12}, q_{13}
\right).
\end{equation*}
Explicitly, the fibers of $\Mul$ are
\begin{equation*}
  \Mul(q_{23}, q_{12}, q_{13}) \defeq \ImpLine(q_{23}q_{12}, q_{13}).
\end{equation*}
\item[{[2, 3]}]
Over $W^{[2, 3]}$, we have an isomorphism
\begin{equation*}
 \tilde{\mu} :  \Mul^2 \otimes \ImpLine_{23} \otimes \ImpLine_{12} \longrightarrow \ImpLine_{13} \otimes \Mul^1,
\end{equation*}
fiberwise given as the cospan
\begin{equation*}
\begin{tikzcd}[row sep=0.1cm]
\Mul(q_{23}^\prime, q_{12}^\prime, q_{13}^\prime) \otimes \ImpLine(q_{23}, q_{23}^\prime) \otimes \ImpLine(q_{12}, q_{12}^\prime)
\ar[dr, bend left=10]
\ar[dd, dashed, "\tilde{\mu}"]
\\
&
\Mul(q_{23}, q_{12}, q_{13}^\prime)
\\
 \ar[ur, bend right=10, "\text{composition}"']
  \ImpLine(q_{13}, q_{13}^\prime) \otimes \Mul(q_{23}, q_{12}, q_{13}),
\end{tikzcd}
\end{equation*}
where the bottom right map is composition and the top right map, using formula \eqref{TransformedMultiplication}, is
\begin{equation*}
  \Phi_{123} \otimes \Phi_{23} \otimes \Phi_{12}  \longmapsto \Phi_{123} L^2(\Cl_{q_{23}^\prime}) \Phi_{12} L^2(\Cl_{q_{23}^\prime})^* \Phi_{23}.
\end{equation*}
\item[{[1, 4]}]
Over $W^{[1, 4]}$, the associator 2-morphism 
\begin{equation*}
  \tilde{\alpha} : \Mul_{124} \otimes \Mul_{234}
  \longrightarrow 
  \Mul_{134} \otimes \Mul_{123}
\end{equation*}
is canonically provided by the functor \eqref{InclusionRefinements}.
Explicitly, it is fiberwise given by the cospan
\begin{equation*}
\begin{tikzcd}[row sep=0.1cm, column sep=-1cm]
 \Mul(q_{24}, q_{12}, q_{14}) \otimes \Mul(q_{34}, q_{23}, q_{24})
\ar[dr, bend left=10, "\text{composition}"]
\ar[dd, dashed, "\tilde{\alpha}"]
\\
&
\Mul^{(3)}(q_{34}, q_{23}, q_{12}, q_{14}) \defeq \ImpLine(q_{34}q_{23}q_{12}, q_{14})
\\
 \ar[ur, bend right=10]
 \Mul(q_{34}, q_{13}, q_{14}) \otimes \Mul(q_{23}, q_{12}, q_{13}),
\end{tikzcd}
\end{equation*}
where the top right map is composition of operators and the bottom right map is given by
\begin{equation*}
\begin{aligned}
  \Phi_{134} \otimes \Phi_{123} &\longmapsto \Phi_{134} L^2(\Cl_{q_{34}}) \Phi_{123}L^2(\Cl_{q_{34}})^*.
\end{aligned}
\end{equation*}
%
%
\end{enumerate}

The data described above can be arranged with respect to the cover diagram \eqref{CoverDiagramCS} as follows.
\begin{equation} 
\label{DataCSImp}
%
\LiftTwoGrb_{P}^{\mathrm{dbl}} =  \left[
\begin{tikzcd}[column sep=0.7cm]
 & 
 	\tilde{\lambda}
	\ar[d,  dotted, -]
			\\
 & 
 	\ImpLine
 	\ar[d,  dotted, -]
	& 
		\tilde{\mu}
			\ar[l,  dotted, -]
 		\ar[d,  dotted, -]
			\\
&
	{\color{white}.}
	\ar[d,  dotted, -]
	& 
		\Mul
			\ar[l,  dotted, -]
		\ar[d,  dotted, -]
		&
			\tilde{\alpha}
			\ar[d,  dotted, -]
			\ar[l,  dotted, -]
			\\
(P \times P) \ar[d]
&
	(P \times P)^{[2]}
		\ar[l, shift left=1mm] 
		\ar[l, shift right=1mm]
	&
		(P \times P)^{[3]}
			\ar[l]
 			\ar[l, shift left=2mm]
 			\ar[l, shift right=2mm]
		&
			(P \times P)^{[4]}
			\ar[l, shift left=1mm] 
		\ar[l, shift right=1mm]
			\ar[l, shift left=3mm] 
		\ar[l, shift right=3mm]
			\\
			X \times X
\end{tikzcd}
\right]
\end{equation}

\begin{proof}[of Thm.~\ref{ThmIsoCSLag}]
The isomorphism $\mathcal{h}$ will be given in terms of the common cover
\begin{equation*}
\begin{tikzcd}[column sep=2cm]
 \Spin(X) \times \Spin(X)&  \ar[l, "\substack{\text{endpoint} \\ \text{evaluation}}"'] \ar[r, "\substack{\text{footpoint} \\ \text{projection}}"] P\Spin(X) & P X.
\end{tikzcd}
\end{equation*}
Its bundle gerbe over $P\Spin(X)$ is trivial, hence the first row of its cover diagram is trivial. The rest of its cover diagram is
\begin{equation} 
\label{CoverDiagramIsoCsLag}
\begin{tikzcd}
		& 
 	V^{[3, 2]}  
	\ar[d]
 	\ar[d, shift left=2mm]
 	\ar[d, shift right=2mm]
	& & 
		& 
			\\
	& 
 	V^{[2, 2]}  
 	\ar[d, shift left=1mm]
 	\ar[d, shift right=1mm]
	& 
		V^{[2, 3]}
			\ar[l]
 			\ar[l, shift left=2mm]
 			\ar[l, shift right=2mm]
 		\ar[d, shift left=1mm]
 		\ar[d, shift right=1mm]
		& & 
			\\
	&
	V^{[1, 2]} 
	\ar[d]
	& 
	   V^{[1, 3]}
			\ar[l]
 			\ar[l, shift left=2mm]
 			\ar[l, shift right=2mm]
		\ar[d]
		&
			V^{[1, 4]}
			\ar[d]
			\ar[l, shift left=1mm] 
		\ar[l, shift right=1mm]
			\ar[l, shift left=3mm] 
		\ar[l, shift right=3mm]
		& 
			\\
P\Spin(X) \ar[d]
&
	P\Spin(X)^{[2]}
		\ar[l, shift left=1mm] 
		\ar[l, shift right=1mm]
	&
		P\Spin(X)^{[3]}
			\ar[l]
 			\ar[l, shift left=2mm]
 			\ar[l, shift right=2mm]
		&
			P\Spin(X)^{[4]}
			\ar[l, shift left=1mm] 
		\ar[l, shift right=1mm]
			\ar[l, shift left=3mm] 
		\ar[l, shift right=3mm]
			\\
			X \times X
\end{tikzcd}
\end{equation}
with
\begin{equation*}
V^{[m, n]} = 
\left\{ 
\left.
\left(
 g_1, \dots, g_n;
\begin{pmatrix} q_{ij}^1 \\ \vdots \\ q_{ij}^m \end{pmatrix}_{i < j};
\begin{pmatrix} L_{ij}^1 \\ \vdots \\ L_{ij}^m \end{pmatrix}_{i < j}
\right)
~
\right| 
~
\begin{matrix}
q_{ij}^a(\pi) = g_j(\pi)^{-1} g_i(\pi) \\
  L_{ij}^a \in \Lag_{\gamma_i \fuse \gamma_j}
\end{matrix}
 \right\}.
\end{equation*}
Here $g_i \in P\Spin(X)$ lifts the path $\gamma_i$ and $q_{ij}^a \in P\Spin(d)$, $L_{ij}^a \in \Lag$.
The diagram \eqref{CoverDiagramIsoCsLag} maps to the cover diagrams \eqref{CoverDiagramCS} and \eqref{CoverDiagramLag} in the obvious way.

The data of $\mathcal{h}$ can be arranged with respect to the cover \eqref{CoverDiagramIsoCsLag} as follows.
\begin{equation} 
\label{CoverDiagramIsoCsLag2}
\begin{tikzcd}
		& 
 	\substack{\text{coherence} \\ \text{condition} \\ \text{for } \nu}  
	\ar[d, dotted, -]
	& & 
		& 
			\\
	& 
 	\nu 
 	\ar[d,  dotted, -]
	& 
		\substack{\text{coherence} \\ \text{condition} \\ \text{for } \beta} 
			\ar[l, dotted, -]
 		\ar[d,  dotted, -]
		& & 
			\\
	&
	\mathfrak{N} 
	\ar[d, dotted, -]
	& 
		\beta
			\ar[l,  dotted, -]
		\ar[d, dotted, -]
		&
			\substack{\text{coherence} \\ \text{condition} \\ \text{for } b} 
			\ar[d, dotted, -]
			\ar[l, dotted, -]
		& 
			\\
P\Spin(X) \ar[d]
&
	P\Spin(X)^{[2]}
		\ar[l, shift left=1mm] 
		\ar[l, shift right=1mm]
	&
		P\Spin(X)^{[3]}
			\ar[l]
 			\ar[l, shift left=2mm]
 			\ar[l, shift right=2mm]
		&
			P\Spin(X)^{[4]}
			\ar[l, shift left=1mm] 
		\ar[l, shift right=1mm]
			\ar[l, shift left=3mm] 
		\ar[l, shift right=3mm]
			\\
			X \times X
\end{tikzcd}
\end{equation}
Explicitly, these data are the following.
\begin{enumerate}
\item[{[1, 2]}]
The line bundle $\mathfrak{N}$ over $V^{[1, 2]}$ has fibers
\begin{equation*}
   \mathfrak{N}(g_1, g_2, q, L) = \underline{\Hom}_{g_2, g_1}(L^2(\A_0)_{q}, \F_L),
\end{equation*}
where the right hand side denotes the space of unitary intertwiners $L^2(\A_0)_{q} \to \F_L$ along $\Cl_{g_2}$ and $\Cl_{g_1}$ (recall that $L^2(\A_0)_q$ denotes the $\A_0$-$\A_0$-bimodule obtained from $L^2(\A_0)$ by modifying the right action via the Boguliubov automorphism $\Cl_q$).
%

\item[{[2, 2]}]
Over $V^{[2, 2]}$, we have a canonical isomorphism
\begin{equation*}
\nu : \mathfrak{N}^2 \otimes \ImpLine \longrightarrow \Pf \otimes \mathfrak{N}^1, 
\end{equation*}
fiberwise given by the cospan
\begin{equation*}
\begin{tikzcd}[row sep=0cm]
\overbrace{\underline{\Hom}_{g_2, g_1}(L^2(\A_0)_{q^\prime}, \F_{L^\prime})}^{\mathfrak{N}(g_1, g_2, q^\prime, L^\prime)} \otimes \overbrace{\underline{\Hom}(L^2(\A_0)_q, L^2(\A_0)_{q^\prime})}^{\ImpLine(q, q^\prime)} 
\ar[dd, dashed, "\nu"] 
\ar[dr, "\text{composition}", bend left=10]
 \\
 & \underline{\Hom}_{g_2, g_1}(L^2(\A_0)_q, \F_L).
 \\
 \ar[ur, "\text{composition}"', bend right=10]
 \underbrace{\underline{\Hom}(\F_L, \F_{L^\prime})}_{\Pf(L, L^\prime)} \otimes \underbrace{\underline{\Hom}_{g_2, g_1}(L^2(\A_0)_{q}, \F_{L})}_{\mathfrak{N}(g_1, g_2, q, L)}
\end{tikzcd}
\end{equation*}

\item[{[1, 3]}]
The isomorphism 
\begin{equation*}
  \beta : \mathfrak{N}_{13} \otimes \Mul \longrightarrow \Fus \otimes \mathfrak{N}_{23} \otimes \mathfrak{N}_{12},
\end{equation*}
over $V^{[1, 3]}$ is fiberwise given by the cospan
\begin{equation*}
\begin{tikzcd}[column sep=-5.5cm, row sep=0.8cm]
\overbrace{\underline{\Hom}_{g_3, g_1}(L^2(\A_0)_{q_{13}}, \F_{L_{13}})}^{\mathfrak{N}(g_1, g_3, q_{13}, L_{13})}
	\otimes 
\overbrace{\underline{\Hom}\bigl(L^2(\A_0)_{q_{23}q_{12}}, L^2(\A_0)_{q_{13}} \bigr)}^{\Mul(q_{23}, q_{12}, q_{13})}
\ar[dr, bend left=11, "\text{composition}", near end]
\ar[dd, dashed, "\beta"]
\\
& \underline{\Hom}_{g_3, g_1}(L^2(\A_0)_{q_{23}q_{12}}, \F_{L_{13}}) 
\\
\ar[ur, bend right=8]
\underbrace{\underline{\Hom}(\F_{L_{23}} \boxtimes_{\A_{\gamma_2}} \F_{L_{12}}, \F_{L_{13}})}_{\Fus(L_{23}, L_{12}, L_{13})} \otimes \underbrace{\underline{\Hom}_{g_3, g_2}(L^2(\A_0)_{q_{23}}, \F_{L_{23}})}_{\mathfrak{N}(g_2, g_3, q_{23}, L_{23})} \otimes \underbrace{\underline{\Hom}_{g_2, g_1}(L^2(\A_0)_{q_{12}}, \F_{L_{12}})}_{\mathfrak{N}(g_1, g_2, q_{12}, L_{12})}
\end{tikzcd}
\end{equation*}
where the top right arrow is composition, while the bottom right arrow is the isomorphism
\begin{equation*}
  \Upsilon \otimes \Psi_{23} \otimes \Psi_{12} \longmapsto \Upsilon (\Psi_{23} \boxtimes \Psi_{12}) \chi_{q_{23}, q_{12}}^*.
\end{equation*}
Here $\chi_{q_{23}, q_{12}} : L^2(\A_0)_{q_{23}} \boxtimes_{\A_0} L^2(\A_0)_{q_{12}} \to L^2(\A_0)_{q_{23}q_{12}}$ is the unitary isomorphism from Example~\ref{ExampleModifiedActionConnesFusion}.
\end{enumerate}

We have to check three coherences.
As the bundle gerbe $\mathcal{H}$ of $\mathcal{h}$ is trivial, the coherence diagram \eqref{CoherenceEta} over $V^{[3, 2]}$ collapses in our case to
\begin{equation*}
\begin{tikzcd}[column sep=2cm]
  \mathfrak{N}^3 \otimes \ImpLine^{23} \otimes \ImpLine^{12} 
 \ar[r, "1 \otimes \tilde{\lambda}"]
 \ar[d, "\nu^{23}  \otimes 1 \otimes  1"']
 &
    \mathfrak{N}^3 \otimes \ImpLine^{13}
     \ar[dd, "\nu^{13}"]
 \\
\Pf^{23} \otimes   \mathfrak{N}^2 \otimes  \ImpLine^{12}
	 \ar[d, " 1 \otimes  \nu^{12}"'] 
	  \\
\Pf^{23} \otimes \Pf^{12} \otimes \mathfrak{N}^1
	\ar[r, "\lambda \otimes 1"] 
	&
	\Pf^{13}  \otimes \mathfrak{N}^1,
\end{tikzcd}
\end{equation*}
which is clearly commutative.

The coherence diagram \eqref{CoherenceBetaTilde} for over $V^{[2, 3]}$ is
\begin{equation*}
\begin{tikzcd}[column sep=2cm]
  \mathfrak{N}_{13}^2 \otimes \Mul^2 \otimes \ImpLine_{23} \otimes \ImpLine_{12}
  \ar[r, "\beta^2\otimes 1 \otimes 1"]
  \ar[dd, "1 \otimes \tilde{\mu}"']
  	&
  	\Fus^2 \otimes \mathfrak{N}_{23}^2 \otimes \mathfrak{N}_{12}^2 \otimes \ImpLine_{23} \otimes \ImpLine_{12}
		\ar[d, equal]
   \\
   &
   \Fus^2 \otimes \mathfrak{N}_{23}^2  \otimes \ImpLine_{23} \otimes \mathfrak{N}_{12}^2 \otimes \ImpLine_{12}
		\ar[d, "1\otimes \nu_{23} \otimes 1\otimes 1"]
	\\
  \mathfrak{N}_{13}^2 \otimes \ImpLine_{13} \otimes \Mul^{1}
  	\ar[ddd, "\nu_{13} \otimes 1"']
	&
	\Fus^2 \otimes \Pf_{23}  \otimes  \mathfrak{N}_{23}^1\otimes \mathfrak{N}_{12}^2 \otimes \ImpLine_{12}
	\ar[d, "1 \otimes 1 \otimes 1\otimes \nu_{12} "]
	\\
	&
	\Fus^2 \otimes \Pf_{23}  \otimes  \mathfrak{N}_{23}^1\otimes \Pf_{12} \otimes \mathfrak{N}_{12}^1
	\ar[d, equal]	\\
	&
	\Fus^2\otimes \Pf_{23}  \otimes \Pf_{12}  \otimes  \mathfrak{N}_{23}^1 \otimes \mathfrak{N}_{12}^1
	\ar[d, "\mu \otimes 1\otimes  1"]
	 \\
\Pf_{13} \otimes \mathfrak{N}_{13}^1  \otimes \Mul^{1}
	\ar[r, "1 \otimes 1 \otimes \beta^1"]
	 &
	 \Pf_{13}   \otimes  \Fus^1  \otimes \mathfrak{N}_{23}^1 \otimes \mathfrak{N}_{12}^1
\end{tikzcd}
\end{equation*}
We consider this diagram over a point
\begin{equation*}
\left(
 g_1, g_2, g_3;
\begin{pmatrix} 
q_{23} & q_{12} & q_{13} \\ 
q_{23}^\prime & q_{12}^\prime & q_{13}^\prime 
\end{pmatrix}
;
\begin{pmatrix} 
L_{23} & L_{12} & L_{13} \\ 
L_{23}^\prime & L_{12}^\prime & L_{13}^\prime 
\end{pmatrix}
\right) 
\in V^{[2, 3]}.
\end{equation*}
Here, the bottom left line is, via composition, isomorphic to
\begin{equation}
\label{BottomLeftLine}
 \mathfrak{N}(g_1, g_3, q_{23}q_{12}, L_{13}^\prime) = \underline{\Hom}_{g_3, g_1}(L^2(\A_0)_{q_{23}q_{12}}, \F_{L_{13}}).
\end{equation}
A diagram chase starting with an element $\Upsilon \otimes \Psi_{23} \otimes \Psi_{12} \otimes \Phi_{23} \otimes \Phi_{12}$ of the top right line reveals that proving commutativity of the diagram boils down to showing the equality
\begin{equation*}
\Upsilon  (\Psi_{23} \boxtimes \Psi_{12})  \chi_{q_{23}^\prime, q_{12}^\prime}^* L^2(\Cl_{q_{23}^\prime}) \Phi_{12}   L^2(\Cl_{q_{23}^\prime})^*\Phi_{23} = 
\Upsilon (\Psi_{23} \Phi_{23} \boxtimes \Psi_{12} \Phi_{12}) \chi_{q_{23}}^*
\end{equation*}
inside \eqref{BottomLeftLine};
 the left hand side corresponds to the north-west composition, while the right hand side corresponds to the east-south composition.
 But this identity follows from \eqref{FunctorialityConnesFusion} and \eqref{TwistFunctoriality}.

Finally, it is convenient to split the diagram \eqref{CocycleConditionb} into two halves, 
\begin{equation*}
\begin{tikzcd}[column sep=1.8cm]
\mathfrak{N}_{14} \otimes \Mul_{124} \otimes \Mul_{234} 
	\ar[r]
	\ar[d, "\beta_{124} \otimes 1"']
	&
	\mathfrak{N}_{14} \otimes \Mul^{(3)}
	\ar[d, "\text{composition}"]
	\\
	\Fus_{124} \otimes \mathfrak{N}_{24}  \otimes \mathfrak{N}_{12} \otimes \Mul_{234}
	\ar[d, equal]
	&
	\mathfrak{NMul}^{(3)}
	\ar[d]
	\\
\Fus_{124} \otimes \mathfrak{N}_{24}  \otimes \Mul_{234} \otimes \mathfrak{N}_{12}
	\ar[d, "1 \otimes \beta_{234} \otimes 1"]
	&
	\mathfrak{FusN}^{(3)}
	\\
\Fus_{124} \otimes \Fus_{234} \otimes \mathfrak{N}_{34} \otimes \mathfrak{N}_{23}  \otimes \mathfrak{N}_{12}
	\ar[r]
	&
	\Fus^{(3)}  \otimes \mathfrak{N}_{34} \otimes \mathfrak{N}_{23}  \otimes \mathfrak{N}_{12}
	\ar[u, "\text{composition}"']
\end{tikzcd}
\end{equation*}
and
\begin{equation*}
\begin{tikzcd}[column sep=1.8cm]
\mathfrak{N}_{14} \otimes \Mul^{(3)}
	\ar[d, "\text{composition}"']
	&
	\mathfrak{N}_{14} \otimes \Mul_{134} \otimes \Mul_{123}
	\ar[d, "\beta_{134} \otimes 1"]
	\ar[l]
	\\
	\mathfrak{NMul}^{(3)}
	\ar[d]
	&
	\Fus_{134} \otimes \mathfrak{N}_{34} \otimes \mathfrak{N}_{13} \otimes \Mul_{123}
	\ar[d, "1 \otimes 1 \otimes \beta_{123}"]
	\\
	\mathfrak{FusN}^{(3)}
	&
	\Fus_{134} \otimes \mathfrak{N}_{34} \otimes \Fus_{123} \otimes \mathfrak{N}_{23} \otimes \mathfrak{N}_{12}
	\ar[d, equal]
	\\
\Fus^{(3)}  \otimes \mathfrak{N}_{34} \otimes \mathfrak{N}_{23}  \otimes \mathfrak{N}_{12}
\ar[u, "\text{composition}"]
	&
	\Fus_{134} \otimes \Fus_{123}  \otimes \mathfrak{N}_{34} \otimes \mathfrak{N}_{23}  \otimes \mathfrak{N}_{12},
	\ar[l]
\end{tikzcd}
\end{equation*}
where the horizontal arrows are the ones from the defining cospan of $\alpha$, respectively $\tilde{\alpha}$.
Joining those two diagrams at the common side gives \eqref{CocycleConditionb}, hence commutativity of the latter follows from commutativity of those individual diagrams.
Here 
\begin{align*}
\mathfrak{NMul}^{(3)}(g_1, g_4, L_{14}, q_{34}, q_{23}, q_{12}) &= \underline{\Hom}_{g_4, g_1}\bigl(L^2(\A_0)_{q_{34}q_{23}q_{12}}, \F_{L_{14}}\bigr) \\
\mathfrak{FusN}^{(3)}(g_1, g_4, L_{14}, q_{34}, q_{23}, q_{12}) &= \underline{\Hom}_{g_4, g_1}\bigl(L^2(\A_0)_{q_{34}} \boxtimes_{\A_0} L^2(\A_0)_{q_{23}}\boxtimes_{\A_0} L^2(\A_0)_{q_{12}}, \F_{L_{14}}\bigr)
\end{align*}
and the map $\mathfrak{NMul}^{(3)} \to \mathfrak{FusN}^{(3)}$ is  fiberwise given by
\begin{equation*}
\Psi \longmapsto \Psi \chi_{q_{34}q_{23}, q_{12}} (\chi_{q_{34}, q_{23}} \boxtimes \id_{L^2(\A_0)}) = \Psi \chi_{q_{34}, q_{23}q_{12}} (\id_{L^2(\A_0)} \boxtimes \chi_{q_{23}, q_{12}}),
\end{equation*}
where the equality on the right hand side follows from \eqref{AssociativityChi}.
Commutativity of both diagrams is easily checked using \eqref{FunctorialityConnesFusion} and \eqref{TwistFunctoriality} again.
\end{proof}

\appendix

\section{Super von Neumann algebras and Connes fusion}

In this appendix, we give some background to super von Neumann algebras and the Connes fusion product of super bimodules between them.
Standard references for the Connes fusion product in the ungraded case are \cite{thom11, BrouwerBicategorical}, but here we present its adaptation to the ungraded case, which we were unable to find in the literature.

\medskip

Throughout, all Hilbert spaces are separable and all $C^*$-algebras are assumed to have a faithful representation on a separable Hilbert space.
Recall that a \emph{von Neumann algebra} is a unital $*$-algebra $A$ that is isomorphic to a subalgebra of $\BB(H)$, $H$ a Hilbert space, which is closed with respect to the ultraweak operator topology (or, equivalently, the weak or strong topology).
Equivalently (by the theorem of Sakai), a von Neumann algebra is a $C^*$-algebra $A$ that admits a \emph{predual}, i.e., a Banach space $A_*$ such that $A$, as a Banach space, is the dual of $A_*$ (the predual is unique up to unique isomorphism).
The \emph{ultraweak topology} on a von Neumann algebra $A$ is the weak-$*$-topology induced by its predual; it coincides with the ultraweak topology induced by any faithful $*$-representation.

For von Neumann algebras $A$, $B$, a \emph{normal} $*$-homomorphism $\varphi: A \to B$  is a $*$-homomorphism that is continuous with respect to the ultraweak topologies.
This is equivalent to $\varphi$ being the dual of a bounded linear operator $\varphi_* : B_* \to A_*$ between the preduals.
It is a fact that every $*$-\emph{iso}morphism between von Neumann operators is automatically normal.
%

\subsection{Super von Neumann algebras}
\label{SectionASuperVN}

A \emph{super Hilbert space} is a Hilbert space $H$ with a direct sum decomposition $H = H^0 \oplus H^1$, where the summands are called the ``even'' and ``odd'' parts of $H$.
The \emph{grading operator} of a super Hilbert space $H$ is the operator that acts as multiplication with $(-1)^i$ on $H^i$.

A \emph{super von Neumann algebra} is a von Neumann algebra together with an involutive $*$-automorphism $\tau$.
The eigendecomposition $A = A^0 \oplus A^1$ of $\tau$ gives a decomposition into ``even'' and ``odd'' parts of $A$, which satisfy $A^i \cdot A^j \subseteq A^{i+j}$.
Equivalently, $A$ is isomorphic to an ultraweakly closed, graded subalgebra of $\BB(H)$ for some super Hilbert space $H$, where we say that a subalgebra is \emph{graded} if it is preserved by conjugation with the grading operator of $H$; 
the grading operator of $A$ is then given by conjugation with that of $H$.
$*$-homomorphisms between super von Neumann algebras will always be required to preserve the grading.

Let $A$ be a super von Neumann algebra.
The \emph{(super) opposite} of $A$ is the super von Neumann algebra $A^{\op}$ consisting of elements $a^{\op}$ for $a \in A$.
It is equipped with the super vector space structure such that $a \mapsto a^{\op}$ is an isomorphism, but carries the opposite multiplication and $*$-operation,
\begin{equation}
\label{OppositeMultiplication}
  a^{\op} \cdot b^{\op} = (-1)^{|a||b|} (ba)^{\op}, \qquad (a^{\op})^* = (-1)^{|a|} (a^*)^{\op}.
\end{equation}
 A \emph{(super) anti-$*$-homomorphism}  $\varphi: A \to B$ is a linear map satisfying
\begin{equation*}
  \varphi(ab) = (-1)^{|a||b|}\varphi(b)\varphi(a).
\end{equation*}
An anti-$*$-homomorphism $\varphi : A \to B$ is the same thing as a $*$-homomorphism $A^{\op} \to B$ or $A \to B^{\op}$.
If $A$ and $B$ are super $*$-algebras, then the $*$-operation is an anti-$*$-homomorphism $\overline{A} \to A$, but \emph{not} a \emph{super} anti-$*$-homomorphism unless $A$ is purely even.
However, the map $\sharp$ given by 
 \begin{equation*}
  a^\sharp \defeq \begin{cases} a^* & a \text{ even} \\ i a^* & a \text{ odd}. \end{cases}
\end{equation*}
is a super anti-$*$-homomorphism $\overline{A} \to A$, for any super $*$-algebra $A$.
%

\subsection{Super bimodules}
\label{SectionSuperBimodules}

Let $A$ and $B$ be super von Neumann algebras. 
Recall that a super \emph{$A$-$B$-bimodule} is a super Hilbert space $M$ together with grading preserving unital normal $*$-representations $\pi_A$ of $A$ and $\pi_B$ of $B^{\op}$ on $M$, whose images super commute.
Here $B^{\op}$ denotes the super opposite algebra; see \eqref{OppositeMultiplication}.
In terms of these representations, the bimodule action is defined by
\begin{equation}
\label{BimoduleStructureRepresentations}
      a\lact  \xi \ract b \defeq (-1)^{|\xi||b|}\pi_A(a)\pi_B(b)\xi, \qquad \xi \in M, ~~a \in A, ~b \in B.
\end{equation}
We remark that it is crucial that $\pi_B$ is a representation of the \emph{super} opposite algebra, and that $\pi_A$ and $\pi_B$ \emph{super} commute in order for formula \eqref{BimoduleStructureRepresentations} to indeed define a bimodule structure (where, of course, the left and right action are required to \emph{commute}, not super commute).

\medskip

Every super von Neumann algebra $A$ has a canonically associated super $A$-$A$-bimodule $L^2(A)$, called the \emph{standard bimodule} of $A$ \cite{HaagerupStandardForm, Kosaki}, \cite[IX]{Takesaki2}.
It is characterized (unique up to isomorphism) by the existence of an even, anti-linear, isometric involution $J : L^2(A) \to L^2(A)$ and a self-dual cone $P \subset L^2(A)$ with the following properties:
\begin{enumerate}[$(1)$]
\item The right action is given by $\xi \ract a = J\pi_\ell(a^*)J\xi$, $a \in A$, $\xi \in L^2(A)$, where $\pi_\ell$ denotes the representation of $A$ on $L^2(A)$ that induces the left action.
Equivalently, the right representation is given by $\pi_r(a) = (-1)^{|a||\xi|} J\pi_\ell(a^*)J\xi$.
\item $J \pi_\ell(A) J = \pi_\ell(A)^\prime$, the (ungraded) commutant in $\BB(L^2(A))$.
\item $J \xi = \xi$ for all $\xi \in P$.
\item $P$ is \emph{self-dual}, meaning that $P = \{ \eta \in H \mid \langle \forall \xi \in P : \xi, \eta \rangle \geq 0 \}$.
\item $a \lact \xi \ract a^* \in P$ for all $a \in A$, $\xi \in P$.
\item $J\pi_\ell(c)J = \pi_\ell(c)^*$ for each $c$ in the (ungraded) center of $A$.
\end{enumerate}

\begin{remark}
\label{RemarkCyclicVector}
If $H$ is a super Hilbert space with a faithful $*$-representation $\pi$ of $A$ and $\Omega \in H$ is an even cyclic and separating vector for $\pi$, then there exists a unique grading preserving unitary isomorphism $L^2(A) \cong H$ that intertwines the conjugation $J$ of $L^2(A)$ with the modular conjugation $J_\Omega$ for $\Omega$, obtained via Tomita-Takesaki theory, and that sends the cone $P$ to the self-dual cone given by
\begin{equation}
\label{ConeFromCyclicSeparatingVector}
  P_\Omega = \overline{\{ \pi(a) J \pi(a) J \Omega \mid a \in A\}}.
\end{equation}
The point here is that if $\Omega \in H$ and $\Omega^\prime \in H^\prime$ are two cyclic and separating vectors for $*$-representations $\pi$ and $\pi^\prime$, we get a canonical unitary isomorphism $H \cong L^2(A) \cong H^\prime$; however, this  isomorphism does in general not send $\Omega$ to $\Omega^\prime$, but only $P_\Omega$ to $P_{\Omega^\prime}$.
\end{remark}

\begin{remark}
\label{RemarkL2phi}
By uniqueness of the standard bimodule, an isomorphism $\phi : A \to B$ of super von Neumann algebras induces a grading preserving unitary isomorphism $L^2(\phi) : L^2(A) \to L^2(B)$.
In particular, if $B = A$ so that $\phi$ is an automorphism, then $L^2(A)$ is a unitary automorphism of $L^2(A)$ that intertwines both module actions along $\phi$, i.e., 
\begin{equation*}
  L^2(\phi) (a \lact \xi \ract b) = \phi(a) \lact L^2(\phi) \xi \ract \phi(b), \qquad a, b \in A, ~~ \xi \in L^2(A).
\end{equation*}
This is \emph{Haagerup's canonical implementation}.
The map $\phi \mapsto L^2(\phi)$ is continuous for the u-topology on $\Aut(A)$ (see \S\ref{SectionVNBundlesBimodules} below) and the strong topology on $\U(L^2(A))$ \cite[Prop.~3.5]{HaagerupStandardForm}, \cite[IX.1.15\&1.16]{Takesaki2}.
\end{remark}

We will use the following generalized notion of bimodule homomorphisms.
Let $M$ be a super $A$-$B$-bimodule, let $M^\prime$ be a super $A^\prime$-$B^\prime$-bimodule and let $\phi_A: A \to A^\prime$ and $\phi_B : B \to B^\prime$ be (grading preserving) $*$-isomorphisms. 
Then a unitary isomorphism $U : M \to M^\prime$, which we assume to be either grading preserving or grading reversing, is an \emph{intertwiner along} $\phi_A$ and $\phi_B$ if
\begin{equation}
\label{IntertwiningRelation}
  U(a \lact \xi \ract b) = (-1)^{|U||a|} \phi_A(a) \lact U(\xi) \ract \phi_B(b), \qquad a \in A, ~b\in B, ~\xi \in M.
\end{equation}
Intertwiners $U$ along $\phi_A$ and $\phi_B$ can be depicted as 2-cells,
\begin{equation*}
\begin{tikzcd}[column sep=1.8cm, row sep=1.4cm]
  A  \ar[d, "\theta_A"'] & B \ar[d, "\theta_B"] \ar[l, "M"']\ar[d, Rightarrow, shift right=1.3cm, shorten=0.2cm, "U"]\\
  A^\prime & B^\prime  \ar[l, "M^\prime"].
\end{tikzcd}
\end{equation*}
\begin{remark}
The additional sign in \eqref{IntertwiningRelation} involving the parity of $U$ is according to the Koszul sign rule.
It is necessary for the following to be true:
If $M$ is a super left $A$-module, then $M$ is also a super right $A^{\op}$-module, where
\begin{equation*}
  \xi \ract a^{\op} \defeq (-1)^{|\xi||a|} a \lact \xi, \qquad \xi \in M, ~a \in A.
\end{equation*}
Now, an even or odd unitary $U : M \to M^\prime$ intertwines the left $A$-actions if and only if it intertwines the corresponding right $A^{\op}$-actions.
%
%
\end{remark}

In the main body of the paper, we will often deal with bimodule isomorphisms that are, more generally, a scalar multiple of a unitary; this generalization causes no further complications and will be ignored in this appendix.

\subsection{Connes Fusion of super bimodules}
\label{SectionConnesFusion}

Let $B$ be a super von Neumann algebra.
If $M$ is a right super $B$-module and $N$ is a left super $B$-module, there is a notion of a relative tensor product over $B$, the \emph{Connes fusion product} \cite[Chapter~V]{Wassermann}; for details, see, e.g, \cite{thom11}.
To review its definition, fix a standard bimodule $L^2(B)$ for $B$ and denote by $\mathscr{D}(M) = \BB_B(L^2(B), M)$ the space of bounded linear maps that intertwine the right $B$-actions.
$\mathscr{D}(M)$ is graded by parity and has a graded right $B$-action, given by $(x \ract b)(\xi) = x(b \lact \xi)$.
%
%
The Connes fusion product $M \boxtimes_B N$ of $M$ with $N$ is then defined as the completion
\begin{equation*}
        M \boxtimes_B N = \overline{\mathscr{D}(M) \otimes N}
\end{equation*} 
of the algebraic tensor product with respect to the (degenerate) sesquilinear form
\begin{equation}\label{ConnesFusionSesquiForm}
        \langle x \otimes \xi, y \otimes \eta\rangle_{M \boxtimes N} \defeq \langle \xi, (x^* y) \lact \eta \rangle_N,
\end{equation}
where we use that $x^* y \in \BB_B(L^2(B), L^2(B)) = B$ (identifying $b \in B$ with its left multiplication operator on $L^2(B)$).
Observe that since $N$ is $\Z_2$-graded and the left $B$-action is compatible with this grading, elements of different grading in the algebraic tensor product $\mathscr{D}(M) \otimes N$ are orthogonal with respect to the inner product \eqref{ConnesFusionSesquiForm}, so $M \boxtimes_B N$ is a super Hilbert space.
Abusing notation, we use the same notation for elements of $M \boxtimes_B N$ as for their preimages in $\mathscr{D}(M) \otimes N$.

\begin{remark}
Elements of the form $(x \ract b) \otimes \xi - x \otimes (b \lact \xi)$ are in the kernel of the inner product \eqref{ConnesFusionSesquiForm}, hence vanish in the completion. 
Therefore, $M \boxtimes_B N$ can indeed be viewed as a tensor product \emph{over $B$}.
\end{remark}

\begin{example}
\label{ExampleTensorProduct}
Let $M$ be a super right $B$-module and let $N$ be a super left $B$-module, for a super von Neumann algebra $B$.
Let $V$ and $W$ be finite-dimensional super Hilbert spaces.
Then $M \otimes V$ and $N \otimes W$ are again $B$-modules and we have a canonical isomorphism
\begin{equation*}
  (M \boxtimes_B N) \otimes (V \otimes W) \longrightarrow  (M \otimes V) \boxtimes_B (N \otimes E),
\end{equation*}
given by
\begin{equation*}
  (x \otimes \eta) \otimes (v \otimes w) \longmapsto (-1)^{|v||\eta|}(x \otimes v) \otimes (\eta \otimes w),
\end{equation*} 
where for $x \in \mathscr{D}(M)$ and $v \in V$, $x \otimes v$ denotes the element of $\mathscr{D}(M \otimes V)$ given by $(x \otimes v)\xi = (-1)^{|\xi||v|}x\xi \otimes v$.
\end{example}

If $M$ is a super $A$-$B$-bimodule and $N$ is a super $B$-$C$-bimodule, then $M \boxtimes_B N$ naturally has the structure of a super $A$-$C$-bimodule, with the actions given by 
\begin{equation*}
a \lact (x \otimes \xi) \ract c = (a \lact x) \otimes (\xi \ract c), \qquad a \in A, ~c \in C, ~x \in \mathscr{D}(M),~ \xi \in N.
\end{equation*}
Here the left $A$-action on $\mathscr{D}(M)$ is defined pointwise, $(a \lact x)(\xi) = a \lact x(\xi)$.

Connes fusion is associative, in the sense that there are isomorphisms
\begin{equation}
\label{AssociatorsConnesFusion}
  M \boxtimes_B (N \boxtimes_C P) \cong (M \boxtimes_B N) \boxtimes_C P,
\end{equation}
called \emph{associators}, which are functorial in $M$, $N$ and $P$ and satisfy the pentagon axiom for any quadruple of bicategories \cite[Prop.~3.5.3]{BrouwerBicategorical}. 
We will therefore drop the brackets and just write $M \boxtimes_B N \boxtimes_C P$ for the tensor product of three super bimodules.

\medskip

We now discuss Connes fusion of intertwiners and its functoriality.
It will be important to define Connes fusion not only for ordinary unitary intertwiners but, more generally, for intertwiners along $*$-isomorphisms.
Explicitly, let $U : M \to M^\prime$ be a unitary intertwiner along $\phi_A$ and $\phi_B$, and let $V : N \to N^\prime$ be a unitary intertwiner along $\phi_B$ and $\phi_C$ (here it is necessary that $U$ and $V$ intertwine along the \emph{same} middle isomorphisms).
Connes fusion then provides a composition 2-cell
\begin{equation*}
\begin{tikzcd}[column sep=1.8cm, row sep=1.4cm]
 A \ar[d, "\varphi_A"'] & B  \ar[d, "\theta_B"'] \ar[l, "M"'] \ar[d, Rightarrow, shift right=1.3cm, shorten=0.2cm, "U"] & C \ar[d, "\theta_C"] \ar[l, "N"']\ar[d, Rightarrow, shift right=1.3cm, shorten=0.2cm, "V"]\\
 A^\prime & B^\prime \ar[l, "M^\prime"] & C^\prime  \ar[l, "N^\prime"],
\end{tikzcd}
\qquad = \qquad
\begin{tikzcd}[column sep=1.8cm, row sep=1.4cm]
  A \ar[d, "\theta_C"'] & C \ar[d, "\theta_A"] \ar[l, "M \boxtimes_{B} N"']\ar[d, Rightarrow, shift right=1.3cm, shorten=0.2cm, "U \boxtimes V"]\\
  A^\prime & C^\prime  \ar[l, "M^\prime \boxtimes_{B} N^\prime"],
\end{tikzcd}
\end{equation*}
in other words, a unitary intertwiner $U \boxtimes V : M \boxtimes_{B} N \to M^\prime \boxtimes_{B^\prime} N^\prime$, along $\theta_A$ and $\theta_C$, compare \cite[Prop.~A.2.3]{KristelWaldorf3}, \cite[Thm.~6.23]{DualizabilityBartelsDouglasHenriques}.
This intertwiner is defined on $\mathscr{D}(M) \otimes N$ by 
\begin{equation} \label{DefinitionFusionIntertwiners}
        (U \boxtimes V)(x \otimes \xi) = (-1)^{|V||x|} U x L^2({\phi}_B)^* \otimes V\xi, \qquad x \in \mathscr{D}(M),~ \xi \in N,
\end{equation}
and extends by continuity.
Here $L^2({\phi}_B) : L^2(B) \to L^2(B^\prime)$ is the unitary isomorphism of standard bimodules induced by the isomorphism $\phi_B : B \to B^\prime$ (see Remark~\ref{RemarkL2phi}).
The parity of $U \boxtimes V$ is the sum of the parities of $U$ and $V$.
%
%

\begin{example}
\label{ExampleModifiedActionConnesFusion}
For an automorphism $\theta \in \Aut(A)$, denote by $L^2(A)_\theta$ the $A$-$A$-bimodule where the right $A$-action is twisted by $\theta$.
Then given $\theta_1, \theta_2 \in \Aut(A)$, there are canonical even unitary isomorphisms
\begin{equation*}
 \chi_{\theta_1,\theta_2} : L^2(A)_{\theta_1} \boxtimes_A L^2(A)_{\theta_2} \longrightarrow L^2(A)_{\theta_1\theta_2}
\end{equation*}
given on elements of $\mathscr{D}(L^2(A)_{\theta_1}) \otimes L^2(A)_{\theta_2}$ by the formula
\begin{equation*}
 x \otimes \xi \longmapsto x \xi.
\end{equation*}
It is straightforward to check that these isomorphisms enjoy the following relation with respect to intertwiners: 
If $U_i : L^2(A)_{\theta_i} \to L^2(A)_{\theta_i^\prime}$, $i=1, 2$, are such that $U_1$ is right intertwining along $\varphi$ and $U_2$ is left intertwining along $\varphi$, then
\begin{equation*}
\chi_{\theta_1^\prime, \theta_2^\prime} \circ (U_1 \boxtimes U_2) = U_1 L^2(\theta_1)L^2(\varphi)^*U_2 L^2(\theta_1)^*.
\end{equation*}
If $U_1$ and $U_2$ are intertwiners along the identity on both sides (in particular, $\varphi = \id$), then this can be rewritten as
\begin{equation}
\label{TwistFunctoriality}
  \chi_{\theta_1^\prime,\theta_2^\prime}(U_1 \boxtimes U_2) = L^2(\theta_1^\prime) U_2 L^2(\theta_1^\prime)^* U_1 \chi_{\theta_1,\theta_2}.
\end{equation}
%
%
Moreover, these isomorphisms are associative, in the sense that we have the identity
\begin{equation}
\label{AssociativityChi}
  \chi_{\theta_1\theta_2,\theta_3} (\chi_{\theta_1,\theta_2} \boxtimes \id_{L^2(A)}) = \chi_{\theta_1,\theta_2\theta_3} (\id_{L^2(A)} \boxtimes \chi_{\theta_2,\theta_3}), 
\end{equation}
of isomorphisms $L^2(A)_{\theta_1} \boxtimes_A L^2(A)_{\theta_2} \boxtimes_A L^2(A)_{\theta_3} \to L^2(A)_{\theta_1\theta_2\theta_3}$, for any triple $\theta_1, \theta_2, \theta_3 \in \Aut(A)$.
Here the associator \eqref{AssociatorsConnesFusion} is implicit.
%
\end{example}

Connes fusion is functorial in the sense that if $U : M \to M^\prime$, $V: N \to N^\prime$ intertwine along $\phi_B$ and $U^\prime : M^\prime \to M^{\prime\prime}$, $V^\prime : N^\prime \to N^{\prime\prime}$ intertwine along $\phi_B^\prime$, then 
\begin{equation} 
\label{FunctorialityConnesFusion}
        (U^\prime \boxtimes V^\prime)(U \boxtimes V) = (-1)^{|V^\prime||U|}U^\prime U \boxtimes V^\prime V.
\end{equation}

\section{Bundles of von Neumann algebras and bimodules}
\label{SectionVNBundlesBimodules}

In this section, we explain the notion of (continuous, locally trivial) bundles of super von Neumann algebras and bimodules between them.

\subsection{Bundles of super von Neumann algebras}
\label{SectionVNBundles}

The \emph{automorphism group} of a super von Neumann algebra $A$ is the group $\Aut(A)$ of grading preserving, (normal) $*$-automorphisms of $A$.
The \emph{u-topology} on the automorphism group is the topology on $\Aut(A) \subset \BB(A)$ generated by the seminorms
\begin{equation}
\label{Definitionutopology}
    \theta \mapsto \sup_{\substack{a \in A \\ \|a\| = 1}} |\rho(\theta(a))| = \|\rho \circ \theta\|_{A_*},
\end{equation}
where $\rho$ ranges over all elements of the predual of $A$ (viewed as a linear functional on $A$).
The u-topology turns $\Aut(A)$ into a completely metrizable topological group.

\begin{lemma} \label{LemmaCharacterizationUTopology}
        Let $X$ be a topological space, let $A$ be a super von Neumann algebra and let $\theta : X \to \Aut(A)$ be a map.
        Then $\theta$ is continuous if and only if the corresponding map $\hat{\theta} : X \times A \to A$ is continuous on bounded sets, when $A$ carries its ultraweak topology and $\Aut(A)$ carries the u-topology.
\end{lemma}

Here we say that a map $X \times A \to A$ is continuous on bounded if for one (hence for all) $r>0$, its restriction $X \times A_r \to A_r$ are continuous, where $A_r = \{a \in A \mid \|a\| \leq r\}$ denotes the $r$-ball of $A$. 
Recall here that by the uniform boundedness principle, the norm bounded sets of $A$ coincide with the ultraweakly bounded ones.

\begin{remark}
Similarly to Lemma~\ref{LemmaCharacterizationUTopology}, it was shown in \cite[Lemma~4.3]{PennigWStar} that if $A$ is a factor of type II$_1$, then a map $\theta : X \to \Aut(A)$ is continuous if and only if the corresponding map $\hat{\theta} : X \times A \to A$ is continuous for the pointwise 2-topology \cite[Lemma~4.3]{PennigWStar}.
\end{remark}

\begin{proof}[of Lemma~\ref{LemmaCharacterizationUTopology}]
Suppose that $\theta$ is continuous and let $(x_i, a_i)_{i \in I}$ be a net in $X \times A$ converging to $(x, a)$.
We have to show that for each $\rho \in A_*$, we have $\rho(\hat{\theta}(x_i, a_i)) \rightarrow \rho(\hat{\theta}(x, a))$.
To this end, we calculate
\begin{equation*} 
\begin{aligned}
|\rho(\hat{\theta}(x_i, a_i)) - \rho(\hat{\theta}(x, a))| 
&= |\rho(\theta_{x_i}(a_i)) - \rho(\theta_x(a))| \\
&\leq |\rho((\theta_{x_i} - \theta_x)(a_i-a))| + |\rho((\theta_{x_i}-\theta_x)(a))| + |\rho(\theta_x(a_i-a))|\\
&\leq (\|a_i - a\| + \|a\|) \|\rho \circ (\theta_{x_i} - \theta_x)\|_{A_*} + |\rho(\theta_x(a_i-a))|.
\end{aligned}
\end{equation*}
The second term converges to zero by weak convergence of the net $(a_i)_{i \in I}$ to $a$.
If the net $(a_i)_{i \in I}$ is norm-bouned, the first term converges to zero by continuity of the $\theta$ (see \eqref{Definitionutopology}).
The second term converges to zero by weak convergence of the net $(a_i)_{i \in I}$ to $a$.
\\
Conversely, suppose that $\hat{\theta}$ is continuous and let $(x_i)_{i \in I}$ be a net in $X$, convergent to $x$.
We have to show that $\theta_{x_i} \to \theta_x$ in $\Aut(A)$, meaning that 
\begin{equation*}
\|\rho \circ (\theta_{x_i} - \theta_x)\|_{A_*} = \sup_{a \in A_1} |\rho((\theta_{x_i}-\theta_x)(a))| \longrightarrow 0
\end{equation*}
for each $\rho \in A_*$.
Since the unit ball $A_1$ of $A$ is weakly compact and the functional $\rho \circ (\theta_{x_i} - \theta_x)$ is weakly continuous, the supremum is attained.
Hence for each $i \in I$, there exists $a_i \in A_1$ such that $\|\rho \circ (\theta_{x_i} - \theta_x)\|_{A_*} = |\rho((\theta_{x_i}-\theta_x)(a_i))|$.
This gives a net $(a_i)_{i \in I}$.
Let $(a_i)_{i \in I^\prime}$, $I^\prime \subset I$ be any subnet.
Since $A_1$ is ultraweakly compact, there exists a further subnet $(a_i)_{i \in I^{\prime\prime}}$, $I^{\prime\prime} \subset I^{\prime}$ that converges ultraweakly to some element $a \in A_1$.
Then
\begin{align*}
  \|\rho \circ (\theta_{x_i} - \theta_x)\|_{A_*} &= |\rho( (\theta_{x_i} - \theta_x)(a_i)| \\
  &= |\rho(\hat{\theta}(x_i, a_i) - \hat{\theta}(x, a_i))| \\
  &= |\rho(\hat{\theta}(x_i, a_i) - \hat{\theta}(x, a))| + |\rho(\hat{\theta}(x, a) - \hat{\theta}(x, a_i))|
\end{align*}
As net indexed by $I^{\prime\prime}$, both terms converge to zero, by continuity of $\hat{\theta}$ on $X \times A_1$.
We showed that any subnet of the net $(\|\rho \circ (\theta_{x_i} - \theta_x)\|_{A_*})_{i \in I}$ in $\R$ has a further subnet converging to zero. 
But this implies that the net itself converges to zero.
\end{proof}

We now explain the notion of von Neumann algebra bundle we use in this paper.
Let $X$ be a topological space together with a collection $\A=(\A_x)_{x \in X}$ of super von Neumann algebras.
We will also write $\A$ for the disjoint union
\begin{equation*}
  \A = \coprod_{x \in X} \A_x
\end{equation*}
and call this the \emph{total space} of $\A$.
It comes with a canonical footpoint projection $\pi : \A \to X$ and, as usual, for $O \subset X$, we write $\A|_O = \pi^{-1}(O)$.
By a  \emph{trivialization} of $\A$ over an open set $O \subset X$, we mean a map
\begin{equation*}
  \phi : \A|_O \longrightarrow O \times A,
\end{equation*} 
where $A$ is some super von Neumann algebra, such that the fibers $\phi_x : \A_x \to A$ are grading preserving $*$-isomorphisms.
A bundle structure on $\A$ consists of a suitable collection of \emph{local trivializations} $(\phi_{i})_{i\in I}$ whose transition functions
\begin{equation*}
  \phi_j \circ \phi_i^{-1} : (O_i \cap O_j) \times A \longrightarrow (O_i \cap O_j) \times A
\end{equation*}
are \emph{continous on bounded sets} when $A$ carries its ultraweak topology, in the sense that they are continuous when restricted to the sets $(O_i \cap O_j) \times A_r$, where $A_r$, $r>0$, denotes the $r$-ball of $A$.
By Lemma~\ref{LemmaCharacterizationUTopology}, this is equivalent to requiring that the corresponding maps $O_i \cap O_j \to \Aut(A)$ are continuous with respect to the u-topology.
Compatible collections of local trivializations are ordered by inclusion.

\begin{definition}[Von Neumann algebra bundle]
\label{DefinitionvNBundle}
A super von Neumann algebra bundle over a space $X$ is a collection $=(\A_x)_{x \in X}$ of super von Neumann algebras, together with a maximal compatible collection of local trivializations.
\end{definition}

Note that the total space $\A$ does \emph{not} canonically possess the structure of a topological space as the transition functions need not be continuous.
However, a compatible collection of local trivializations on $\A$ induces, for each fixed $r>0$, a continuous fiber bundle structure on the set 
\begin{equation*}
\A_r = \coprod_{x \in X} (\A_x)_r,
\end{equation*}
of fiberwise $r$-balls, each fiber carrying the ultraweak topology.
Note that since each $\A_x$, as a super von Neumann algebra, admits a predual, the Banach-Alaoglu theorem implies that the fibres of $\A_r$ are compact.

\begin{definition}[Bundle homomorphism]
\label{def:vonNeumannbundlehomomorphism}
A \emph{homomorphism} between super von Neumann algebra bundles $\A$ and $\B$ is a fibre-preserving map $\rho: \A \to \B$ whose fibers $\rho_x : \A_x \to \B_x$ are grading preserving normal $*$-homomorphisms and which is continuous on bounded sets with respect to the ultraweak topologies.
\end{definition}

Explicitly, the continuity requirement on $\rho$ means that for some (hence for all) $r>0$, the map $\rho_r : \A_r \to \B_r$, given over $x \in X$ by the restriction of $\rho_x$ to the $r$-ball in $\A_x$, is a continuous when $\A_r$ and $\B_r$ carry the topology explained above.

\begin{example}[Trivial bundles]
\label{ExampleTrivialvonNeumannAlgebraBundle}
If $A$ is a super von Neumann algebra, the trivial super von  Neumann algebra bundle over $X$ with typical fibre $A$ is $\underline{A} = X \times A$, whose collection of local trivializations is given by restrictions of the identity map  $\underline{A} \to X \times A$ to open subsets of $X$.  
\end{example}

\begin{example}[Associated bundles]
 \label{ExampleAssociatedBundlesAutA}
        Let $A$ be a super von Neumann algebra and let $G$ be a topological group with a continuous group homomorphism $\rho : G \to \Aut(A)$ (with respect to the u-topology).
        Then any principal $G$-bundle $P$ over $X$ gives rise to a von Neumann algebra bundle over $X$ with typical fiber $A$ by forming the associated bundle 
        \begin{equation*}
        \A \defeq P \times_{G} A.
        \end{equation*}
        Since $G$ acts via grading preserving $*$-automorphisms, the fibers $\A_x = P_x \times_G A$ carry induced structures of super von Neumann algebras.
        To obtain local trivializations, let $\{O_i\}_{i \in I}$ be an open cover of $X$ such that $P$ admits local sections $p_i: O_i \to P$, then defining 
        \begin{equation*}
        \phi_i: \A|_{O_i} \to O_i \times A, \qquad [p_i(x), a] \to (x, a)
        \end{equation*}
        defines a set of local trivializations.  
        Given $i, j \in I$, there are continuous functions $g_{ij} : O_i \cap O_j \to \Aut(A)$ such that 
        \begin{equation*}
        p_i(x) \cdot g_{ij}(x) = p_j(x), \qquad x \in O_i \cap O_j.
        \end{equation*}
        The transition functions for the collection of local trivializations are then
\begin{equation*}
(\phi_i \circ \phi_j^{-1})(x,a) = \rho(g_{ij}(x))(a)\text{,}
\end{equation*}
As the maps $g_{ij}$ and $\rho$ are continuous, it follows from Lemma~\ref{LemmaCharacterizationUTopology} that these transition functions are continuous on bounded sets.
\end{example}

\subsection{Super bimodule bundles}
\label{SectionvNbimodbundles}

(Super) Hilbert space bundles over a topological space  $X$ are unproblematic to define: they are fiber bundles with typical fiber a super Hilbert space $H$ (with its norm topology), equipped with fibrewise Hilbert space structures, such that the local trivializations are fiberwise grading preserving linear isometries.
It is a standard fact that the continuity of a transition function $O \cap O^\prime \times H \to H$ is equivalent to the continuity of the corresponding map $O \cap O^\prime \to \U(H)$.
%
%
In the ungraded case, since $\U(H)$ is contractible (provided that $H$ is infinite-dimensional), all such bundles are trivializable.
In the graded case, if we allow both even and odd unitary transformation as transition functions, the unitary group $\U(H)$ has two connected components, each of which is contractible (assuming that both the even and odd part of $H$ are infinite-dimensional).
Hence such bundles $\mathfrak{H}$ of super Hilbert spaces over a space $X$ are classified by a single characteristic class $\orclass(\mathfrak{H}) \in H^1(X, \Z_2)$, for which the terminology \emph{orientation class} seems sensible.

\begin{definition}[Bimodule bundle] 
\label{DefinitionBimoduleBundle}
        Let $\A$, $\B$ be super von Neumann algebra bundles over a space $X$.
        Let $\mathfrak{M}$ be a super Hilbert space bundle over $X$ such that each fiber $\mathfrak{M}_x$ carries the structure of a super $\A_x$-$\B_x$-bimodule.
        A \emph{local bimodule trivialization} of $\mathfrak{M}$ over an open set $O\subseteq X$ consists of local trivializations $\phi : \A|_O \to O \times A$, $\psi : \B|_O \to O \times B$ and $U: \mathfrak{M}|_O \to O \times M$ (the latter as a super Hilbert space bundle) such that for each $x \in O$, $U_x$ is an intertwiner along $\phi_x$ and $\psi_x$.
        $\mathfrak{M}$ is a \emph{super $\A$-$\B$-bimodule bundle} if there exists a local bimodule trivialization in a neighborhood of each point of $X$.
\end{definition}

\begin{definition}[Bimodule bundle isomorphism]
\label{DefinitionBimoduleBundleIsomorphism}
        An \emph{isomorphism} between super $\A$-$\B$-bimodule bundles $\mathfrak{M}$ and $\mathfrak{N}$ is a unitary isomorphism $u: \mathfrak{M} \to \mathfrak{N}$ of super Hilbert space bundles that fiberwise intertwines the bimodule actions.
        \end{definition}

\begin{remark}
The \emph{typical fiber} for super $\A$-$\B$-bimodule bundles $\mathfrak{M}$ is a somewhat tricky notion.
Naively, one might think that if if $(\phi, U, \psi)$ is a triple of local trivializations as in Definition~\ref{DefinitionBimoduleBundle}, then the typical fiber of $\mathfrak{M}$ would be the $A$-$B$-bimodule bundle $M$.
However, in this situation, we can form a new trivialization $\psi^\prime = \theta^{-1} \circ \phi$ of $\B$  for any choice of automorphism $\theta \in \Aut(B)$; then $U$ is intertwining when the target is taken to be the $A$-$B$-bimodule $M_\theta$, obtained from $M$ by modifying the $B$-action by $\theta$.
Similarly, we can modify $\phi$ by an automorphism of $A$ to modify the left action of $M$.
We conclude that the typical fiber $M$, as an $A$-$B$-bimodule, is only well defined up to replacing it with a module of the form ${}_\theta M_{\theta^\prime}$, where the left and right action are modified by automorphisms.
\end{remark}

\subsection{Connes Fusion of bimodule bundles}
\label{SectionConnesFusionBundle}

Let $\A$, $\B$, $\CC$ be super von Neumann algebra bundles over a topological space $X$, let $\mathfrak{M}$ be a super $\A$-$\B$-bimodule bundle and let $\mathfrak{N}$ be a super $\B$-$\CC$-bimodule bundle. 
In this situation, we can form the fiberwise Connes fusion products $\mathfrak{M}_x \boxtimes_{\B_x} \mathfrak{N}_x$.
The challenge is to endow the ``total space''
\begin{equation*}
\mathfrak{M} \boxtimes_{\B} \mathfrak{N} = \coprod_{x \in X} \mathfrak{M}_x \boxtimes_{\B_x} \mathfrak{N}_x
\end{equation*}
 with the structure of a continuous Hilbert space bundle. 
 We observe:
 \begin{enumerate}
 \item[$(\star)$]
If $(\phi, U, \psi)$ and $(\psi, V, \kappa)$ are local bimodule trivializations of $\mathfrak{M}$, respectively $\mathfrak{N}$, both defined over some open set $O \subseteq X$, then $(\phi, U \boxtimes V, \kappa)$ is a candidate for a local bimodule trivialization of $\mathfrak{M} \boxtimes_{\B} \mathfrak{N}$ over $O$, where 
\begin{equation*}
U \boxtimes V : \mathfrak{M} \boxtimes_{\B} \mathfrak{N}|_O \longrightarrow O \times (M \boxtimes_B N)
\end{equation*} 
denotes the pointwise fusion product along $\psi$.
 \end{enumerate}
However, for this  fusion product to exist, it is crucial that $U$ and $V$ intertwine along the \emph{same} trivialization $\psi$ of $\B$, and in general, $\mathfrak{M}$ and $\mathfrak{N}$ might not admit local bimodule trivializations with this property.

For finite-dimensional algebras and bimodules, the conditions for such a pair of trivializations to exist near every point of $X$ were investigated in \cite{InsideousBicategory}; to formulate this condition in the von Neumann algebra setting, we define, for a super $A$-$B$-bimodule $M$, the \emph{group of implementers on} $M$ as the subgroup
 \begin{equation*}
   I(M) \subseteq \Aut(A) \times \U(M) \times \Aut(B)
\end{equation*}
consisting of those triples $(\varphi, U, \psi)$ such that $U$ is an even intertwiner along $\varphi$ and $\psi$.
Endowed with the subspace topology, $I(M)$ is a topological group.
Projection onto the first, respectively third factor yields continuous group homomorphisms 
\begin{align*}
 p_\ell : I(M) &\longrightarrow \Aut(A), \\
 p_r : I(M) &\longrightarrow \Aut(B).
\end{align*}

\begin{remark}
\label{RemarkIMsubgroup}
If the actions of $A$ and $B$ on $M$ are faithful, then the projection map 
\begin{equation*}
I(M) \longrightarrow \U(M)
\end{equation*}
 is injective.
In fact, it is a homeomorphism onto its image, so that $I(M)$ can be identified with a subgroup of $\U(M)$.

To see latter claim, we have to show that whenever $(\varphi_n, U_n, \psi_n)_{n \in \N}$ is a sequence in $I(M)$ and $(\varphi, U, \psi) \in I(M)$ is such that $U_n$ converges against $U$ in the topology of $\U(M)$, then $\varphi_n \to \varphi$ and $\psi_n \to \psi$ with respect to the u-topology.

To this end, let $\varrho$ be an element of the predual of $A$.
It suffices to consider the case that $\varrho$ is positive. 
As $A$ acts faithfully, such an element can be written as 
\begin{equation*}
\varrho(a) = \sum_{k=1}^\infty \langle \xi_k, a \lact \xi_k \rangle, 
\end{equation*}
with $(\xi_k)_{k \in \N}$ a square-summable sequence in $M$.
%
%
Then given $a \in A$ with $\|a\| \leq 1$, we have
\begin{equation*}
\begin{aligned}
\bigl|\varrho (\varphi_n(a)) - \varrho (\varphi(a))\bigr|
&= \sum_{k=1}^\infty\bigl| \langle \xi_k, U_n (a \lact U_n^*\xi_k)\rangle - \langle \xi_k, U (a \lact U^*\xi_k)\rangle\bigr| \\
&= \sum_{k=1}^\infty\bigl| \langle a^*\lact U_n^* \xi_k, (U_n - U)^*\xi_k\rangle - \langle (U_n - U)\xi_k, a \lact U^*\xi_k)\rangle\bigr| \\
&= 2 \sum_{k=1}^\infty\|\xi_k\|\|(U_n - U)^*\xi_k\| .
\end{aligned}
\end{equation*}
Since $U_n \to U$ in the topology of $\U(H)$, the differences $U_n - U$ and $(U_n- U)^*$ converge strongly to zero (see \cite{EspinozaUribe}). Therefore the right hand side above converges to zero by the dominated convergence theorem.
As it is independent from $a$, this convergence is uniform in $a$, and we see that $\|\varrho \circ \varphi_n - \varrho \circ \varphi\|_{A_*}$.
Varying $\varrho$, we get that $\varphi_n \to \varphi$ in the u-topology of $\Aut(A)$.
Similarly, one shows that $\psi_n \to \psi$ in the u-topology of $\Aut(B)$.
\end{remark}

We first check that the trivializations of the form $(\star)$, if they exist, are indeed compatible.
We need the following lemma.

\begin{lemma} 
\label{LemmaContinuityConnesFusion}
The fusion map
\begin{equation*}
  I(M) \times_{\Aut(B)} I(N) \longrightarrow I(M \boxtimes_B N), \qquad (U, V) \longmapsto U\boxtimes V
\end{equation*}
is continuous.
\end{lemma}

\begin{proof}
$I(M)$ and $I(N)$ are closed subgroups of polish groups, hence polish; in particular, we just need to check sequential continuity.
Let $(\phi_n, U_n, \psi_n)_{n \in \N}$, $(\psi_n, V_n, \kappa_n)_{n \in \N}$, be sequences converging against $(\phi, U, \psi)$, respectively $(\psi, V, \kappa)$.
Then for any $x \boxtimes \xi \in M \otimes_B N$, we have
\begin{align*}
  &\bigl\|\bigl((U_n \boxtimes V_n) - (U \boxtimes V)\bigr) x \boxtimes \xi\bigr\| = \| U_n x \boxtimes V_n \xi - Ux \boxtimes V \xi\|\\
  &~~~~\leq \|(U_n-U)x \boxtimes V\xi \| + \|U_n x \boxtimes (V_n - V)\xi\|\\
  &~~~~= \langle V \xi, (x^*(U_n - U)^* (U_n - U) x) \lact V\xi\rangle^{1/2} 
   + \langle (V_n - V)\xi, ((U_nx)^*U_nx) \lact (V_n - V)\xi \rangle^{1/2}\\
  &~~~~\leq \|\xi\|^{1/2} \|(x^*(U_n - U)^* (U_n - U) x) \lact V \xi\|^{1/2} 
   + \|(x^*x)\|^{1/2}\|(V_n - V)\xi\|
\end{align*}
The second term converges to zero as $V_n \to V$  strongly.
To get a clean argument that the first term converges to zero as well, let $\varphi: B \to \mathrm{B}(N)$ be the normal unitary $*$-homomorphism inducing the action.
By the classification of normal unital $*$-homomorphisms (see e.g., \cite[\S IV, Thm.~5.5]{Takesaki1}), there exists a Hilbert space $H$ with $B \subset \mathrm{B}(H)$ and a (necessarily surjective) partial isometry $w: H \to N$ with $w^*w \in B^\prime$ such that $\varphi(b) = wbw^*$.
Then
\begin{equation*}
\begin{aligned}
(x^*(U_n - U)^* (U_n - U) x) \lact V \xi &= \varphi\bigl(x^*(U_n - U)^* (U_n - U) x\bigr)V \xi \\
&= wx^*(U_n - U)^* (U_n - U) xw^*V \xi.
\end{aligned}
\end{equation*}
Since $U_n \to U$ strongly, this converges to zero.
By linearity, this shows that
\begin{equation} \label{StrongConvergenceSubset}
(U_n \boxtimes V_n) X \to (U \boxtimes V) X.
\end{equation}
whenever $X$ is in the image of $\mathscr{D}(M) \otimes N \to M \boxtimes_B N$. 
This set is by definition dense and the sequence $(U_n \boxtimes V_n)_{n \in \N}$ is uniformly bounded in the operator norm, hence an elementary argument shows that \eqref{StrongConvergenceSubset} holds for all $X \in M \boxtimes_B N$.
%
%
\end{proof}

\begin{corollary}
\label{CorollaryCompatibilityOfTransitionFunctions}
The local trivializations of the form $(\star)$ are all compatible.
\end{corollary}

\begin{proof}
If $(\phi_i, U_i,  \psi_i)$, $(\psi_i, V_i, \kappa_i)$, $i=1, 2$, are two pairs of trivializations as in $(\star)$, defined over an open set $O$, we can form the transition functions
\begin{equation*}
  (\phi, U, \psi) : O \to I(M), \qquad (\psi, V, \kappa) : O \to I(N), 
\end{equation*}
where $\phi = \phi_1 \circ \phi_2^{-1}$, $U = U_1 U_2^*$, etc.
By functoriality of Connes fusion, the transition function between the ``fused'' trivializations $(\phi_i, U_i \boxtimes V_i, \kappa_i)$, $i=1, 2$, is then the pointwise fusion product $(\phi, U \boxtimes V, \kappa)$.
Therefore, continuity of the transition functions follows from Lemma~\ref{LemmaContinuityConnesFusion}.
\end{proof}

We now turn to the existence question of the trivializations of the form $(\star)$. 
To this end, we introduce the following notion.

\begin{definition}[Implementing bimodule]
\label{DefinitionImplementingBimodule}
We say that a super $A$-$B$-bimodule $M$ is \emph{right implementing} (\emph{left implementing}) if the projection map $p_r : I(M) \to \Aut(B)$ (respectively $p_\ell: I(M) \to \Aut(A)$) admits a unit-preserving, continuous section in a neighborhood of the identity.
We say that $M$ is \emph{implementing} if it is both left and right implementing.
\end{definition}

\begin{remark}
\label{RemarkLifting}
That $M$ is right implementing means, equivalently, that whenever $X$ is a topological space and $\varphi : X \to \Aut(B)$ is a continuous map such that $\varphi_x = \id_B$, there exists an open neighborhood $O \subset X$ of $x$ and a continuous map $(\psi, U, \varphi) : O \to I(M)$  such that $\psi_x = \id_A$ and $U_x = \id_M$.
This is the case, in particular, if $p_r : I(M) \to \Aut(B)$ is a fiber bundle.
\end{remark}

We now give several examples of implementing and non-implementing bimodules.
In the context of finite-dimensional algebras, examples of non-implementing bimodules can be found in \cite[Example~3.1.5]{InsideousBicategory}).

\begin{example}
Let $A = \BB(H)$, $H$ a super Hilbert space.
Then $H$ is implementing as a super $A$-$\C$-bimodule.
That $H$ is right implementing is obvious, as $\Aut(\C)$ is trivial and $p_r : I(H) \to \Aut(\C)$ is the trivial map.
To see that $H$ is left implementing, we note that $I(H) \cong \U^+(H)$ (in view of Remark~\ref{RemarkIMsubgroup}), which is a principal $\U(1)$-bundle over $\Aut(A) = P\U(H)_0$, the identity component of $P\U(H)$, and $p_{\ell}$ is the bundle projection.
\end{example}
 
 \begin{example}
 The standard bimodule $L^2(A)$ of a super von Neumann algebra $A$ is always implementing.
 In fact, both maps $p_\ell$ and and $p_r$ admit global sections by a group homomorphism, namely the canonical implementation, respectively the canonical implementation followed by conjugation by $J$; see Remark~\ref{RemarkL2phi}.
 \end{example}
 
 \begin{example}
\label{RemarkRightImplementingAgain}
If both $M$ and $N$ are right (left) implementing, then $M \boxtimes_{B} N$ is again right (left) implementing.
Indeed, if $N$ is right implementing, then for any space $X$ and any continuous map $\kappa : X \to \Aut(C)$ with $\kappa_x = \id_C$, there exists a continuous map $(\psi, V, \kappa) : O \to I(N)$ on some open neighborhood $O \subseteq X$ of $x$ such that $\psi_x = \id_B$; see Remark~\ref{RemarkLifting}.
Then if additionally $M$ is right implementing, after possibly passing to a smaller open neighborhood $\tilde{O} \subseteq O$ of $x$, the map $\psi : O \to \Aut(B)$ can be lifted to a map $(\phi, U, \psi) : \tilde{O} \to I(M)$.
Now, by Lemma~\ref{LemmaContinuityConnesFusion}, the fiberwise fusion product $U \boxtimes V : \tilde{O} \to \U(M)$ along $\psi$ is continuous, hence $(\phi, U \boxtimes V, \kappa) : \tilde{O} \to I(M \boxtimes_B N)$ is a continuous lift of $\kappa$.
Since $\kappa$ was arbitrary, this implies that $M \boxtimes_B N$ is right implementing.
The case that $M$ and $N$ are left implementing is similar.
\end{example}
 
 \begin{example}
 \label{ExampleMeasureSpace}
 We give an example of a bimodule that is neither right nor left implementing.
 Let $A = L^\infty([0, 1])$ and $I \subsetneq X$ be a proper subinterval. 
 Then the Hilbert space $M = L^2(I)$ is an $A$-$A$-bimodule with the bimodule structure given by restriction,
 \begin{equation*}
 a \lact \xi \ract b = a|_I \cdot \xi \cdot b|_I, \qquad a, b \in A, ~\xi \in M.
 \end{equation*}
 Any measure space automorphism $f$ of $[0, 1]$ induces an automorphism $\varphi_f$ of $A$ by pullback.
 If both $f$ and $f^{-1}$ preserve $I$ (up to possibly a set of measure zero), then the operator $U_f : L^2(I) \to L^2(I)$ given by
  \begin{equation}
 \label{FormulaUf}
   (U_f \xi)(t) = \xi(f(t)) \cdot \sqrt{f^\prime(t)}, 
 \end{equation}
 is well-defined, unitary and implements $\varphi_f$.
(Here $f^\prime$ denotes the Radon-Nikodym derivative of $f^*dt$ with respect to the Lebesgue measure $dt$ on $I$.)
%
%
 Otherwise, it is easy to check that $\varphi_f$ is not implemented by any unitary.
%
%
 Since automorphisms $f$ that do not essentially preserve $I$ can be found in any neighborhood of the identity in $\Aut(A)$, this shows that $M$ is not right (nor left) implementing.
 %
 \end{example}

 We have the following result (compare Prop.~4.2.3 of \cite{InsideousBicategory}).

\begin{lemma} \label{LemmaConnesFusionBundle}
Let $\mathfrak{M}$ be a super $\A$-$\B$-bimodule bundle and let $\mathfrak{N}$ be a super $\B$-$\CC$-bimodule bundle. 
If $\mathfrak{M}$ is right implementing or $\mathfrak{N}$ is left implementing (in the sense that each fiber has this property), then a pair of local bimodule trivializations as in $(\star)$ exists near each point of $X$.
\end{lemma}

\begin{proof}
We consider the case that $\mathfrak{M}$ is right implementing, the other case is similar.
Let $x \in X$ and let $(\phi^\prime, U^\prime, \psi^\prime)$ and $(\psi, V, \kappa)$ be local bimodule trivializations  of $\mathfrak{M}$, respectively $\mathfrak{N}$, both defined over some open neighborhood $O \subset X$ of $x$.
Explicitly, we have local trivializations
\begin{equation*}
\begin{aligned}
\phi^\prime : \A|_O &\longrightarrow O \times A,  \qquad
\psi, \psi^\prime : \B|_O &\longrightarrow O \times B, \qquad
 \kappa : \CC|_O &\longrightarrow O \times C
\end{aligned}
\end{equation*}
of von Neumann algebra bundles, as well as trivializations 
\begin{equation*}
U^\prime : \mathfrak{M}|_O \to O \times M^\prime, \qquad V : \mathfrak{N}|_O \to O \times N,
\end{equation*}
of Hilbert space bundles, which are intertwining along $\phi^\prime$ and $\psi^\prime$, respectively $\psi$ and $\kappa$.
Set $\psi_0 = \psi_x \circ (\psi_x^\prime)^{-1} \in \Aut(B)$, which we identify with the corresponding constant $\Aut(B)$-valued function on $O$.

That $\mathfrak{M}$ is right implementing implies that its typical fiber $M^\prime$ is right implementing.
Hence, after possibly passing to a small open neighborhood $\tilde{O} \subseteq O$ of $x$, the map $\psi_0 \circ \psi \circ (\psi^\prime)^{-1} : \tilde{O} \to \Aut(B)$ can be lifted to a map $(\alpha, W, \psi_0 \circ \psi \circ (\psi^\prime)^{-1}) : \tilde{O} \to I(M^\prime)$; see Remark~\ref{RemarkLifting}.
For $x \in \tilde{O}$, we view $W_x$ as an intertwiner $M^\prime \to M = M_{\psi_0}^\prime$ along $\alpha$ and $\psi \circ (\psi^\prime)^{-1}$, so that
\begin{equation*}
  U \defeq WU^\prime : \mathfrak{M}|_{\tilde{O}} \longrightarrow \tilde{O} \times M
\end{equation*}
is intertwining along $\phi \defeq \alpha \circ \phi^\prime$.
Then $(\phi, U, \psi)$ is a new local bimodule trivialization of $\mathfrak{M}$, on a neighborhood of $x$, which together with $(\psi, V, \kappa)$ forms a pair of local bimodule trivializations as in $(\star)$.
\end{proof}

Combining Lemma~\ref{LemmaConnesFusionBundle} and Corollary~\ref{CorollaryCompatibilityOfTransitionFunctions}, we obtain the following result.

\begin{proposition}
\label{PropositionExistenceOfFusionBundle}
Let $\mathfrak{M}$ be a super $\A$-$\B$-bimodule bundle and let $\mathfrak{N}$ be a super $\B$-$\CC$-bimodule bundle. 
If $\mathfrak{M}$ is right implementing or $\mathfrak{N}$ is left implementing, then there exists a unique structure of a super $\A$-$\CC$-bimodule on the fiberwise fusion product $\mathfrak{M} \boxtimes_{\B}\mathfrak{N}$ with local trivializations given by $(\star)$.
\end{proposition}

We finish this section with an example where the fiberwise Connes fusion product does \emph{not} admit the structure of a Hilbert space bundle (where, of course, the typical fibers do not satisfy the condition of Prop.~\ref{PropositionExistenceOfFusionBundle}).

\begin{example}
Consider the algebra $A = L^\infty([0, 1])$ and let $\mathfrak{M}$ and $\mathfrak{N}$ be the following bimodule bundles over $X = [0, \frac{1}{2}]$ (for the trivial von Neumann algebra bundle $\underline{A}$):
$\mathfrak{M}$ is the trivial bimodule bundle with $\mathfrak{M}_x = L^2([0, \frac{1}{2}])$ for every $x \in X$, while $\mathfrak{N}$ is the bimodule bundle with fibers $\mathfrak{N}_x = L^2([x, x+\frac{1}{2}])$.
The fibers carry the bimodule structure by restriction, discussed in Example~\ref{ExampleMeasureSpace}.
The bundle structure on $\mathfrak{N}$ is the following: Let $(f_x)_{x \in X}$ be given by
\begin{equation*}
  f_x(t) = \begin{cases} t-x & \text{if}~~x-t \geq 0 \\ t-x+1 & \text{if}~~t-x < 0. \end{cases}, \qquad t \in [0, 1], ~ x \in X.
\end{equation*}
Then $f_x$ sends $[x, x+\frac{1}{2}]$ to $[0, \frac{1}{2}]$ and $(\varphi_{f_x})_{x \in X}$ is a continuous family of automorphisms of $A$, which we view as a local trivialization of the trivial bundle $\underline{A}$.
%
%
The corresponding unitary $U_{f_x}$ given by \eqref{FormulaUf} is both left and right intertwining along $\varphi_{f_x}$ and sends $\mathfrak{N}_x$ to $L^2([0, \frac{1}{2}])$.
Hence the corresponding family $(U_{f_x})_{x \in X}$ is a global bimodule trivialization of $\mathfrak{N}$ and hence provides the structure of an $\underline{A}$-$\underline{A}$-bimodule bundle.
Now, it is not hard to see that we have canonical isomorphisms
\begin{equation*}
\mathfrak{M}_x \boxtimes_A \mathfrak{N}_x \cong L^2([x, \tfrac{1}{2}]), \qquad x \in X,
\end{equation*}
 of $\underline{A}$-$\underline{A}$-bimodules.
 The right hand side is infinite-dimensional if $x \in [0, \frac{1}{2})$ and zero if $x=\frac{1}{2}$, so the fiberwise Connes fusion products cannot assemble to a continuous bundle of Hilbert spaces.
\end{example}

\section{Super bundle gerbes and super bundle 2-gerbes}
\label{SectionPreliminariesSuperBundleGerbes}

In this section, we present some basic definitions and classification results regarding super bundle gerbes and 2-gerbes.
The material in this section is certainly well known in the ungraded case, but seems not to exist in the literature in quite the form needed.

Throughout, it is important that we work with \emph{super} line bundles instead of ungraded line bundles, by which we mean a line bundle together with a $\Z_2$-grading; explicitly, such a grading just determines, for each connected component of $Y$, whether $\L$ is even or odd on this component.

\subsection{Super bundle gerbes}
\label{SectionBundleGerbes}

Let $X$ be a (possibly infinite-dimensional) manifold.
Traditionally, a \emph{(super) bundle gerbe} is presented in terms of an open cover $\mathcal{U} = \{U_i\}_{i \in I}$ and consists of (super) line bundles $\mathfrak{L}_{ij}$ over the two-fold intersections $U_i \cap U_j$, together with (grading preserving) isomorphisms $\mathfrak{L}_{jk} \otimes \mathfrak{L}_{ij} \to \mathfrak{L}_{ik}$ over $U_i \cap U_j \cap U_k$, required to satisfy a cocycle condition over four-fold intersections.
This categorifies the notion of a (super) line bundle, which can be given by $\C^\times$-valued functions $g_{ij}$ over $U_i \cap U_j$ satisfying the obvious cocycle condition over triple intersections.
As a default in this paper, we take ``bundle gerbe'' to mean ``\emph{line} bundle gerbe'', which we take as a default in this paper; 
for an $A$-bundle gerbe (where $A$ is an abelian Lie group), one replaces the line bundles $\mathfrak{L}_{ij}$ by principal $A$-bundles.

While any bundle gerbe can be represented in this fashion, such a representation is often not canonical, but depends on the (usually rather unnatural) choice of the open cover $\mathcal{U}$.
It is therefore useful to replace the open cover $\mathcal{U}$ be a general surjective submersion $Y$.
Observe that if $Y = \coprod_{i \in I} U_i$ for a countable open cover $\mathcal{U} = \{U_i\}_{i \in I}$, then the $k$-fold fiber product $Y^{[k]}$ can be identified with the disjoint union of all $k$-fold intersections of elements of $\mathcal{U}$.
Throughout the paper, we will use the following notation.

\begin{notation}
\label{NotationSubmersion}
Let $Y \to X$ be a surjective submersion and let
\begin{equation*}
 Y^{[k]} \defeq \underbrace{Y \times_X \cdots \times_X Y}_{k~\text{times}}
\end{equation*}
be its $k$-fold iterated fiber product.
Given some geometric object $\mathfrak{G}$ over the $k$-fold fiber product $Y^{[k]}$ which can be pulled back along smooth maps (for example, a vector bundle, a bundle gerbe, or a map), we denote by
\begin{equation*}
  \mathfrak{G}_{i_1 \dots i_k} \defeq \pi_{i_1 \dots i_k}^* \mathfrak{G}, \qquad 1 \leq i_1, \dots, i_k \leq n
\end{equation*}
the object on $Y^{[n]}$ obtained by pullback along the projection map 
\begin{equation*}
\pi_{i_1\dots i_k} : Y^{[n]} \longrightarrow Y^{[k]}
\end{equation*}
given by projection on the indicated factors.
This is a slight abuse of notation as the number $n$ is implicit, but it is usually clear from the context.
%
%

\end{notation}

Using Notation~\ref{NotationSubmersion}, a \emph{super bundle gerbe} over $X$ is given in terms of a surjective submersion $Y \to X$ and consists of a super line bundle $\mathfrak{L}$ over $Y^{[2]}$, together with a \emph{bundle gerbe product}
\begin{equation*}
  \lambda : \mathfrak{L}_{23} \otimes \mathfrak{L}_{12} \longrightarrow \mathfrak{L}_{13}
\end{equation*}
over $Y^{[3]}$, a grading preserving isomorphism of super line bundles, which over $Y^{[4]}$ is required to satisfy the coherence condition 
\begin{equation}
\label{CoherenceGerbeMultiplication}
 \begin{tikzcd}[column sep=1.7cm]
    \mathfrak{L}_{34} \otimes \mathfrak{L}_{23} \otimes \mathfrak{L}_{12} \ar[r, "1\otimes \lambda_{123}"] \ar[d, "\lambda_{234} \otimes 1"'] 
    	& \mathfrak{L}_{34} \otimes \mathfrak{L}_{31} \ar[d, "\lambda_{134}"] \\
  \mathfrak{L}_{24} \otimes \mathfrak{L}_{12} \ar[r, "\lambda_{124}"] & \mathfrak{L}_{14}.
 \end{tikzcd}
\end{equation}
It is convenient to organize these data into a diagram as follows.
\begin{equation*}
  \mathcal{G} =   \left[
\begin{tikzcd}
  & \L \ar[d, dotted, -] & \substack{\text{bundle gerbe} \\ \text{product}~\lambda} \ar[d, -, dotted] & \substack{\text{associativity} \\ \text{of } \lambda} \ar[d, -, dotted] 
  \\
  Y \ar[d]& Y^{[2]} \ar[l, shift left=1mm] \ar[l, shift right=1mm]&  Y^{[3]} \ar[l, shift left=2mm] \ar[l, shift right=2mm] \ar[l] & Y^{[4]} \ar[l, shift left=1mm] \ar[l, shift right=1mm] \ar[l, shift left=3mm] \ar[l, shift right=3mm]
  \\
  X
\end{tikzcd}
\right]
\end{equation*}
An \emph{morphism} of bundle gerbes $\mathfrak{h} : \tilde{\mathcal{G}} \to \mathcal{G}$ is given in terms of a common refinement of covers $\tilde{Y} \leftarrow Z \to Y$, and consists of a super vector bundle $\mathfrak{H}$ over $Z$, together with a grading preserving isomorphism of super vector bundles
\begin{equation}
\label{CoherenceMorphismsBundleGerbes}
  \eta : \mathfrak{H}_2 \otimes \tilde{\L} \longrightarrow \L \otimes \mathfrak{H}_1
\end{equation}
over $Z^{[2]}$ (where we denote the pullbacks of $\tilde{\L}$ and $\L$ to $Z^{[2]}$ with the same letters again), which is compatible with the bundle gerbe products of $\tilde{\mathcal{G}}$ and $\mathcal{G}$ over $Z^{[3]}$.
$\mathfrak{h}$ is an isomorphism if $\mathfrak{H}$ is a line bundle.
In particular, a \emph{trivialization} of a bundle gerbe $\mathcal{G}$ is an isomorphism $\mathcal{G} \to \mathcal{I}$, where $\mathcal{I}$ denotes the trivial super gerbe.

Super bundle gerbes over $X$ form a bicategory $\sGerb(X)$, as first observed in \cite{StevensonBundleGerbes}.
A detailed description of the 1- and 2-morphisms is given in \cite{WaldorfDiploma, WaldorfMoreMorphisms}, which is adapted to the case of super bundle gerbes in a straightforward fashion (see also \S2.3 of \cite{kristel20212vector}). 
 The bicategory $\sGerb(X)$ is symmetric monoidal, where, if $\tilde{Y}$ and $Y$ are the covers for $\tilde{\mathcal{G}}$, respectively $\mathcal{G}$, then the tensor product gerbe $\tilde{\mathcal{G}} \otimes \mathcal{G}$ is given in terms of the fiber product cover $\tilde{Y} \times_X Y$.

\medskip

There is another notion of morphism between bundle gerbes, that of a \emph{refinement}.
A refinement $\tilde{\mathcal{G}} \to \mathcal{G}$ between super bundle gerbes consists of a map of covers $\tilde{Y} \to Y$ and a super line bundle isomorphism $\tilde{\L} \to \L$ covering the induced map $\tilde{Y}^{[2]} \to Y^{[2]}$ that intertwines the bundle gerbe products.
Super bundle gerbes over $X$ together with refinements form a symmetric monoidal (1-)category $\sGerb_{\mathrm{ref}}(X)$.
There is a canonical symmetric monoidal functor
\begin{equation}
\label{InclusionRefinements}
 \iota : \sGerb_{\mathrm{ref}}(X) \longrightarrow \sGerb(X)
\end{equation}
that sends refinements to isomorphisms.
This functor (respectively its generalization to 2-vector bundles) is explicitly described, e.g., in \cite[\S3.5]{kristel20212vector}.
It is essentially surjective on objects and any isomorphism in $\sGerb(X)$ is isomorphic to a span of images of refinements.

\medskip

Bundle gerbes often appear in the context of \emph{geometric lifting problems}.
Let $G$ be a (possibly infinite-dimensional) Lie group and let
\begin{equation}
\label{LieGroupExtension}
  \begin{tikzcd}
  A \ar[r] & \widetilde{G} \ar[r] & G
  \end{tikzcd}
\end{equation}
be a central extension of $G$ by an abelian Lie group $A$ (in this paper, we will always have $A = \Z_2$ or $\U(1)$).
Given a principal $G$-bundle $P$ over $X$, one can ask whether it is possible to lift the structure group of $P$ from $G$ to $\widetilde{G}$; in other words, whether there exists a principal $\widetilde{G}$-bundle $\widetilde{P}$ together with a map $\widetilde{P} \to P$ that is equivariant along the right map in \eqref{LieGroupExtension}.
The geometric obstruction to the existence of $\widetilde{P}$ is the corresponding \emph{lifting gerbe} $\Lift_P$.
This is a principal $A$-bundle gerbe, described in terms of the cover $P$.
Its principal $A$-bundle over $P^{[2]}$ is the pullback of the principal $A$-bundle $\widetilde{G}$ over $G$ along the difference map
\begin{equation}
\label{DifferenceMap}
\delta : P^{[2]} \longrightarrow G, \qquad (p_1, p_2) \longmapsto p_2^{-1} p_1
\end{equation}
of the principal bundle; here $p_2^{-1} p_1$ denotes the unique element $g \in G$ with $p_2 \cdot g = p_1$.
Finally, the bundle gerbe product of $\Lift_P$ is just group multiplication in $\widetilde{G}$.
The whole structure can be visualized as follows.
\begin{equation}
\label{EqLiftingGerbe}
  \Lift_P =   \left[
\begin{tikzcd}
   \widetilde{G} \ar[d]
  \\
   G
  & \delta^*\widetilde{G} \ar[d, dashed] \ar[ul, dashed] 
  & \substack{\text{group} \\ \text{multiplication} \\ \text{of } \widetilde{G}} \ar[d, -, dotted] 
  & \substack{\text{associativity} \\ \text{of group} \\ \text{multiplication}} \ar[d, -, dotted] 
  \\
  P \ar[d]& P^{[2]}  \ar[ul, "\delta"] \ar[l, shift left=1mm] \ar[l, shift right=1mm]&  P^{[3]} \ar[l, shift left=2mm] \ar[l, shift right=2mm] \ar[l] & P^{[4]} \ar[l, shift left=1mm] \ar[l, shift right=1mm] \ar[l, shift left=3mm] \ar[l, shift right=3mm]
  \\
  X
\end{tikzcd}
\right]
\end{equation}

\begin{theorem}
{\normalfont \cite{MurrayBundleGerbes}} ~
\label{ThmLiftingGerbe}
 The principal $G$-bundle $P$ admits a lift of the structure group from $G$ to $\widetilde{G}$ if and only if the lifting gerbe $\Lift_P$ admits a trivialization.
\end{theorem}

This statement can be refined to saying that there exists an equivalence of categories between the category $\textsc{Triv}(\Lift_P)$ of trivializations of $\mathcal{G}_P$ and the category of structure group extensions $\widetilde{P}$, defined in the obvious way \cite[Thm.~2.1]{WaldorfLifting}.

\subsection{Classification of super bundle gerbes}

Let $X$ be a manifold.

\begin{definition}[orientation bundle]
\label{DefinitionOrientationBundle}
Given a super bundle gerbe $\mathcal{G}$ over $X$ with cover $Y$, its \emph{orientation bundle} $\orclass(\mathcal{G})$ is the principal $\Z_2$-bundle over $X$ with total space
\begin{equation*}
  \orclass(\mathcal{G}) \defeq (Y \times \Z_2) / \sim,
\end{equation*}
where $(y_1, \epsilon_1) \sim (y_2, \epsilon_2)$ if and only if $y_1$ and $y_2$ lie in the same fiber of $Y$ over $X$ and the parity of the super line $\L_{(y_1, y_2)}$ equals the parity of $\epsilon_1+\epsilon_2$.
The total space of $\orclass(\mathcal{G})$ carries the unique smooth structure that turns the projection map into a local diffeomorphism, and it carries the $\Z_2$-action given by $[y, \epsilon] \ract \epsilon^\prime = [y, \epsilon+\epsilon^\prime]$.
\end{definition}

Observe that the defining relation of the orientation bundle is an equivalence relation due to the existence of the (grading preserving) bundle gerbe product of $\mathcal{G}$.

Let $\mathfrak{h} : \tilde{\mathcal{G}} \to \mathcal{G}$ be an isomorphism of super bundle gerbes over $X$, given in given in terms of a line bundle $\mathfrak{H}$ over a common refinement $\tilde{Y} \leftarrow Z \to Y$. 
Then there is an isomorphism $\orclass(\mathfrak{h}) : \orclass(\tilde{\mathcal{G}}) \to \orclass(\mathcal{G})$, given by sending an element $[\tilde{y}, \tilde{\epsilon}]$ of $\orclass(\tilde{\mathcal{G}})$ to the element $[y, \epsilon]$ of $\orclass(\mathcal{G})$ if $y$ and $\tilde{y}$ have a common lift $z \in Z$ and the parity of $\mathfrak{H}_z$ equals $\tilde{\epsilon} + \epsilon \mod 2$.
%
%
This gives a symmetric monoidal functor
\begin{equation}
\label{OrientationFunctor}
 \orclass :  \sGerb^\times (X) \longrightarrow \Z_2\text{-}\Bdl(X)
\end{equation}
from the bigroupoid of super bundle gerbes and isomorphisms to the category of principal $\Z_2$-bundles over $X$.

\begin{lemma}
\label{LemmaTrivialOrientationBundle}
If the orientation bundle $\orclass(\mathcal{G})$ of a super bundle gerbe $\mathcal{G}$ admits a trivialization, then $\mathcal{G}$ can be refined to a purely even bundle gerbe.
\end{lemma}

Here by a \emph{purely even} bundle gerbe $\mathcal{G}$, we mean one whose defining line bundle $\mathfrak{L}$ is purely even. 
Equivalently, $\mathcal{G}$ lies in the image of the inclusion functor
\begin{equation}
\label{InclusionFunctorGerbSGerb}
  \Gerb(X) \longrightarrow \sGerb(X),
\end{equation} 
mapping to the subcategory where all bundles are purely even.
Observe that bundle gerbes in the image of the functor \eqref{InclusionFunctorGerbSGerb} have \emph{canonically} trivial orientation bundle (with the canonical section given by $\sigma(x) = [y, 0]$ for any $y \in Y$ over $x$), while the lemma says that super bundle gerbes in the \emph{essential} image of \eqref{InclusionFunctorGerbSGerb} are those with trivial\emph{izable} orientation bundle.

\begin{proof}
Let $\mathcal{G}$ be a super bundle gerbe over $X$ with cover $Y$ and let $\tau : \orclass(\mathcal{G}) \to \Z_2$ be a trivialization.
We define $\tilde{Y}$ to be the pre-image of the canonical section of $\underline{\Z}_2$ under the map
\begin{equation*}
\begin{tikzcd}
  Y \ar[r] & \orclass(\mathcal{G}) \ar[r, "\tau"] & \Z_2,
\end{tikzcd}
\end{equation*}
where the first map sends $y \mapsto [y, 0]$.
Then $\tilde{Y}$ is a subcover of $Y$ and the refinement $\tilde{\mathcal{G}}$ of $\mathcal{G}$ along the inclusion $\tilde{Y} \to Y$ is a bundle gerbe isomorphic to $\mathcal{G}$ whose line bundle is purely even. 
%
%
Hence ${\mathcal{G}}$ is in the essential image of the functor \eqref{InclusionFunctorGerbSGerb}.
\end{proof}

\begin{remark}
\label{RemarkTrivializationInducesRefinement}
In fact, the lemma shows that every trivialialization of $\orclass(\mathcal{G})$ induces a refinement $r: \tilde{\mathcal{G}} \to \mathcal{G}$ in a canonical way, where $\tilde{\mathcal{G}}$ is a purely even bundle gerbe.
It is characterized by the property that the composition of the canonical trivialization $\Z_2 \to \orclass(\tilde{\mathcal{G}})$ with the induced map $\orclass(r) : \orclass(\tilde{\mathcal{G}}) \to \orclass(\mathcal{G})$ is the inverse of $\tau$.
\end{remark}

$\Z_2$-principal bundles $P$ are classified by a single characteristic class $w_1(P) \in H^1(X, \Z_2)$, the \emph{first Stieffel-Whitney class}, hence any super bundle gerbe has an associated $\Z_2$-valued cohomology class $w_1(\orclass(\mathcal{G}))$, which by Lemma~\ref{LemmaTrivialOrientationBundle} vanishes if and only if $\mathcal{G}$ is isomorphic in $\sGerb(X)$ to a purely even bundle gerbe.

\medskip

Any ordinary bundle gerbe $\mathcal{G}$ has an associated Dixmier-Douady class $\DDclass(\mathcal{G}) \in H^3(X, \Z)$.
As there exists a (non-monoidal) functor
\begin{equation}
\label{ForgetGradingFunctor}
 \sGerb(X) \longrightarrow \Gerb(X)
\end{equation}
which forgets the grading, the same is true for super bundle gerbes.
It turns out that the two characteristic classes  $w_1(\orclass(\mathcal{G}))$ and $\DDclass(\mathcal{G})$ characterize a super bundle gerbe $\mathcal{G}$ up to isomorphism (see \cite[\S4]{kristel20212vector}).
%

\subsection{Super bundle 2-gerbes}
\label{SectionBundle2Gerbes}

The following definition is the straightforward generalization of \cite[Def.~5.1.1]{WaldorfString} to the $\Z_2$-graded case;  see also \cite[Def.~5.3]{StevensonBundle2Gerbes}.
Throughout, we use Notation~\ref{NotationSubmersion}.

\begin{definition}[Super bundle 2-gerbe]
\label{DefinitionSuper2Gerbe}
A \emph{super bundle 2-gerbe} $\mathbb{G}$ on a manifold $X$ consists of a surjective submersion $Y \to X$, a super bundle gerbe $\mathcal{G}$ on the 2-fold fiber product $Y^{[2]}$, an isomorphism of super bundle gerbes 
\begin{equation*}
\m:  \mathcal{G}_{23} \otimes  \mathcal{G}_{12} \to \mathcal{G}_{13}
\end{equation*}
 over $Y^{[3]}$ and an invertible 2-morphism 
 \begin{equation} \label{Associator}
\begin{tikzcd}[column sep=1.5cm, row sep=1.3cm]
\mathcal{G}_{34} \otimes  \mathcal{G}_{23} \otimes  \mathcal{G}_{12} 
\ar[d, "\m_{234} \otimes 1"'] 
\ar[r, "1 \otimes  \m_{123}"]
	& 
	\mathcal{G}_{34} \otimes  \mathcal{G}_{13} 
		\ar[d, "\m_{134}"]
		\\
\mathcal{G}_{24} \otimes  \mathcal{G}_{12} \ar[r, " \m_{124}"'] 
\ar[ur, Rightarrow, shorten=0.5cm,  "a"] 
	&
	 \mathcal{G}_{14}
\end{tikzcd}
\end{equation}
over $Y^{[4]}$ satisfying a cocycle condition over $Y^{[5]}$ (see \cite[Figure 1]{StevensonBundle2Gerbes}).
\end{definition}


Bundle 2-gerbes contain quite a large amount of data, which can be best visualized in relation to their \emph{cover diagram}, the relevant part of which is
\begin{equation} 
\label{CoverDiagram}
\begin{tikzcd}
& W^{[4, 2]} 
\ar[d, shift left=1mm] 
\ar[d, shift left=3mm] 
\ar[d, shift right=1mm] 
\ar[d, shift right=3mm] 
 &  &  \\
 & 
 	W^{[3, 2]} 
	\ar[d]
 	\ar[d, shift left=2mm]
 	\ar[d, shift right=2mm]
	& 
		W^{[3, 3]}
			\ar[l]
 			\ar[l, shift left=2mm]
 			\ar[l, shift right=2mm]
		\ar[d]
 		\ar[d, shift left=2mm]
 		\ar[d, shift right=2mm]
		& 
		& 
			\\
 & 
 	W^{[2, 2]}  
 	\ar[d, shift left=1mm]
 	\ar[d, shift right=1mm]
	& 
		W^{[2, 3]}
			\ar[l]
 			\ar[l, shift left=2mm]
 			\ar[l, shift right=2mm]
 		\ar[d, shift left=1mm]
 		\ar[d, shift right=1mm]
		&
			W^{[2, 4]}
 			\ar[d, shift left=1mm]
 			\ar[d, shift right=1mm] 
			\ar[l, shift left=1mm] 
		\ar[l, shift right=1mm]
			\ar[l, shift left=3mm] 
		\ar[l, shift right=3mm]
		& 
			\\
&
	W^{[1, 2]} 
	\ar[d]
	& 
		W^{[1, 3]}
			\ar[l]
 			\ar[l, shift left=2mm]
 			\ar[l, shift right=2mm]
		\ar[d]
		&
			W^{[1, 4]}
			\ar[d]
			\ar[l, shift left=1mm] 
		\ar[l, shift right=1mm]
			\ar[l, shift left=3mm] 
		\ar[l, shift right=3mm]
		& 
			W^{[1, 5]}
		\ar[l, shift left=2mm] 
		\ar[l, shift right=2mm]
			\ar[l, shift left=4mm] 
		\ar[l, shift right=4mm]
		\ar[l]
		\ar[d]
			\\
Y \ar[d]
&
	Y^{[2]}
		\ar[l, shift left=1mm] 
		\ar[l, shift right=1mm]
	&
		Y^{[3]}
			\ar[l]
 			\ar[l, shift left=2mm]
 			\ar[l, shift right=2mm]
		&
			Y^{[4]}
			\ar[l, shift left=1mm] 
		\ar[l, shift right=1mm]
			\ar[l, shift left=3mm] 
		\ar[l, shift right=3mm]
		& 
		Y^{[5]}
			\ar[l, shift left=2mm] 
		\ar[l, shift right=2mm]
			\ar[l, shift left=4mm] 
		\ar[l, shift right=4mm]
		\ar[l]
			\\
			X
\end{tikzcd}
\end{equation}
Here $W^{[1, 2]}$, $W^{[1, 3]}$ and $W^{[1, 4]}$ are the covers of the bundle gerbe $\mathcal{G}$, the morphism $\mathfrak{m}$ and the 2-morphism $a$, respectively, and $W^{[m, n]}$ is the $m$-fold fiber product of $W^{[1, n]}$ over $Y^{[n]}$.
The vertical arrows are the usual projection maps for the iterated fiber product.
Moreover, for $m \in \N$ and $3 \leq n \geq 5$, there is a surjective submersion
\begin{equation}
\label{MapsWmn}
W^{[m, n]} \longrightarrow \underbrace{\cdots \times_{Y^{[n]}} W^{[m, n-1]}_{i_1 \dots i_{n-1}} \times_{Y^{[n]}} \cdots }_{\text{fiber product over all tuples}~1 \leq i_1 < \dots < i_{n-1} \leq n},
\end{equation}
where $W^{[m, n-1]}_{i_1 \dots i_{n-1}}$ denotes the pullback of $W^{[m, n-1]}$ along the projection $Y^{[n]} \to Y^{[n-1]}$ onto the indicated factors (see Notation~\ref{NotationSubmersion}).
This provides the horizontal arrows.
However, the diagram is only partially bisimplicial as the first two columns are missing (and we only need the depicted part), and its rows are only \emph{semi}-simplicial, as there are no degeneracy maps in general.

The data and conditions of a bundle 2-gerbe $\mathbb{G}$ can now be organized with respect to this diagram as follows.
\begin{equation} 
\label{Bundle2GerbeBigDiagram}
\begin{aligned}
%
\mathbb{G} =  \left[
\begin{tikzcd}[column sep=0.7cm]
& \substack{\text{coherence} \\ \text{condition} \\ \text{for}~\lambda} 
\ar[d, dotted, -] 
&  &  \\
 & 
 	\lambda 
	\ar[d, dotted, -]
	& 
		\substack{\text{compatibility} \\ \text{of}~\lambda~\text{and}~\mu}
			\ar[l, dotted, -]
		\ar[d, dotted, -]
		& 
		& 
			\\
 & 
 	\mathfrak{L}
 	\ar[d, dotted, -]
	& 
		\mu
			\ar[l,  dotted, -]
 		\ar[d,  dotted, -]
		&
			\substack{\text{compatibility} \\ \text{of}~\mu~\text{and}~\alpha}
 			\ar[d,  dotted, -] 
			\ar[l, dotted, -]
		& 
			\\
&
	{\color{gray}\bullet} 
	\ar[d,  dotted, -]
	& 
		\mathfrak{M}
			\ar[l,  dotted, -]
		\ar[d,  dotted, -]
		&
			\alpha
			\ar[d,  dotted, -]
			\ar[l,  dotted, -]
		& 
		\substack{\text{coherence} \\ \text{condition} \\ \text{for}~\alpha}
			\ar[l,  dotted, -] 
		\ar[d,  dotted, -]
			\\
Y \ar[d]
&
	Y^{[2]}
		\ar[l, shift left=1mm] 
		\ar[l, shift right=1mm]
	&
		Y^{[3]}
			\ar[l]
 			\ar[l, shift left=2mm]
 			\ar[l, shift right=2mm]
		&
			Y^{[4]}
			\ar[l, shift left=1mm] 
		\ar[l, shift right=1mm]
			\ar[l, shift left=3mm] 
		\ar[l, shift right=3mm]
		& 
		Y^{[5]}
			\ar[l, shift left=2mm] 
		\ar[l, shift right=2mm]
			\ar[l, shift left=4mm] 
		\ar[l, shift right=4mm]
		\ar[l] 
			\\
			X
\end{tikzcd}
\right]
\end{aligned}
\end{equation}
In other words, the data consist of two super line bundles, $\mathfrak{L}$ and $\mathfrak{M}$, and three grading preserving line bundle isomorphisms, $\lambda$, $\mu$ and $\alpha$, subject to four coherence conditions.

\begin{notation}
\label{NotationBisimplicialPullbacks}
Refining Notation~\ref{NotationSubmersion}, we denote pullbacks along the vertical maps as upper indices and pullbacks along the horizontal maps by lower indices.
\end{notation}

Using this notation, the mapping properties of the line bundle isomorphisms from \eqref{Bundle2GerbeBigDiagram} are
\begin{equation*}
\begin{aligned}
  &\lambda : & \mathfrak{L}^{23} \otimes \mathfrak{L}^{12} &\longrightarrow \mathfrak{L}^{13} \\
   &\mu : & \mathfrak{M}^2 \otimes \mathfrak{L}_{23} \otimes \mathfrak{L}_{12} &\longrightarrow \mathfrak{L}_{13} \otimes \mathfrak{M}^1\\
   &\alpha : & \mathfrak{M}_{124} \otimes \mathfrak{M}_{234} &\longrightarrow \mathfrak{M}_{134} \otimes \mathfrak{M}_{123}.
\end{aligned}
\end{equation*}
The coherence condition for $\lambda$ over $W^{[4, 2]}$ is the usual condition \eqref{CoherenceGerbeMultiplication} for a bundle gerbe product.
The compatibility condition for $\lambda$ and $\mu$ over $W^{[3,3]}$ is
\begin{equation}
\label{CoherenceNu}
\begin{tikzcd}
\mathfrak{M}^{3} \otimes \mathfrak{L}_{23}^{23} \otimes \mathfrak{L}_{12}^{23} \otimes \mathfrak{L}_{23}^{12} \otimes \mathfrak{L}_{12}^{12} \ar[r, equal] \ar[d, "\mu^{23} \otimes 1\otimes 1"']&
\mathfrak{M}^{3} \otimes \mathfrak{L}_{23}^{23} \otimes \mathfrak{L}_{23}^{12}  \otimes \mathfrak{L}_{12}^{23} \otimes \mathfrak{L}_{12}^{12} \ar[d, "1\otimes \lambda_{23} \otimes \lambda_{12}"]\\
 \mathfrak{L}_{13}^{23} \otimes \mathfrak{M}^{2}  \otimes \mathfrak{L}_{23}^{12} \otimes \mathfrak{L}_{12}^{12} \ar[d, "1 \otimes \mu^{12}"']&
 \mathfrak{M}^{3} \otimes \mathfrak{L}_{23}^{13} \otimes \mathfrak{L}_{12}^{13} \ar[d, "1 \otimes \mu^{13}"]\\
  \mathfrak{L}_{13}^{23}  \otimes \mathfrak{L}_{13}^{12} \otimes \mathfrak{M}^1 \ar[r, "\lambda_{13}"]
  & \mathfrak{L}_{13}^{13} \otimes \mathfrak{M}^1.
\end{tikzcd} 
\end{equation}
The compatibility condition for $\alpha$ and $\mu$ over $W^{[2, 4]}$ is
\begin{equation}
\label{CoherenceAlphaTilde}
\begin{tikzcd}[column sep=2cm]
\mathfrak{M}_{124}^2 \otimes \mathfrak{M}_{234}^2 \otimes \mathfrak{L}_{34} \otimes \mathfrak{L}_{23} \otimes \mathfrak{L}_{12} 
	\ar[r, "\alpha^2 \otimes 1 \otimes 1 \otimes 1"] 
	\ar[d, "1 \otimes \mu_{234} \otimes 1"']
& 
\mathfrak{M}_{134}^2 \otimes\mathfrak{M}_{123}^2 \otimes \mathfrak{L}_{34} \otimes \mathfrak{L}_{23} \otimes \mathfrak{L}_{12}
	\ar[d, equal]
\\
\mathfrak{M}_{124}^2 \otimes \mathfrak{L}_{24} \otimes \mathfrak{M}^1_{234} \otimes  \mathfrak{L}_{12}
	\ar[d, equal]
&
\mathfrak{M}_{134}^2 \otimes \mathfrak{L}_{34} \otimes\mathfrak{M}_{123}^2 \otimes \mathfrak{L}_{23} \otimes \mathfrak{L}_{12}
	\ar[d, "1 \otimes 1 \otimes \mu_{123} "]
\\
\mathfrak{M}_{124}^2 \otimes \mathfrak{L}_{24} \otimes  \mathfrak{L}_{12} \otimes \mathfrak{M}^1_{234}
	\ar[d, "1 \otimes \mu_{124}"']
&
\mathfrak{M}_{134}^2 \otimes \mathfrak{L}_{34} \otimes \mathfrak{L}_{13}  \otimes\mathfrak{M}_{123}^1 
	\ar[d, "\mu_{134} \otimes 1"]
\\
 \mathfrak{L}_{14} \otimes  \mathfrak{M}_{124}^1 \otimes \mathfrak{M}^1_{234}
 	\ar[r, "1 \otimes \alpha^1"]
	&
	\mathfrak{L}_{14} \otimes \mathfrak{M}_{134}^1 \otimes\mathfrak{M}_{123}^1.
\end{tikzcd}
\end{equation}
Finally, the cocycle condition for $\alpha$ over $W^{[1, 5]}$ is the commutativity of
\begin{equation}
\label{CoherenceAlpha} 
\begin{tikzcd}[column sep=-1cm]
& 
\mathfrak{M}_{125} \otimes \mathfrak{M}_{235} \otimes \mathfrak{M}_{345}
	\ar[dl, "1 \otimes \alpha_{2345}"'] \ar[dr, "\alpha_{1235} \otimes 1"]
	&
	\\
\mathfrak{M}_{125} \otimes \mathfrak{M}_{245} \otimes \mathfrak{M}_{234}
	\ar[d, "\alpha_{1245} \otimes 1"']
&
&
\mathfrak{M}_{135} \otimes \mathfrak{M}_{123} \otimes \mathfrak{M}_{345}
	\ar[d, equal]
	\\
\mathfrak{M}_{145} \otimes \mathfrak{M}_{124} \otimes \mathfrak{M}_{234}
	\ar[dr, "1 \otimes \alpha_{1234}"']
	&
	&
\mathfrak{M}_{135} \otimes \mathfrak{M}_{345} \otimes \mathfrak{M}_{123}
	\ar[dl, "\alpha_{1345} \otimes 1"]
	\\
	&
\mathfrak{M}_{145} \otimes \mathfrak{M}_{134} \otimes \mathfrak{M}_{123}.
\end{tikzcd} 
\end{equation}
Throughout, morphisms labeled as equality are the braiding isomorphism in the category of super line bundles.

\begin{remark}
Just as (super) bundle gerbes over a manifold $X$ are the objects of a bicategory, (super) bundle 2-gerbes over $X$ ought to be the objects of a tricategory (in fact, varying $X$, one should obtain a stack of tricategories on the site of manifolds).
However, no description of this tricategory exists in the literature.
This lack of theory exists for a good reason:
While working with bicategories is unproblematic (although sometimes slightly cumbersome), the complexity increases immensely when passing to tricategories (see, e.g., \cite[\S2.2]{GPSTricategories} \cite[\S3.1]{GurskiTricategories}).
Therefore, it may be more sensible to use the language of $\infty$-categories from the outset when working with (higher) bundle gerbes; 
however, in practice, it is often convenient to work with the explicit bicategorical description of bundle gerbes and the transition from this to an $\infty$-categorical setup has not been worked out in the literature.
In this paper, we will not need any tricategory of (super) bundle 2-gerbes, and will just give ad hoc definitions whenever necessary.
\end{remark}

\begin{remark}
This way of representing bundle gerbes is similar to the definition of a \emph{bigerbe}, as introduced in \cite{Bigerbes}; see also \cite{Rigid2Gerbes}.
The key difference is that for bigerbes, the $W^{[m, n]}$ are required to be fiber products also in the $m$ index while the maps \eqref{MapsWmn} need not be surjective submersions.
On the other hand, for bigerbes, the authors require the line bundle $\mathfrak{M}$ to be trivial.
\end{remark}

\subsection{Isomorphisms of super bundle 2-gerbes}
\label{SectionIso2Gerbes}

We now discuss \emph{isomorphisms} of super bundle gerbes. 
While it is possible to extract from the literature when two bundle 2-gerbes are isomorphic \cite[Def.~4.6]{CJMSW}, the definition of \emph{isomorphisms} of bundle 2-gerbes does not exist in the literature (to our knowledge).
However, one can find the definition of a \emph{trivialization} of a bundle 2-gerbe, which is an isomorphism $\mathbb{G} \to \mathbb{I}$, where $\mathbb{I}$ is the trivial bundle 2-gerbe; this is \cite[Definition~11.1]{StevensonBundle2Gerbes}, see also \cite[Definition~2.2.1]{WaldorfStringConnections}.
This definition generalizes to the case of general 2-morphisms (and lifts to the setting of \emph{super} bundle 2-gerbes) in a straightforward fashion, with the following result.

\begin{definition}[Isomorphism of super bundle 2-gerbes]
\label{DefinitionIso2gerbe}
Let $\tilde{\mathbb{G}}$ and $\mathbb{G}$ be super bundle 2-gerbes, described in terms of open covers $\tilde{Y}$ and $Y$ of $X$.
An \emph{isomorphism of super bundle 2-gerbes} $\mathcal{h} : \tilde{\mathbb{G}} \to \mathbb{G}$ consists of an common refinement
\begin{equation*}
\begin{tikzcd}
\tilde{Y} & \ar[l] Z \ar[r] & Y
\end{tikzcd}
\end{equation*}
 of covers of $X$, a super bundle gerbe $\mathcal{H}$ over $Z$, an isomorphism
\begin{equation*}
  \mathfrak{n} : \mathcal{H}_2 \otimes \tilde{\mathcal{G}} \longrightarrow \mathcal{G} \otimes \mathcal{H}_1
\end{equation*}
of super bundle gerbes over $Z^{[2]}$ and a 2-morphism
\begin{equation}
\label{TwoMorphismForIsomorphism}
\begin{tikzcd}[column sep=1.5cm]
  \mathcal{H}_3 \otimes \tilde{\mathcal{G}}_{23}  \otimes \tilde{\mathcal{G}}_{12} 
   \ar[r, "1 \otimes\tilde{\m}"] \ar[d, "\mathfrak{n}_{23} \otimes 1"'] 
  & \mathcal{H}_{3} \otimes \tilde{\mathcal{G}}_{13} 
    \ar[dd, " \mathfrak{n}_{13}"] 
        \ar[ddl, Rightarrow, "b"', shorten=1.4cm]
  \\
  \mathcal{G}_{23} \otimes \mathcal{H}_2 \otimes \tilde{\mathcal{G}}_{12} 
  \ar[d, "1 \otimes \mathfrak{n}_{12}"']
  &
  \\
  \mathcal{G}_{23} \otimes \mathcal{G}_{12} \otimes \mathcal{H}_1 
  \ar[r, "\m \otimes 1"]
  & \mathcal{G}_{13} \otimes \mathcal{H}_1
\end{tikzcd}
\end{equation}
over $Z^{[3]}$ satisfying a cocycle condition over $Z^{[4]}$.
Here and throughout, we denote the pullbacks of $\tilde{\mathcal{G}}$, $\mathcal{G}$, $\tilde{\mathfrak{m}}$, $\mathfrak{m}$, $\tilde{a}$ and $a$ to fiber products of $Z$ with the same letters again.
Two super bundle 2-gerbes are said to be \emph{(stably) isomorphic} if and only if there exists an isomorphism between them.
\end{definition}

\begin{remark}
To obtain morphisms of super bundle 2-gerbes that are not isomorphisms, one can use the language of 2-vector bundles introduced in \cite{kristel20212vector}.
Then the bundle gerbe $\mathcal{H}$ in Definition~\ref{DefinitionIso2gerbe} is replaced by a general super 2-vector bundle, and $\mathfrak{n}$ is an isomorphism of such.
This is a natural categorification of the corresponding concept of non-invertible morphisms of bundle gerbes introduced in \cite{WaldorfMoreMorphisms}.
\end{remark}

As super bundle 2-gerbes themselves, the data corresponding to morphisms between them can be organized with respect to a partial semi-bisimplicial diagram, which in this case takes the form
\begin{equation} 
\label{CoverDiagramMorphism}
\begin{tikzcd}
V^{[4, 1]}
\ar[d, shift left=1mm] 
\ar[d, shift left=3mm] 
\ar[d, shift right=1mm] 
\ar[d, shift right=3mm] 
    &  &  &  \\
V^{[3, 1]}
\ar[d]
\ar[d, shift left=2mm]
\ar[d, shift right=2mm]
		& 
 	V^{[3, 2]}  
	\ar[d]
 	\ar[d, shift left=2mm]
 	\ar[d, shift right=2mm]
			\ar[l, shift left=1mm]
			\ar[l, shift right=1mm]
	& & 
		& 
			\\
V^{[2, 1]} 
\ar[d, shift left=1mm]
\ar[d, shift right=1mm]
	& 
 	V^{[2, 2]}  
 	\ar[d, shift left=1mm]
 	\ar[d, shift right=1mm]
	\ar[l, shift left=1mm]
	\ar[l, shift right=1mm]
	& 
		V^{[2, 3]}
			\ar[l]
 			\ar[l, shift left=2mm]
 			\ar[l, shift right=2mm]
 		\ar[d, shift left=1mm]
 		\ar[d, shift right=1mm]
		& & 
			\\
V^{[1, 1]}
\ar[d]
	&
	V^{[1, 2]} 
	\ar[l, shift left=1mm]
	\ar[l, shift right=1mm]
	\ar[d]
	& 
		V^{[1, 3]}
			\ar[l]
 			\ar[l, shift left=2mm]
 			\ar[l, shift right=2mm]
		\ar[d]
		&
			V^{[1, 4]}
			\ar[d]
			\ar[l, shift left=1mm] 
		\ar[l, shift right=1mm]
			\ar[l, shift left=3mm] 
		\ar[l, shift right=3mm]
		& 
			\\
Z \ar[d]
&
	Z^{[2]}
		\ar[l, shift left=1mm] 
		\ar[l, shift right=1mm]
	&
		Z^{[3]}
			\ar[l]
 			\ar[l, shift left=2mm]
 			\ar[l, shift right=2mm]
		&
			Z^{[4]}
			\ar[l, shift left=1mm] 
		\ar[l, shift right=1mm]
			\ar[l, shift left=3mm] 
		\ar[l, shift right=3mm]
			\\
			X
\end{tikzcd}
\end{equation}
Here $V^{[1, 1]}$, $V^{[1, 2]}$ and $V^{[1, 3]}$ are the covers for the super bundle gerbe $\mathcal{H}$, the isomorphism $\mathfrak{n}$ and the 2-morphism $b$, respectively.
As before, for $m \in \N$ and $1 \leq n \leq 4$, these come with surjective submersions
\begin{equation*}
  V^{[m, n]} \longrightarrow \underbrace{ \cdots \times_{Z^{[n]}} V^{[m, n-1]}_{i_1 \dots i_{n-1}} \times_{Z^{[n]}} \cdots}_{\text{fiber product over all tuples}~1 \leq i_1 < \dots < i_{n-1} \leq n}
\end{equation*}
as well as
\begin{equation*}
 V^{[m, n]} \longrightarrow \tilde{W}^{[m, n]} \times_{Z^{[n]}} W^{[m, n]},
\end{equation*}
where we assume that $\mathbb{G}$ is given in terms of the cover diagram \eqref{CoverDiagram} and in the corresponding diagram of $\tilde{\mathbb{G}}$, all covers are decorated with a tilde.
Hence the whole diagram \eqref{CoverDiagramMorphism}, past the first column, maps to the fiber product of the cover diagrams of $\tilde{\mathbb{G}}$ and $\mathbb{G}$ over $X$.
We use the same letters again for the pullbacks of the defining data for $\tilde{\mathbb{G}}$ and $\mathbb{G}$ along these maps.
For example, $\mathfrak{L}$ (respectively $\tilde{\mathfrak{L}}$) denotes the pullback to $V^{[2, 2]}$ of the line bundle of $\mathbb{G}$ over $W^{[2, 2]}$ (respectively of the line bundle of $\tilde{\mathbb{G}}$ over $\tilde{W}^{[2, 2]}$).

The data and conditions for the isomorphism $\mathcal{h}: \tilde{\mathbb{G}} \to \mathbb{G}$ can now be arranged into a diagram corresponding to \eqref{CoverDiagramMorphism} as follows.
\begin{equation} 
\label{TwoMorphismDiagram}
\begin{aligned}
\mathcal{h}  = 
\left[
\begin{tikzcd}[column sep=0.7cm]
\substack{\text{coherence} \\ \text{condition}\\ \text{for}~\eta}
\ar[d,  dotted, -] 
    &  &   \\
 \eta
\ar[d,  dotted, -]
		& 
 	\substack{\text{compatibility} \\ \text{of}~\eta~\text{and}~\nu\\ \text{with} ~\tilde{\lambda}~\text{and}~\lambda}  
	\ar[d,  dotted, -]
	\ar[l,  dotted, -]
	& & 
			\\
\mathfrak{H}
\ar[d,  dotted, -]
	& 
 	\nu 
 	\ar[d,  dotted, -]
	\ar[l,  dotted, -]
	& 
 	\substack{\text{compatibility} \\ \text{of}~\nu~\text{and}~\beta \\ \text{with}~\tilde{\mu}~\text{and}~\mu}  
			\ar[l,  dotted, -]
 		\ar[d,  dotted, -]
		& 
			\\
{\color{gray}\bullet}
\ar[d,  dotted, -]
	&
	\mathfrak{N} 
	\ar[l,  dotted, -]
	\ar[d,  dotted, -]
	& 
		\beta
			\ar[l,  dotted, -]
		\ar[d,  dotted, -]
		&
			\substack{\text{compatibility} \\ \text{of}~\beta~\text{with} \\ \tilde{\alpha}~\text{and}~\alpha}
			\ar[d, dotted, -]
			\ar[l,  dotted, -]
			\\
Z \ar[d]
&
	Z^{[2]}
		\ar[l, shift left=1mm] 
		\ar[l, shift right=1mm]
	&
		Z^{[3]}
			\ar[l]
 			\ar[l, shift left=2mm]
 			\ar[l, shift right=2mm]
		&
			Z^{[4]}
			\ar[l, shift left=1mm] 
		\ar[l, shift right=1mm]
			\ar[l, shift left=3mm] 
		\ar[l, shift right=3mm]
			\\
			X
\end{tikzcd}
\right]
\end{aligned}
\end{equation}
Hence there are again two super line bundles $\mathfrak{H}$ and $\mathfrak{N}$, three isomorphisms of super line bundles $\eta$, $\nu$ and $\beta$, and four coherence conditions. 
As before, denoting horizontal pullbacks with lower indices and vertical pullbacks with upper indices, the line bundle isomorphisms are
\begin{equation*}
\begin{aligned}
  &\eta :  & \mathfrak{H}^{23} \otimes \mathfrak{H}^{12} &\longrightarrow \mathfrak{H}^{13} \\
  &\nu : & \mathfrak{N}^2 \otimes \mathfrak{H}_2 \otimes \tilde{\mathfrak{L}}  &\longrightarrow  \mathfrak{L} \otimes \mathfrak{H}_1 \otimes \mathfrak{N}^1  \\
  &\beta : & \mathfrak{N}_{13} \otimes \tilde{\mathfrak{M}} &\longrightarrow \mathfrak{M} \otimes \mathfrak{N}_{23} \otimes \mathfrak{N}_{12}.
\end{aligned}
\end{equation*}
Here we denoted the pullback of $\tilde{\mathfrak{L}}$ and $\mathfrak{L}$ to $S^{[2]}$ by the same letters again.
The coherence condition for $\eta$ is just the usual associativity condition for a bundle gerbe product.
The compatibility condition of $\eta$ and $\nu$ with $\lambda$ and $\tilde{\lambda}$ over $V^{[3, 2]}$ is
\begin{equation}
\label{CoherenceEta}
\begin{tikzcd}[column sep=3.4cm]
  \mathfrak{N}^3 \otimes \mathfrak{H}_2^{23}  \otimes  \mathfrak{H}_2^{12} \otimes \tilde{\mathfrak{L}}^{23} \otimes \tilde{\mathfrak{L}}^{12} 
 \ar[d, equal]
 \ar[r, "1 \otimes \eta_2 \otimes \tilde{\lambda}"]
 &
    \mathfrak{N}^3 \otimes\mathfrak{H}_2^{13} \otimes \tilde{\mathfrak{L}}^{13}
     \ar[dddd, "\nu^{13}"]
 \\
\mathfrak{N}^3 \otimes \mathfrak{H}_2^{23}  \otimes \tilde{\mathfrak{L}}^{23}  \otimes  \mathfrak{H}_2^{12} \otimes \tilde{\mathfrak{L}}^{12}
 \ar[d, "\nu^{23}  \otimes 1 \otimes  1"']
 &
 \\
\tilde{\mathfrak{L}}^{23} \otimes \mathfrak{H}_1^{23} \otimes  \mathfrak{N}^2 \otimes  \mathfrak{H}_2^{12} \otimes \tilde{\mathfrak{L}}^{12}
	 \ar[d, " 1 \otimes 1\otimes \nu^{12}"'] 
	 \\
{\mathfrak{L}}^{23} \otimes \mathfrak{H}_1^{23}  \otimes \mathfrak{L}^{12} \otimes \mathfrak{H}_1^{12} \otimes   \mathfrak{N}^1
	  \ar[d, equal]
	  \\
{\mathfrak{L}}^{23} \otimes \mathfrak{L}^{12} \otimes \mathfrak{H}_1^{23}  \otimes \mathfrak{H}_1^{12} \otimes   \mathfrak{N}^1
	\ar[r, "\lambda  \otimes \eta_1 \otimes 1"] 
	&
	\mathfrak{L}^{13}  \otimes   \mathfrak{H}_1^{13} \otimes \mathfrak{N}^1.
\end{tikzcd}
\end{equation}
The compatibility condition for $\beta$, $\nu$, $\tilde{\mu}$ and $\mu$ over $V^{[2, 3]}$ is 
\begin{equation}
\label{CoherenceBetaTilde}
\begin{tikzcd}[column sep=1.3cm]
  \mathfrak{N}_{13}^2 \otimes \tilde{\mathfrak{M}}^2 \otimes \mathfrak{H}_3 \otimes \tilde{\mathfrak{L}}_{23} \otimes \tilde{\mathfrak{L}}_{12}
  \ar[r, "\beta^2\otimes 1 \otimes 1"]
  \ar[d, equal]
  	&  
	\mathfrak{M}^2 \otimes \mathfrak{N}_{23}^2 \otimes \mathfrak{N}_{12}^2 \otimes \mathfrak{H}_3 \otimes \tilde{\mathfrak{L}}_{23} \otimes \tilde{\mathfrak{L}}_{12}
		\ar[d, equal]
  \\
     \mathfrak{N}_{13}^2\otimes \mathfrak{H}_3  \otimes \tilde{\mathfrak{M}}^2 \otimes \tilde{\mathfrak{L}}_{23} \otimes \tilde{\mathfrak{L}}_{12}
  \ar[ddd, "1 \otimes 1 \otimes \tilde{\mu}"']
   &
     	\mathfrak{M}^2 \otimes \mathfrak{N}_{23}^2 \otimes \mathfrak{H}_3 \otimes \tilde{\mathfrak{L}}_{23} \otimes \mathfrak{N}_{12}^2  \otimes \tilde{\mathfrak{L}}_{12}
		\ar[d, "1\otimes \nu_{23} \otimes 1 \otimes 1"]
	\\
  	&
	\mathfrak{M}^2  \otimes \mathfrak{L}_{23}\otimes  \mathfrak{H}_2 \otimes \mathfrak{N}_{23}^1   \otimes \mathfrak{N}_{12}^2 \otimes \tilde{\mathfrak{L}}_{12}
	\ar[d, equal]
	\\
	&
	\mathfrak{M}^2\otimes \mathfrak{L}_{23}  \otimes \mathfrak{N}_{23}^1 \otimes \mathfrak{N}_{12}^2  \otimes  \mathfrak{H}_2 \otimes \tilde{\mathfrak{L}}_{12} 
	\ar[d, "1 \otimes 1 \otimes 1\otimes \nu_{12}"]
	\\
  \mathfrak{N}_{13}^2 \otimes \mathfrak{H}_3 \otimes \tilde{\mathfrak{L}}_{13} \otimes \tilde{\mathfrak{M}}^1
  	\ar[ddd, "\nu_{13}\otimes 1"']
	&
	\mathfrak{M}^2\otimes \mathfrak{L}_{23}  \otimes \mathfrak{N}_{23}^1 \otimes \mathfrak{L}_{12}  \otimes  \mathfrak{H}_1  \otimes \mathfrak{N}_{12}^1
	\ar[d, equal]
	\\
	&
	\mathfrak{M}^2\otimes \mathfrak{L}_{23} \otimes \mathfrak{L}_{12}  \otimes \mathfrak{N}_{23}^1  \otimes  \mathfrak{H}_1  \otimes \mathfrak{N}_{12}^1
	\ar[d, "\mu \otimes 1\otimes 1\otimes 1"]
	\\
	&
	 \mathfrak{L}_{13}  \otimes \mathfrak{M}^1  \otimes \mathfrak{N}_{23}^1 \otimes  \mathfrak{H}_1  \otimes \mathfrak{N}_{12}^1
	 \ar[d, equal]
	 \\
\mathfrak{L}_{13} \otimes \mathfrak{H}_1 \otimes \mathfrak{N}_{13}^1  \otimes \tilde{\mathfrak{M}}^1
	\ar[r, "1 \otimes 1 \otimes \beta^1"]
	 &
	 \mathfrak{L}_{13}   \otimes  \mathfrak{H}_1 \otimes \mathfrak{M}^1  \otimes \mathfrak{N}_{23}^1 \otimes \mathfrak{N}_{12}^1
\end{tikzcd}
\end{equation}
and, finally, the compatibility condition for $\beta$, $\tilde{\alpha}$ and $\alpha$ over $V^{[1, 4]}$ is
\begin{equation}
\label{CocycleConditionb}
\begin{tikzcd}[column sep=1.8cm]
\mathfrak{N}_{14} \otimes \tilde{\mathfrak{M}}_{124} \otimes \tilde{\mathfrak{M}}_{234} 
	\ar[r, "1 \otimes \tilde{\alpha}"]
	\ar[d, "\beta_{124} \otimes 1"']
	&
	\mathfrak{N}_{14} \otimes \tilde{\mathfrak{M}}_{134} \otimes \tilde{\mathfrak{M}}_{123}
	\ar[d, "\beta_{134} \otimes 1"]
	\\
	\mathfrak{M}_{124} \otimes \mathfrak{N}_{24}  \otimes \mathfrak{N}_{12} \otimes \tilde{\mathfrak{M}}_{234}
	\ar[d, equal]
	&
	\mathfrak{M}_{134} \otimes \mathfrak{N}_{34} \otimes \mathfrak{N}_{13} \otimes \tilde{\mathfrak{M}}_{123}
	\ar[d, "1 \otimes 1 \otimes \beta_{123}"]
	\\
\mathfrak{M}_{124} \otimes \mathfrak{N}_{24}  \otimes \tilde{\mathfrak{M}}_{234} \otimes \mathfrak{N}_{12}
	\ar[d, "1 \otimes \beta_{234} \otimes 1"]
	&
	\mathfrak{M}_{134} \otimes \mathfrak{N}_{34} \otimes \mathfrak{M}_{123} \otimes \mathfrak{N}_{23} \otimes \mathfrak{N}_{12}
	\ar[d, equal]
	\\
\mathfrak{M}_{124} \otimes \mathfrak{M}_{234} \otimes \mathfrak{N}_{34} \otimes \mathfrak{N}_{23}  \otimes \mathfrak{N}_{12}
	\ar[r, "\alpha \otimes 1 \otimes 1 \otimes 1"]
	&
	\mathfrak{M}_{134} \otimes \mathfrak{M}_{123}  \otimes \mathfrak{N}_{34} \otimes \mathfrak{N}_{23}  \otimes \mathfrak{N}_{12}
\end{tikzcd}
\end{equation}
%

\subsection{Classification of super bundle 2-gerbes}
\label{SectionClassificationBundle2Gerbes}

Just as super bundle gerbes have an associated orientation bundle, bundle 2-gerbes over a manifold $X$ have an associated orientation  gerbe.

\begin{definition}[Orientation gerbe]
\label{DefinitionOrientationGerbe}
The \emph{orientation gerbe} of a super bundle 2-gerbe $\mathbb{G}$ over $X$ is the $\Z_2$-bundle gerbe $\orclass(\mathbb{G})$ over $X$ obtained by applying the functor \eqref{OrientationFunctor} to all data of $\mathbb{G}$.
\end{definition}

Explicitly if we use the same notation as in Definition~\ref{DefinitionSuper2Gerbe} for the defining data of $\mathbb{G}$, $\orclass(\mathbb{G})$ has cover $Y$ and its principal $\Z_2$-bundle over $Y^{[2]}$ is $\orclass(\mathcal{G})$.
Using that the orientation bundle functor is monoidal, we obtain a bundle gerbe product
\begin{equation*}
\begin{tikzcd}
  \orclass(\mathcal{G}_{23}) \otimes \orclass(\mathcal{G}_{12}) \cong \orclass(\mathcal{G}_{23} \otimes \mathcal{G}_{12}) \ar[r, "\orclass(\mathfrak{m})"] & \orclass(\mathcal{G}_{13}),
\end{tikzcd}
\end{equation*}
which is associative over $Y^{[3]}$ by the existence of the associator 2-morphism $a$.
An isomorphism  of bundle 2-gerbes $\mathcal{h} : \tilde{\mathbb{G}} \to \mathbb{G}$ as in Definition~\ref{DefinitionIso2gerbe} induces an isomorphism $\orclass(\mathcal{h}) : \orclass(\tilde{\mathbb{G}}) \to \orclass(\mathbb{G})$ of $\Z_2$-bundle gerbes in a straight forward way, again by applying the functor \eqref{OrientationFunctor} to all data. 
In particular, isomorphic super bundle 2-gerbes have isomorphic orientation gerbes.

A principal $\Z_2$-bundle gerbe $\mathcal{P}$ over a manifold $X$ is classified by a single characteristic class 
\begin{equation*}
  w_2(\mathcal{P}) \in H^2(X, \Z_2),
\end{equation*}
which one could call the \emph{second Stiefel-Whitney class}.
Hence any super bundle 2-gerbe $\mathbb{G}$ has an associated characteristic class $w_2(\orclass(\mathbb{G}))$, which depends only on its isomorphism class, by taking $w_2$ of its orientation gerbe.

\begin{lemma}
\label{LemmaPurelyEven2Gerbe}
If $\mathbb{G}$ is a super bundle 2-gerbe such that its orientation gerbe $\orclass(\mathbb{G})$ is trivializable, then it is isomorphic to a purely even super bundle 2-gerbe.
\end{lemma}

\begin{proof}
Let us denote the defining data for $\mathbb{G}$ as in Definition~\ref{DefinitionSuper2Gerbe}.
Let $\mathfrak{t}$ be a trivialization of $\orclass(\mathbb{G})$, consisting of a principal $\Z_2$-bundle $P$ over $Y$ and an isomorphism 
\begin{equation}
\label{IsoTauP}
\tau : P_2 \otimes \orclass(\mathcal{G}) \to P_1
\end{equation}
 over $Y^{[2]}$.
We define a bundle 2-gerbe isomorphic to $\tilde{\mathbb{G}}$ given in terms of the larger cover $P \to Y \to X$.
Denote by $\rho$ the projection $P \to Y$ and use the same notation for the map induced map $P^{[k]} \to Y^{[k]}$ on fiber products over $X$.
As the pullback $\rho^*P$ is canonically trivial (and so are the pullbacks $\rho^*P_i$ to $P^{[k]}$),
%
%
we obtain from \eqref{IsoTauP} a trivialization $\rho^*\tau : \orclass(\rho^*\mathcal{G}) \to \Z_2$ of $\rho^*\mathcal{G}$ over $P^{[2]}$.
As in the proof of Lemma~\ref{LemmaTrivialOrientationBundle} (see also Remark~\ref{RemarkTrivializationInducesRefinement}), this provides a purely even bundle gerbe $\tilde{\mathcal{G}}$ over $P^{[2]}$ together with a refinement $r : \tilde{\mathcal{G}} \to \rho^*\mathcal{G}$.

Now, there is a unique morphism $\tilde{\mathfrak{m}}$ of purely even super bundle gerbes over $P^{[3]}$ such that the diagram
\begin{equation}
\label{2cellInProof}
\begin{tikzcd}
\rho^*\mathcal{G}_{23} \otimes \rho^* \mathcal{G}_{12} 
\ar[d, equal] 
& &
  \tilde{\mathcal{G}}_{23} \otimes \tilde{\mathcal{G}}_{12} 
  	\ar[d, dashed, "\tilde{\mathfrak{m}}"]  
	\ar[ll, "r_{23} \otimes r_{12}"']
	\\
	\rho^*(\mathcal{G}_{23} \otimes \mathcal{G}_{12})
	\ar[r, "\rho^*\mathfrak{m}"]
	 &
	 \rho^*\mathcal{G}_{13}
	 &
	 \ar[l, "r_{13}"']
 \tilde{\mathcal{G}}_{13}.
\end{tikzcd}
\end{equation}
strictly commutes after turning all refinements into morphisms of super bundle gerbes using \eqref{InclusionRefinements}. 
This gives the bundle 2-gerbe multiplication of $\tilde{\mathbb{G}}$.
The associator $\tilde{\alpha}$ of $\tilde{\mathbb{G}}$ is just the pullback $\rho^*\alpha$ composed with various pullbacks of the 2-cell \eqref{2cellInProof}.
The cocycle condition for $\tilde{\alpha}$ then follows directly from that of $\alpha$.
Hence we constructed a super bundle 2-gerbe $\tilde{\mathbb{G}}$.
An isomorphism $\tilde{\mathbb{G}} \cong \mathbb{G}$ is easily constructed using the refinement $r$.

We claim that $\tilde{\mathbb{G}}$ is purely even.
By construction, its defining bundle gerbe $\tilde{\mathcal{G}}$ is purely even.
Observe that it follows from the coherence condition \eqref{CoherenceMorphismsBundleGerbes} for morphisms of super bundle gerbes that a morphism between purely even bundle gerbes is either purely even or purely odd.
To see that $\tilde{\mathfrak{m}}$ is purely even, it suffices to show that the induced map $\orclass(\tilde{\mathfrak{m}})$ on orientation bundles is the identity (using the fact that the orientation bundle of a purely even bundle gerbe is \emph{canonically} trivial).
To this end, we use that the morphism $\tau$ satisfies the compatibility condition
\begin{equation*}
\begin{tikzcd}
P_3 \otimes \orclass(\mathcal{G}_{23}) \otimes \orclass(\mathcal{G}_{12}) \ar[r, "\tau_{23}"] \ar[d, equal] &
P_2 \otimes \orclass(\mathcal{G}_{12})  \ar[r, "\tau_{12}"] &
  P_1 \ar[d, equal] \\
  P_3 \otimes \orclass(\mathcal{G}_{23} \otimes \mathcal{G}_{12}) \ar[r, "1 \otimes \orclass(\mathfrak{m})"'] & P_3 \otimes \orclass(\mathcal{G}_{13}) \ar[r, "\tau_{13}"] &
 P_1.
\end{tikzcd}
\end{equation*}
Pulling back this diagram to $P^{[3]}$ along $\rho$, we get the commutative diagram
\begin{equation*}
\begin{tikzcd}
\orclass(\rho^*\mathcal{G}_{23}) \otimes \orclass(\rho^*\mathcal{G}_{12}) 
\ar[rr, "\rho^*\tau_{23} \otimes \rho^*\tau_{12}"] \ar[d, equal] 
& &
  {\Z}_2 \ar[d, equal] \\
  \orclass(\rho^*\mathcal{G}_{23} \otimes \rho^*\mathcal{G}_{12}) \ar[r, "1 \otimes \orclass(\rho^*\mathfrak{m})"'] 
  &
  \orclass(\rho^*\mathcal{G}_{13}) \ar[r, "\rho^*\tau_{13}"] &
 {\Z}_2
\end{tikzcd}
\end{equation*}
using the canonical trivializations of $\rho^*P_i$.
Now, by Remark~\ref{RemarkTrivializationInducesRefinement}, the refinement $r$ has the property that the induced map $\orclass(r) : \underline{\Z}_2 = \orclass(\tilde{\mathcal{G}}) \to \orclass(\rho^*\mathcal{G})$ equals the inverse of $\rho^*\tau$ under the canonical trivializations of $\rho^*P_2$ and $\rho^*P_1$.
Hence the last diagram equals the diagram obtained from applying the orientation bundle functor to the diagram \eqref{2cellInProof}, except that the latter has $\orclass(\tilde{\mathfrak{m}})$ as rightmost vertical arrow.
Comparing the diagrams, this shows that $\orclass(\tilde{\mathfrak{m}})$ is the identity, as desired.
\end{proof}
%
%

\begin{remark}
It is well known that ungraded bundle 2-gerbes $\mathbb{G}$ over a manifold $X$ have a \emph{characteristic class}
\begin{equation*}
  \CCclass(\mathbb{G}) \in H^4(X, \Z).
\end{equation*}
Any cohomology class in $H^4(X, \Z)$ is represented by a bundle 2-gerbe and two isomorphic bundle 2-gerbes share the same characteristic class; see, e.g., \cite[\S7]{StevensonBundle2Gerbes}.
Such a class is not defined for \emph{super} bundle 2-gerbes, as it is \emph{not} true that any super bundle 2-gerbe gives an ungraded bundle 2-gerbe by forgetting the grading -- in contrast to the case of super bundle (1-)gerbes.
This is due to the fact that the braiding isomorphism of the category of \emph{super} bundle gerbes is used at various places in the definition of super bundle 2-gerbes.
\end{remark}

\bibliography{literature.bib}

\end{document}